\documentclass{sl}     %% This the CLASS for Studia Logica,

\usepackage{slsec}     %% PACKAGES for Studia Logica
\usepackage{stud_log} 
\usepackage{slfoot}
\usepackage{slthm}     %% Theorem-like definitions as in amsthm,
\usepackage{amsmath}
\usepackage{amssymb}

\usepackage{enumerate}
\def\kn{\kern.1em}

\renewcommand{\theequation}{\Roman{equation}}
\theoremstyle{definition}

\HeadingsInfo{S\'andor Jenei}{The Hahn embedding theorem for a class of residuated semigroups}

\usepackage{color}
\definecolor{midgrey}{RGB}{150,173,180}

\usepackage{graphicx}

\usepackage{amsfonts}
\usepackage{mathtools}
\DeclarePairedDelimiter{\ceil}{\lceil}{\rceil}

\numberwithin{equation}{section}

\newcommand{\komp}      {{^\prime}}

\newcommand{\kompM}[1]{^{{\prime^{\mkern-4mu^{_{#1}}}}}}
\newcommand{\nega}      [1] {{#1}\komp}
\newcommand{\negaM}      [2] {{#2}\kompM{#1}}

\newcommand{\te}{{\mathbin{*\mkern-9mu \circ}}}

\newcommand{\ted}{\mathbin{\diamond}}

\newcommand{\tebeta}{\mathbin{\te_\beta}}
\newcommand{\tegamma}{\mathbin{\te_\gamma}}

\newcommand{\ite}[1]{\mathbin{\rightarrow_{#1}}}

\newcommand{\g}                 [2] {{#1} \mathbin{\te} {#2}}
\newcommand{\gstar}               [2] {{#1} \mathbin{\star} {#2}}

\newcommand{\gd}                [2] {{#1} \ted {#2}}

\newcommand{\gbeta}             [2] {\mathbin{{#1} \tebeta {#2}}}
\newcommand{\ggamma}             [2] {\mathbin{{#1} \tegamma {#2}}}

\newcommand{\res}               [3] {{#2}\mathbin{\ite{#1}}{#3}}

\newcommand{\lex}                                        {\overset{\rightarrow}{\times}}

\newcommand{\lexADD}                            {\mathbin{\overset{\rightarrow}{\oplus}}}
\newcommand{\Twoheadrightarrow}               {\rightarrow \mkern-23.25mu \rightharpoonup}
\newcommand{\threeheadrightarrow}               {\rightarrow \mkern-12.5mu \rightarrow}
\newcommand{\plexII}               {\overset{\Twoheadrightarrow}{\times}}
\newcommand{\plexI}               {\overset{\threeheadrightarrow}{\times}}

\newcommand{\PLPI}             [3] {{#1}_{{#2}}\plexI{{#3}}}
\newcommand{\PLPII}            [2] {{#1}\plexII{{#2}}}
\newcommand{\PLPIs}            [4] {{#1}_{{\left({#2}\plexI{#3}\right)}_{#4}}}
\newcommand{\PLPIIs}            [4] {{#1}_{{\left({#2}\plexII{#3}\right)}_{#4}}}

\newcommand{\PLPIII}            [4] {{#1}_{{#2}_{#3}}\plexI{{#4}}}
\newcommand{\PLPIV}            [3] {{#1}_{#2}\plexII{{#3}}}

\newtheorem{theorem}{Theorem}[section]
\newtheorem{lemma}[theorem]{Lemma}

\theoremstyle{definition}
\newtheorem{definition}[theorem]{Definition}
\newtheorem{example}[theorem]{Example}

\newtheorem{proposition}[theorem]{Proposition}
\newtheorem{corollary}[theorem]{Corollary}

\theoremstyle{remark}
\newtheorem{remark}[theorem]{Remark}

\numberwithin{equation}{section}

\begin{document}

\setcounter{page}{1}     %%  Initial page number. Do not change. 

\AuthorTitle{S\'andor Jenei}{The Hahn embedding theorem for a class of residuated semigroups}
\footnote{Institute of Mathematics and Informatics, University of P\'ecs, H-7624 P\'ecs, Ifj\'us\'ag u. 6., Hungary, jenei@ttk.pte.hu}

%\subtitle{Do you have a subtitle?\\ If so, write it here}

%\titlerunning{Group Representation and Hahn-type Embedding for Odd Involutive FL$_e$-chains with Finitely Many Idempotent Elements}        % if too long for running head

%\titlerunning{The Hahn embedding theorem for a class of commutative, involutive residuated chains}        % if too long for running head

\PresentedReceived{This paper is \cite{Jenei_Hahn} modified by the corrections of \cite{Jenei_Hahn_err}.}{}

%\renewcommand{\thefootnote}{\alph{footnote}}

%\authorrunning{Short form of author list} % if too long for running head

%\subjclass[2010]{Primary 97H50, 20M30; Secondary 06F05, 06F20, 03B47.}

%\date{Submitted: \today}
%\date{Received: March 2, 2018 / Accepted: date}
% The correct dates will be entered by the editor

%\maketitle

\begin{abstract}
Hahn's embedding theorem asserts that linearly ordered abelian groups embed in some lexicographic product of real groups. Hahn's theorem is generalized to a class of residuated semigroups in this paper, namely, to odd involutive commutative residuated chains which possess only finitely many idempotent elements. To this end, the partial sublex product construction is introduced to construct new odd involutive commutative residuated lattices from a pair of odd involutive commutative residuated lattices, and a representation theorem for odd involutive commutative residuated chains which possess only finitely many idempotent elements, by means of linearly ordered abelian groups and the partial sublex product construction is presented.
\end{abstract}


\section{Introduction}

Hahn's celebrated embedding theorem states that every linearly ordered abelian group $G$ can be embedded as an ordered subgroup into the Hahn product 
$\displaystyle{\overset{\rightarrow_H}{\times_\Omega}}\ \mathbb R$, where 
$\mathbb R$ is the additive group of real numbers (with its standard order), 
$\Omega$ is the set of archimedean equivalence classes of $G$ (ordered naturally by the dominance relation),
${\overset{\rightarrow_H}{\times_\Omega}}\ \mathbb R$ is the set of all functions from $\Omega$ to $\mathbb R$ (alternatively the set of all vectors with real elements and with coordinates taken from $\Omega$) 
which vanish outside a well-ordered set, 
%and ${\overset{\rightarrow_H}{\times_\Omega}}\ \mathbb R$ is 
endowed with a lexicographical order
\cite{hahn}. 
%Hahn's celebrated embedding theorem states that every linearly ordered abelian group $G$ can be embedded as an ordered subgroup into the additive group $\mathbb R^\Omega$ endowed with a lexicographical order, where $\mathbb R$ is the additive group of real numbers (with its standard order), $\Omega$ is the set of archimedean equivalence classes of $G$ (ordered naturally by the dominance relation), and $\mathbb R^\Omega$ is the set of all functions from $\Omega$ to $\mathbb R$ which vanish outside a well-ordered set \cite{hahn}.
Briefly, every linearly ordered abelian group can be represented as a group of real-valued functions on a totally ordered set.
By weakening the linearly ordered hypothesis, Conrad, Harvey, and Holland generalized Hahn's theorem for lattice-ordered abelian groups in \cite{CHH} by showing that any abelian $\ell$-group can be represented as a group of real-valued functions on a partially ordered set. 
%As a corollary of our representation theorem, w
By weakening the hypothesis on the existence of the inverse element but keeping the linear order, in this paper Hahn's theorem will be generalized to a class of residuated semigroups, namely to linearly ordered odd involutive commutative residuated lattices which possess only finitely many idempotent elements.
The generalization will be a by-product of a representation theorem for odd involutive commutative residuated chains which possess only finitely many idempotent elements, by means of the partial sub-lexicographic product construction and using only linearly ordered abelian groups.
Thus, the price for not having inverses in our semigroup framework is that our embedding is made into partial sub-lexicographic products, introduced here, rather than lexicographic ones. The building blocks, which are linearly ordered abelian groups, remain the same. 

Residuation is a basic concept in mathematics \cite{residuation theory} with strong connections to Galois maps \cite{Galoisbook} and closure operators. Residuated semigroups have been introduced in the 1930s by Ward and Dilworth \cite{WD39} to investigate ideal theory of commutative rings with unit.
Recently the investigation of residuated lattices (that is, residuated monoids on lattices) has become quite intense, initiated by the discovery of the strong connection between residuated lattices and substructural logics \cite{gjko} via an algebraic logic link \cite{BlokPigozzi}.
Substructural logics encompass among many others, Classical logic, Intuitionistic logic, \L ukasiewicz logic, Abelian logic, Relevance logics, Basic fuzzy logic, Monoidal $t$-norm logic, Full Lambek calculus, Linear logic, many-valued logics, mathematical fuzzy logics, along with their non-com\-mu\-ta\-tive versions. 
%These logics had different motivations, different methodology, and have mainly been investigated by isolated groups of researchers.
The theory of substructural logics has put all these logics, along with many others, under the same motivational and methodological umbrella, and residuated lattices themselves have been the key component in this remarkable  unification.
Examples of residuated lattices include Boolean algebras, Heyting algebras \cite{heyting}, MV-algebras \cite{cigmun}, BL-algebras \cite{hajekbook}, and lattice-ordered groups, to name a few, a variety of other algebraic structures can be rendered as residuated lattices. Applications of substructural logics and residuated lattices span across proof theory, algebra, and computer science.

%The algebraic investigation of residuated lattices has mainly focused on the integral case.
As for the structural description of classes of residuated lattices, 
non-integral residuated structures, and consequently substructural logics without the weakening rule, are far less understood at present than their integral counterparts. 
Therefore, some authors try to establish category equivalences between integral and non-integral residuated structures to gain a better understanding of the non-integral case, and in particular, to gain a better understanding of substructural logics without weakening \cite{GalRaf,Gal2015}.
Despite the extensive literature devoted to classes of residuated lattices, there are still very few results that effectively describe their structure, and many of these effective descriptions postulate, besides integrality, the naturally ordered condition\footnote{Its dual notion is often called divisibility.}, too \cite{AglMon,cigmun,GenMV,hajekbook,JF,JKpsBCK,lawson,Mos57}.
Linearly ordered odd involutive commutative residuated lattices (another terminology is odd involutive FL$_e$-chains) are non-integral and not naturally ordered, unless they are trivial.
Therefore, from a general point of view, our study is a contribution to the structural description of residuated structures which are neither naturally ordered nor integral.

%\section*{Preliminaries}

\begin{definition}\label{FLe}\rm
An {\em FL$_e$-algebra} %\footnote{Other terminologies for FL$_e$-algebras are: pointed commutative residuated lattices or pointed commutative residuated lattice-ordered monoids.} 
is a structure $( X, \wedge,\vee, \te, \ite{\te}, t, f )$ such that 
$(X, \wedge,\vee )$ is a lattice\footnote{Sometimes the lattice operators are replaced by their induced ordering $\leq$ in the signature, in particular, if an FL$_e$-{\em chain} is considered, that is, if the ordering is linear.}, $( X,\leq, \te,t)$ is a commutative, %\footnote{The subscript $e$ refers to the exchange rule of a substructural logic \cite{gjko}, which is equivalent to the commutativity of the algebras of the related algebraic semantic of that logic.}, 
residuated monoid\footnote{We use the word monoid to mean semigroup with unit element.}, and $f$ is an arbitrary constant.
Here being {\em residuated} means that there exists a binary operation $\ite{\te}$,
called the residual operation of $\te$, such that 
$$
\mbox{
$\g{x}{y}\leq z$ if and only if $\res{\te}{x}{z}\geq y$.
}
$$
This equivalence is called adjointness condition and ($\te,\ite{\te}$) is called an adjoint pair. Equivalently, for any $x,z\in X$, the set $\{v\ | \ \g{x}{v}\leq z\}$ has its greatest element, and $\res{\te}{x}{z}$, the so-called residuum of $x$ and $z$, is defined to be this greatest element: 
$$
\mbox{
$\res{\te}{x}{z}:=\max\{v\ | \ \g{x}{v}\leq z\}$.
}
$$
This is called the residuation property. 
Easy consequences of this definition are the exchange property
$$
\res{\te}{(\g{x}{y})}{z}=\res{\te}{x}{(\res{\te}{y}{z})},
$$
that $\ite{\te}$ is isotone at its second argument,
and that $\te$ distributes over arbitrary joins.
One defines $\nega{x}=\res{\te}{x}{f}$ and calls an FL$_e$-algebra {\em involutive} if $\nega{(\nega{x})}=x$ holds.
We say that the rank of an involutive FL$_e$-algebra is positive if $t>f$, negative if $t<f$, and $0$ if $t=f$. In the zero rank case we also say that the involutive FL$_e$-algebra is {\em odd}.
Denote the set of {\em positive} (resp.\,{\em negative}) elements of $X$ by $X^+=\{x\in X : x\geq t\}$ (resp.\, $X^-=\{x\in X : x\leq t\}$), and call the elements of $X^+$ different from $t$ {\em strictly positive}.
We call the FL$_e$-algebra {\em conical} if all elements of $X$ are comparable with $t$.
If the algebra is linearly ordered we speak about FL$_e$-{\em chains}.
Algebras will be denoted by bold capital letters, their underlying set by the same regular letter.
Commutative residuated lattices are exactly the $f$-free reducts of FL$_e$-algebras.
\end{definition}

\begin{definition}\rm\label{LexRendDefi}
The lexicographic product of two linearly ordered sets 
$\mathbf A=(A,\leq_1)$ and $\mathbf B=(B,\leq_2)$
is a linearly order set 
$\mathbf A\lex\mathbf B=(A\times B,\leq)$, where $A\times B$ is the Cartesian product of $A$ and $B$, and $\leq$ is defined by $\langle a_1,b_1 \rangle\leq\langle a_2,b_2 \rangle$ if and only if $a_1<_1 a_2$ or $a_1=a_2$ and  $b_1\leq_2 b_2$.
The lexicographic product $\mathbf A\lex\mathbf B$ of two FL$_e$-chains $\mathbf A$ and $\mathbf B$ is an FL$_e$-chain over the lexicographic product of their respective universes such that all operations are defined coordinatewise.

\end{definition}
\noindent
We shall view such a lexicographic product %$\mathbf A\lex\mathbf B$ of two linearly ordered sets 
as an {\em enlargement}: each element of $A$ is replaced by a whole copy of $B$.
Accordingly, in Section~\ref{SectPLPs} by a {\em partial} lexicographic product of two linearly ordered sets 
we will mean a kind of partial enlargement: only {\em some} elements of the first algebra will be replaced by a whole copy of the second algebra. %, and the rest of the first algebra will be left unchanged.

\medskip
In any involutive FL$_e$-algebra $\komp$ is an order reversing involution of the underlying set, and $\komp$ has a single fixed point if the algebra is odd.
Hence in odd involutive FL$_e$-algebras $x\mapsto \nega{x}$ is somewhat reminiscent to a reflection operation across a point in geometry, yielding a symmetry of order two with a single fixed point.
In this sense $\nega{t}=f$ means that the position of the two constants is symmetric in the lattice. 
%are mutually the images of one another under $\komp$, that is, their relative position in the lattice is symmetric.
Thus, an extreme situation is the integral case, when $t$ is the top element of $X$ and hence $f$ is its bottom element.
This case has been deeply studied in the literature.
The other extreme situation, when \lq\lq the two constants are in both the middle of the lattice\rq\rq\, (that is, when $t=f$ and so the algebra is odd) is a much less studied scenario.
We remark that $t<f$ can also hold in some involutive FL$_e$-algebras. However, there is no third extremal situation, as it cannot be the case that $f$ is the top element of $X$ and $t$ is its bottom element: For any residuated lattice with bottom element, the bottom element has to be the annihilator of the algebra, but $t$ is its unit element, a contradiction unless the algebra is trivial.

Prominent examples of odd involutive FL$_e$-algebras are odd Sugihara monoids and lattice-ordered abe\-lian groups.
The latter constitutes an algebraic semantics of Abelian Logic \cite{MeyerAbelian,cesari,AbelianNow} while the former constitutes an algebraic semantics of $\mathbf{IUML}^*$, which is a logic at the intersection of relevance logic and many-valued logic \cite{GalRaf}.
These two examples represent two extreme situations from another viewpoint: There is a single idempotent element in any lattice-ordered abelian group, whereas all elements are idempotent in any odd Sugihara monoid. The scope of our investigation lies in between these two extremes; we shall assume that the number of idempotent elements of the odd involutive FL$_e$-chain is finite.
For this class a representation theorem along with some corollaries (e.g. the generalization of Hahn's theorem) will be presented in this paper. %, see Section~\ref{mainsection}.
%: Every odd involutive FL$_e$-algebra which has only finitely many idempotent elements has a representation by means of the partial lexicographic product construction and linearly ordered abelian groups, see Section~\ref{mainsection}.

\section{Odd involutive FL$_e$-algebras vs.\ partially ordered abelian groups}
We start with three preliminary observations, the first of which is being folklore (we present its proof, too, to keep the paper self contained).
\begin{proposition}
For any involutive FL$_e$-algebra $( X, \wedge, \vee, \te, \ite{\te}, t, f )$ the following hold true.
\begin{equation}\label{eq_quasi_inverse}
\res{\te}{x}{y} = \nega{\left(\g{x}{\nega{y}}\right)},
\end{equation}
If $t\geq f$ then
\begin{equation}\label{eq_feltukrozes}
\g{x}{y}\leq\nega{\left(\g{\nega{x}}{\nega{y}}\right)} .
\end{equation}
If $t\leq f$ and the algebra is conical then $y_1>y$ implies
\begin{equation}\label{eq_feltukrozes_CS}
\nega{(\g{\nega{x}}{\nega{y}})}\leq\g{x}{y_1} .
\end{equation}
\end{proposition}
\begin{proof}
Using that $\komp$ is an involution one easily obtains 
$\nega{(\g{x}{\nega{y}})}=\res{\te}{\left(\g{x}{\nega{y}}\right)}{f}=\res{\te}{x}{\left( \res{\te}{\nega{y}}{f} \right)}=\res{\te}{x}{y}$.
To prove (\ref{eq_feltukrozes}) one proceeds as follows:
$\g{(\g{x}{y})}{(\g{\nega{x}}{\nega{y}})}=
\g{[\g{x}{(\res{\te}{x}{f})}]}{[\g{y}{(\res{\te}{y}{f})}]}
\leq \g{f}{f}
\leq \g{t}{f}=f$,
 hence $\g{x}{y} \leq \res{\te}{(\g{\nega{x}}{\nega{y}})}{f}$.
To show (\ref{eq_feltukrozes_CS}) we proceed as follows.
Since the algebra is involutive $\nega{t}=f$ holds.
Since  the algebra is conical, every element is comparable with $f$, too.
Indeed, for any $a\in X$, if $a$ were not compatible with $f$ then since $\komp$ is an order reversing involution, $\nega{a}$ were not compatible with $\nega{f}=t$, which is a contradiction.
Hence, by residuation\footnote{Throughout the paper when we say in a proof \lq by residuation\rq\ or \lq by adjointness\rq\ 
we refer to the residuation property and the adjointness property, respectively.}, $y_1>y=\nega{(\nega{y})}=\res{\te}{\nega{y}}{f}$ implies 
$\g{y_1}{\nega{y}}\not\leq f$, that is, $\g{y_1}{\nega{y}}>f\geq t$.
Therefore, 
$\nega{(\g{x}{y_1})}=\g{\nega{(\g{x}{y_1})}}{t}\leq\g{\nega{(\g{x}{y_1})}}{\g{y_1}{\nega{y}}}=\g{(\g{y_1}{(\res{\te}{y_1}{\nega{x}})})}{\nega{y}}\leq\g{\nega{x}}{\nega{y}}$ follows by using (\ref{eq_quasi_inverse}).
\end{proof}
In our investigations a crucial role will be played by the $\tau$ function.
%In its investigations, a key-role will be played by
\begin{definition}\label{def sk}{\bf ($\tau$)}
\rm
For an involutive FL$_e$-algebra $( X, \wedge, \vee, \te, \ite{\te}, t, f )$, for $x\in X$, define
$\tau(x)$ to be the greatest element of $Stab_x=\{ z\in X \ | \ \g{z}{x}=x \}$, the stabilizer set of $x$.
Since $t$ is the unit element of $\te$, $Stab_x$ is nonempty.
Since $\te$ is residuated, the greatest element of $Stab_x$ exists, and it holds true that
\begin{equation}\label{tnelnagyobb}
\g{\tau(x)}{x}=x \ \ \ \mbox{and} \ \ \  \tau(x)=\res{\te}{x}{x}\geq t. %\ \footnote{In lattice-ordered abelian groups $\tau(x)$ would always result in $t$, the single unit element of $\te$. Intuitively, we may think of $\tau(x)$ as the unit {\em of the component of} $x$; its meaning will become apparent in claim~(\ref{zetaRange_idempotent}) in Proposition~\ref{zeta_1} and in Proposition~\ref{elozetes}, as $\tau(x)$ will act as the unit of the subalgebra on $X_{\tau\geq\tau(x)}$.}
\end{equation}
\end{definition}
\begin{proposition}\label{zeta_1}%{\bf (Basic props of $\tau$)}
For an involutive FL$_e$-algebra $( X,  \wedge, \vee, \te, \ite{\te}, t, f )$ the following holds true.
\begin{enumerate}

\item\label{FENTszimmetrikus}
$\tau(x)=\tau(\nega{x})$.

\item\label{item_zetavege_idempotens} $\tau(\tau(x))=\tau(x)$.

\item\label{lemasoljaGENzeta} $\tau(\g{x}{y})\geq\tau(x)$. 
\item\label{ZetaOfIdempotent} $u\in X^+$ is idempotent if and only if $\tau(u)=u$.
\item\label{zetaRange_idempotent}
$Range(\tau)=\{ \tau(x): x\in X \}$ is equal to the set of idempotent elements in $X^+$. 
\item \label{item_boundary_zeta}
If the order is linear then for $x\in X^+$, $\tau(x)\leq x$ holds.
\end{enumerate}
\end{proposition}
\begin{proof}
\begin{enumerate}
\item
By (\ref{eq_quasi_inverse}), 
$\tau(x)=\res{\te}{x}{x}=\nega{(\g{x}{\nega{x}})}=\nega{(\g{\nega{x}}{\nega{\nega{x}}})}=\res{\te}{\nega{x}}{\nega{x}}=\tau(\nega{x})$.
%\item We shall prove that $y=\tau(x)$ implies $\tau(y)=y$. By (\ref{eq_quasi_inverse}), $\res{\te}{x}{x}=y$ Now, $y\te\nega{y}=y\te (x\te\nega{x})=(y\te x)\te \nega{x}=\left(\tau(x)\te x\right)\te \nega{x}=x\te\nega{x}=\nega{y}$ follows, which is equivalent to $\res{\te}{y}{y}=y$, that is, $\tau(y)=y$.
\item
By (\ref{eq_quasi_inverse}), $\res{\te}{x}{x}=\tau(x)$ is equivalent to $x\te\nega{x}=\nega{\tau(x)}$.
Hence, $\tau(x)\te\nega{\tau(x)}$ $=\tau(x)\te (x\te\nega{x})=(\tau(x)\te x)\te \nega{x}=x\te\nega{x}=\nega{\tau(x)}$ follows.
Hence, $\tau(\tau(x))=\res{\te}{\tau(x)}{\tau(x)}=\nega{(\g{\tau(x)}{\nega{\tau(x)}})}=\nega{\nega{\tau(x)}}=\tau(x)$ follows by (\ref{eq_quasi_inverse}).
\item 
If $\g{u}{x}=x$ then $\g{u}{(\g{x}{y})}=\g{(\g{u}{x})}{y}=\g{x}{y}$, that is, $Stab_x\subseteq Stab_{\g{x}{y}}$.
\item
If $u\geq t$ is idempotent then from $\g{u}{u}=u$, $\res{\te}{u}{u}\geq u$ follows by adjointness. But for any $z>u$, $\g{u}{z}\geq\g{t}{z}=z>u$, 
hence $\tau(u)=u$ follows. 
On the other hand, $\tau(u)=u$ implies $u\geq t$ by (\ref{tnelnagyobb}), and also the idempotency of $u$ since $\g{u}{u}=\g{u}{\tau(u)}=u$.
\item
If $u>t$ is idempotent then claim~(\ref{ZetaOfIdempotent}) shows that $u$ is in the range of $\tau$.
If $u$ is in the range of $\tau$, that is $\tau(x)=u$ for some $x\in X$ then claim~(\ref{item_zetavege_idempotens}) implies $\tau(u)=u$, hence $u$ is an idempotent in $X^+$ by claim~(\ref{ZetaOfIdempotent}).
\item
Since the order is linear, the opposite of the statement is $\tau(x)>x$, but it yields
$\g{x}{\tau(x)}\geq\g{t}{\tau(x)}=\tau(x)>x$, a contradiction to (\ref{tnelnagyobb}).
%For $x_1>x\in X^+$, $\g{x}{x_1}\geq\g{t}{x_1}=x_1>x$ holds. Hence, since the order is linear, for $x_1>x\in X^+$, $\g{x}{x_1}\not\leq x$ follows, and by adjointness it is equivalent to $\res{\te}{x}{x}\not\geq x_1$, hence $\res{\te}{x}{x}<x_1$, since the order is linear. Thus, $\res{\te}{x}{x}\leq x$ follows, as stated.
\end{enumerate}
\end{proof}
The notion of odd involutive FL$_e$-algebras has been defined with respect to the  general notion of residuated lattices by adding further postulates (such as commutativity, an extra constant $f$, involutivity, and the $t=f$ property).
The following theorem relates odd involutive FL$_e$-algebras to (in the setting of residuated lattices, very specific) lattice-ordered abelian groups, thus picturing their precise interrelation. In addition, Theorem~\ref{csoportLesz} will serve as the basic step of the induction in the proof of Theorem~\ref{Hahn_type}, too. 

\begin{theorem}\label{csoportLesz}
For an odd involutive FL$_e$-algebra $\mathbf X=(X, \wedge, \vee,\te, \ite{\te}, t, f)$ the following statements are equivalent:
\begin{enumerate}
\item\label{fosdgs}
Each element of $X$ has inverse given by $x^{-1} = \nega{x}$, and hence $(X, \wedge, \vee,\te, t)$ is a {lattice-ordered abelian group},
\item\label{fkkssjffwgj}
$\te$ is cancellative,
\item\label{njsjsh}
$\tau(x)=t$ for all $x\in X$.
\item\label{jsjhsgsjdb}
The only idempotent element in $X^+$ is $t$.
\end{enumerate}
\end{theorem}
\begin{proof}
First note that by (\ref{eq_quasi_inverse}), claim~(\ref{fosdgs}), that is, for all $x\in X$, 
$\g{x}{\nega{x}}=t$, is equivalent to claim~(\ref{njsjsh}).
Also claim~(\ref{fosdgs}) $\Rightarrow$ claim~(\ref{fkkssjffwgj}) is straightforward.
To see  claim~(\ref{fkkssjffwgj}) $\Rightarrow$ claim~(\ref{fosdgs}) we proceed as follows:
By residuation, $\g{x}{\nega{x}}\leq f$ holds, therefore by isotonicity of $\ite{\te}$ at its second argument,  $\res{\te}{x}{(\g{x}{\nega{x}})}\leq \res{\te}{x}{f}=\nega{x}$ follows.
By residuation $\g{x}{\nega{x}}\leq\g{x}{\nega{x}}$ is equivalent to $\res{\te}{x}{(\g{x}{\nega{x}})}\geq\nega{x}$, hence we infer $\res{\te}{x}{(\g{x}{\nega{x}})}=\nega{x}=\res{\te}{x}{f}=\res{\te}{x}{t}$.
By (\ref{eq_quasi_inverse}), $\g{x}{\nega{(\g{x}{\nega{x}})}}=\g{x}{\nega{t}}$ follows, and cancellation by $x$ implies $t=\g{x}{\nega{x}}$, so we are done.
Finally, claim~(\ref{zetaRange_idempotent}) in Proposition~\ref{zeta_1} ensures the equivalence of claims (\ref{njsjsh}) and (\ref{jsjhsgsjdb}).
\end{proof}
In the light of Theorem~\ref{csoportLesz}, in the sequel when we (loosely) speak about a subgroup of an odd involutive FL$_e$-algebra, we always mean a cancellative subalgebra of it.
Further, let $$X_{gr}=\{ x\in X \ | \ \nega{x} \mbox{ is the inverse of } x\}.$$ 
Evidently, there is a subalgebra $\mathbf X_\mathbf{gr}$ of $\mathbf X$ over $X_{gr}$\footnote{This will also follow from Proposition~\ref{elozetes}.},
and
$\mathbf X_\mathbf{gr}$ is the largest subgroup of $\mathbf X$.
We call $\mathbf X_\mathbf{gr}$ the group part of $\mathbf X$.

\section{Two illustrative examples}

Unfortunately, there does not exist any easily accessible example in the class of odd involutive FL$_e$-chains, apart from the cancellative subclass (linearly ordered abelian groups, each with a single idempotent element) or the idempotent subclass (odd Sugihara monoids, all elements are idempotent).
Here we start with the two simplest nontrivial examples of odd involutive FL$_e$-chains over $\mathbb R$, both of which have only a single strictly positive idempotent element.
These examples will be the inspirational source for the type I-IV partial lexicographic product constructions in Definition~\ref{FoKonstrukcio}.

Using (\ref{eq_feltukrozes}) it takes a half-line proof to show that for any odd involutive FL$_e$-algebra, the residual complement of any negative idempotent element $u$ is also idempotent:
$\g{\nega{u}}{\nega{u}}\leq\nega{(\g{u}{u})}=\nega{u}=\g{\nega{u}}{t}\leq\g{\nega{u}}{\nega{u}}$. 
The author of the present paper has put a considerable effort into proving the converse of this statement, with no avail. In fact, the converse statement does not hold, as shown by the following counterexample.
\begin{example}\label{kiveteles2}
\rm
The odd involutive FL$_e$-chain which is to be defined in this example is 
$\PLPII{\pmb{\mathbb Z}}{\pmb{\mathbb R}}$ (%$\PLPII{\pmb{\mathbb Z}}{\pmb{\mathbb R}}$ 
up to isomorphism), %, a shorthand for $\PLPIV{\pmb{\mathbb Z}}{\pmb{\mathbb Z}}{\pmb{\mathbb R}}$), 
see item~B in Definition~\ref{FoKonstrukcio} for its formal definition.
Here we present an informal account on the motivation and considerations which had led to the discovery of this example.

Let $I=\ ]\text{-}1,0[$. Consider any order-isomorphic copy ${\mathbf I}=(I,\star,\text{-}\frac{1}{2})$ of the linearly ordered abelian group $\pmb{\mathbb{R}}=(\mathbb R, +, 0)$\footnote{Let, for example, $\star$ be given by $\gstar{x}{y}=f^{-1}(f(x)+f(y))$ where $f: I\to \mathbb R$, $f(x)=\tan((x+0.5)\cdot\pi)$).}.
\begin{figure}[ht]
\begin{center}
\includegraphics[width=0.18\textwidth]{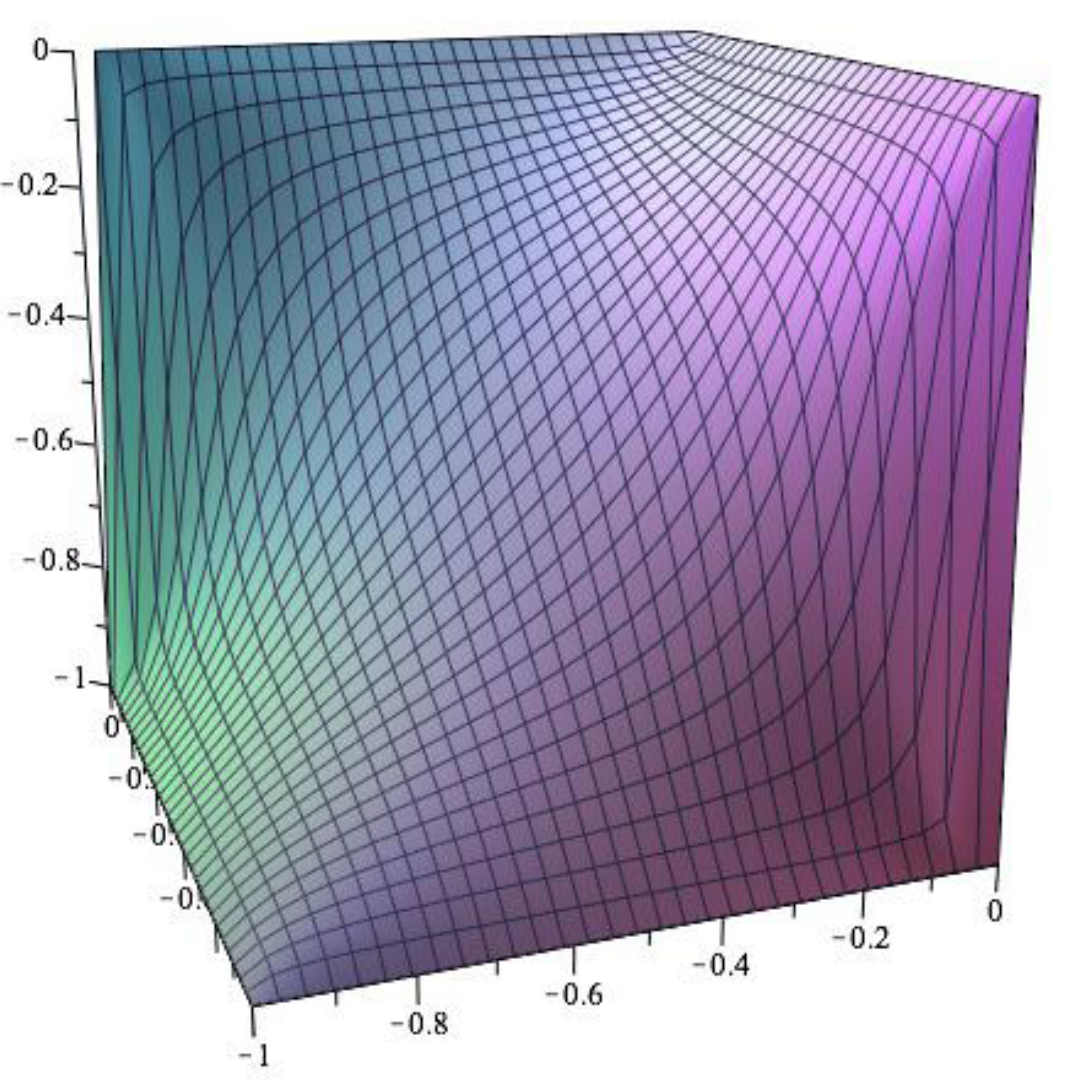} 
%\hskip2cm 
%\includegraphics[height=5.3in]{mehsejt} \hskip2cm 
\caption{Visualization: The graph of the operation $\star$}
%\label{}
\end{center}
\end{figure}
Our plan is to put a series of copies of $I$ in a row (after each other, one for each integer, that is, formally we consider $\mathbb Z\times I$) to get $\mathbb R$, and to define, based on $\star$, an operation $\te$ over $\mathbb R$ so that the resulting structure becomes an odd involutive FL$_e$-chain. The first problem is caused by the fact that whichever involutive structure we start with (in place of $\mathbf I$), the above-described repetition of its universe will never yield $\mathbb R$. % as the underlying set for $\te$:
In our example $I=\ ]\text{-}1,0[$ does not have top and bottom elements, hence by the above-described repetition of it we obtain a set which is order-isomorphic to $\mathbb R\setminus\mathbb Z$, and not to $\mathbb R$. Or if the original structure does have a top element (and hence by involutivity a bottom element, too) then the universe resulted by the above-described repetition will have gaps. 
To overcome this we extend $\star$ to a new universe $I^\top:=I\cup\{0\}=\ ]\text{-}1,0]$ by letting $0$ be an annihilator of $\star$.
In other words, we add a new top element to the structure $\mathbf I$. 
The resulting structure, denote it by $\mathbf I^\top$, will no longer be an odd involutive FL$_e$-chain, but it is still an FL$_e$-chain, albeit obviously not even involutive due to the broken symmetry of its underlying set.
Luckily we can immediately eliminate this broken symmetry by putting 
$\aleph_0$ copies of $I^\top=]\text{-}1,0]$ in a row, thus obtaining $\mathbb R$.
Now we can define an operation $\te$ over $\mathbb R$ by letting for $a,b\in\mathbb R$,
$$\g{a}{b}=\ceil{a}+\ceil{b}+\gstar{(a-\ceil{a})}{(b-\ceil{b})}.\footnote{$\ceil{a}$ stands for the ceiling of $a$}$$
%\begin{figure}[ht]
%\begin{center}
%%\includegraphics[height=2.3in]{tepont} \hskip2cm 
%%\includegraphics[height=3.3in]{mehsejt} \hskip2cm 
%\includegraphics[width=0.42\textwidth]{suru_masolat} \hskip2cm 
%\caption{Visualization: The graph of the operation $\te$ restricted to $[\text{-}3,2]\times[\text{-}3,2]$.
%An order-isomorphic copy of $\te  :  \mathbb R\times\mathbb R\to\mathbb R$ compressed into $]0,1[\times]0,1[\to]0,1[$ is in Figure~\ref{fig:1} (right), named $\PLPII{\mathbb{Z}}{\mathbb{R}}$ according to the general notation, which is introduced in item~(2) of Definition~\ref{FoKonstrukcio}}
%\label{}
%\end{center}
%\end{figure}
(see the graph of an order-isomorphic copy of $\te$ in Figure~\ref{fig:1}, right, for the use of such 3D plots the interested reader is referred to \cite{Refl Inv Jenei Geom SemForum}).
The slightly cumbersome task of verifying that $( {\mathbb R}, \leq, \te, \ite{\te}, \text{-}\frac{1}{2}, \text{-}\frac{1}{2} )$ is an odd involutive FL$_e$-chain, is left to the reader.
It amounts
to
verify the associativity of $\te$\footnote{A computation reveals that $\g{(\g{a}{b})}{c}=(\ceil{a}+\ceil{b})+\ceil{c}+\gstar{  \gstar{(a-\ceil{a})}{(b-\ceil{b})}   }{(c-\ceil{c})}$, thus the associativity of $\star$ ensures the associativity of $\te$.},
to
compute the residual operation of $\te$,
to
confirm that the residual complement operation is $\nega{x}=\text{-}x-1$,
and to confirm that $\text{-}\frac{1}{2}$ is the unit element and also the fixed point of $\komp$.
Note that the only strictly positive idempotent element of the constructed algebra is $0$, while $\text{-}\frac{1}{2}$ is its unit element. 
Note, however, that the residual complement of $0$ is not idempotent: 
$\nega{0}=\text{-}1$ and $\g{\text{-}1}{\text{-}1}=\text{-}2$.
%The graph of $\te$ (which is of type $\mathbb R\times\mathbb R\to\mathbb R$) compressed into $]0,1[\times]0,1[\to]0,1[$ by an order-isomorphism is in Fig.~\ref{fig:1} (right).

In the quest for a deeper understanding of this example, observe that apart from $\star$, the addition operation of $\pmb{\mathbb{R}}$ plays a role in the definition of $\te$.
Moreover, by using a particular order-isomorphism, we can define $\te$ using $\star$ and the addition operation of $\pmb{\mathbb{Z}}$ only, as follows.
As said, putting a series of copies of 
$I^\top$ in a row can, formally, be obtained by considering the lexicographic product $\mathbb Z\lex I^\top$. 
Then, $h(a)=(\ceil{a},a-\ceil{a})$ is an order-isomorphism from $\mathbb R$ to $\mathbb Z\lex I^\top$.
It is left to the reader to verify that $\te$ can equivalently be defined in a {\em coordinatewise manner} %using only the addition $+$ of $\mathbb Z$ and $\star$
by letting $\g{a}{b}=h^{-1}[(a_1+b_1,\gstar{a_2}{b_2})]$ where $h(a)=(a_1,a_2)$ and $h(b)=(b_1,b_2)$.
By using the natural order-reversing involutions ($\negaM{\ast}{a}=\text{-}a$ on $\mathbb Z$ and $\negaM{\star}{x}=\text{-}x-1$ on $I$) the reader can also verify that the residual operation of $\te$ is given by
$$
\res{\te}{(a_1,a_2)}{(b_1,b_2)}=\nega{\left(\g{(a_1,a_2)}{\nega{(b_1,b_2)}}\right)} ,
$$
where \footnote{Compare (\ref{FuraNegaEXTpelda}) with the last two rows of (\ref{FuraNegaEXT});
for the last row of (\ref{FuraNegaEXT}) see Definition~\ref{nyilak} for the notation $_\downarrow$.}
\begin{equation}\label{FuraNegaEXTpelda}
\nega{(x,y)}=\left\{
\begin{array}{ll}
(\negaM{\ast}{x}-1,0) 	& \mbox{if $y=0$}\\
(\negaM{\ast}{x},\negaM{\star}{y}) 	& \mbox{if $y\in I$}\\
\end{array}
\right. .
\end{equation}
With this notation, the only strictly positive idempotent element of the constructed algebra is $(0,0)$, while $(0,\text{-}\frac{1}{2})$ is its unit element. 
A moral of this example is that the latter coordinatewise formalism for $\te$ enlightens the example by providing a deeper insight into and a clear description of why $\te$ makes up for an odd involutive FL$_e$-chain (why is it associative, residuated, etc.)
Compare the example above with the formal definition of $\PLPII{\pmb{\mathbb{Z}}}{\pmb{\mathbb{R}}}$ in item~B in Definition~\ref{FoKonstrukcio}.
Further inspection reveals that in these kind of examples one does not necessarily have to work with groups (like $\pmb{\mathbb{Z}}$ and $\pmb{\mathbb{R}}$ in our example), but also odd involutive FL$_e$-algebras suffice in both coordinates.
But in this case only the {\em group part} of the first algebra (and no more) can be enlarged.
This makes our lexicographic product construction partial in nature, as if the first algebra is not a group then its group part is strictly smaller than the whole algebra.
Leaving the rest of the first algebra unchanged will, formally, mean that we consider a Cartesian product there, too, 
to enable a coordinatewise treatment of the operations, namely, the Cartesian product of the rest of the first algebra with a singleton.
This example has inspired the type II partial lexicographic product construction.
%Moreover, if the second structure is only involutive, one obtains only an involutive FL$_e$-algebra.

Moreover, it turns out that instead of enlarging the {\em group part} one can enlarge any {\em subgroup} of the group part.
This makes our lexicographic product construction definitively partial in nature, since even if the first algebra is a group (and hence its group part is the whole algebra), an even smaller algebra (a subalgebra of the group part) can be enlarged, by the second algebra extended by a top.
This option of {\em partial} enlargement inspires the type IV partial lexicographic construction defined in item~(B) of Definition~\ref{FoKonstrukcio}.
%$\mathbf Z=\mathbf {gr}(\mathbf X)$ (that is, the subalgebra formed by the invertible elements; in our example since $\mathbb Z$ is a group it coincided with its group part); 
%Our example demonstrates a special case of the type II partial lexicographic construction, %(where all the elements of $X$ are invertible and thus $Z=X$) defined in item~(B) of Definition~\ref{FoKonstrukcio}.
%Finally, observe that there is a natural embedding of the first algebra into the newly constructed one via $a\mapsto (a,\text{-}\frac{1}{2})$. Interestingly, there is another copy of the first algebra in the newly constructed one given by $a\mapsto (a,0)$.
%Therefore, the (images under this embedding of the) positive idempotent elements of the first algebra remain positive and idempotent in the newly constructed algebra, too. In addition, since the second algebra is a group and hence it has only a single idempotent element (its unit element $\text{-}\frac{1}{2}$),  a single new positive idempotent element, namely $(\text{-}\frac{1}{2},0)$ arises\footnote{In the first coordinate $\text{-}\frac{1}{2}$ is the unit element, while $0$ is the annihilator in the second one.}, thus the newly constructed algebra has exactly one more positive idempotent element than the algebra, which is used in the first coordinate.
\end{example}

\begin{example}
The odd involutive FL$_e$-chain which is to be defined in this example is 
$\PLPI{\pmb{\mathbb{R}}}{\pmb{\mathbb{Z}}}{\pmb{\mathbb{R}}}$
(up to isomorphism), %, a shorthand for $\PLPIV{\pmb{\mathbb Z}}{\pmb{\mathbb Z}}{\pmb{\mathbb R}}$), 
see item~A in Definition~\ref{FoKonstrukcio} for its formal definition.

In order to further exploit the partial nature of our lexicographic product construction, take the subgroup $\pmb{\mathbb Z}=(\mathbb Z,+,0)$ of the linearly ordered abelian group of the reals $\pmb{\mathbb R}=(\mathbb R,+,0)$.
Our plan is to replace each integer inside the reals by a whole copy of the reals (isomorphically, by $I$ from ${\mathbf I}=(I,\star,\text{-}\frac{1}{2})$ of Example~\ref{kiveteles2}), 
% in order to be able to exhibit a plot of the graph of the monoidal operation of the resulting algebra in the end), 
and to define an odd involutive FL$_e$-algebra over it based on $\star$.
Again, if we wish to obtain a universe, which is order isomorphic to $\mathbb R$, then replacing each integer inside the reals by $I=\ ]\text{-}1,0[$ is not convenient; the resulting universe will not be topologically connected, as it is easy to see. 
To overcome this we extend $\star$ to $I^{\top\bot}:=I\cup\{\text{-}1,0\}=[\text{-}1,0]$ by letting first $0$ be annihilator over $]\text{-}1,0]$, and then $\text{-}1$ be annihilator over $[\text{-}1,0]$.
In other words, we add a new top and a new bottom to the structure $\mathbf I$, thus obtaining $\mathbf I^{\top\bot}$.
Now, let the new universe be $\mathbb Z\times I^{\top\bot}\cup (\mathbb R\setminus \mathbb Z)\times \{\text{-}1\}$.
Again, we define $\te$ coordinatewise by letting $\g{(a_1,a_2)}{(b_1,b_2)}=(a_1+b_1,\gstar{a_2}{b_2})$, and 
we define $\komp$, using the natural order-reversing involutions 
($\negaM{\ast}{a}=\text{-}a$ on $\mathbb R$ and $\negaM{\star}{x}=\text{-}x-1$ on $I^{\top\bot}$)
also coordinatewise by 
$$
\nega{(x,y)}=\left\{
\begin{array}{ll}
(\negaM{\ast}{x},\negaM{\star}{y}) 	& \mbox{if $x\in \mathbb Z$}\\
(\negaM{\ast}{x},\text{-}1) 		& \mbox{if $x\in\mathbb R\setminus\mathbb Z$}\\
\end{array}
\right. .
$$
Finally, let 
$$
\res{\te}{(a_1,a_2)}{(b_1,b_2)}=\nega{\left(\g{(a_1,a_2)}{\nega{(b_1,b_2)}}\right)} .
$$
What we obtain is an odd involutive FL$_e$-chain, see the graph of an order-isomorphic copy of $\te$ in Figure~\ref{fig:1}, left.
Here, too, there is only one strictly positive idempotent element in the constructed algebra, namely $(0,0)$, while $(0,\text{-}\frac{1}{2})$ is its unit element. 
Note that unlike in the previous example here the residual complement of $(0,0)$ is idempotent: 
$\nega{(0,0)}=(0,\text{-}1)$ and 
$\g{(0,\text{-}1)}{(0,\text{-}1)}=(0,\text{-}1)$.
Compare the example above with the formal definition of 
$\PLPI{\pmb{\mathbb{R}}}{\pmb{\mathbb{Z}}}{\pmb{\mathbb{R}}}$
in item~A in Definition~\ref{FoKonstrukcio}.
Further inspection reveals that in these kind of examples one does not necessarily have to work with groups (like $\pmb{\mathbb R}$ together with its subgroup $\pmb{\mathbb Z}$ in the first coordinate, and $\pmb{\mathbb R}$ in the second), 
but any odd involutive FL$_e$-algebra with any of its subgroup suffices in the first coordinate, and any odd involutive FL$_e$-algebra suffices in the second one.
%Here, too, if the second algebra is only involutive, one obtains only an involutive FL$_e$-algebra in the end.
Summing up, we start with an odd involutive FL$_e$-algebra and we only enlarge a subgroup of it by an arbitrary odd involutive FL$_e$-algebra equipped with top and bottom. 
This example inspired the type I partial lexicographic product construction.
%This is the essence of a type I partial lexicographic construction, see its formal definition in item~(A) of Definition~\ref{FoKonstrukcio}.

Moreover, it turns out that instead of enlarging a subgroup, one can 
(1) enlarge only a subgroup of a subgroup by the second algebra equipped with top and bottom, plus 
(2) 
enlarge the difference of the two subgroups by only the top and the bottom.
This second step is equivalent to enlarging the difference of the two subgroups also by the second algebra equipped with top and bottom, and then {\em removing} the second algebra, thus leaving only its top and bottom there. 
Ultimately, at this removal step we {\em take a subalgebra} compared to the corresponding algebra which is constructed by the type I variant, see Theorem~\ref{SubLexiTheo}.
The option of {\em partial} enlargement (when only a subalgebra of a subalgebra is enlarged) in this example inspires the type III partial lexicographic construction defined in item~(A) of Definition~\ref{FoKonstrukcio}.

Moreover, it turns out that 
in order to describe the class of odd involutive FL$_e$-chains having finitely many idempotent elements in its full generality, 
it is also possible to take such subalgebras of the group-part of the constructed algebra which are {\em not} the direct products of their projections. It motivates the most general construction, called partial {\em sub}lex products in Definition~B.
\end{example}

\section{Constructing involutive FL$_e$-algebras -- Partial sublex products}\label{SectPLPs}
We start with a notation.
\begin{definition}\label{nyilak}\rm
Let $(X, \leq)$ be a 
poset.
%chain (a linearly ordered set).
For $x\in X$ define 
%$$
%x_\downarrow=\left\{
%\begin{array}{ll}
%z & \mbox{if there exists $z<x$ such that there is no element between $z$ and $x$,}\\
%x & \mbox{if for any $z<x$ there exists $v\in X$ such that $z<v<x$ holds.}\\
%\end{array}
%\right. 
%$$
$$
x_\downarrow=
\left\{
\begin{array}{ll}
z & \mbox{if there exists a unique $z\in X$ such that $x$ covers $z$,}\\
x & \mbox{otherwise.}\\
\end{array}
\right.
$$
We define $x_\uparrow$ dually.
Note that if $\komp$ is an order-reversing involution of $X$ then it holds true that 
\begin{equation}\label{FelNeg_NegLe}
\nega{x}_\uparrow=\nega{(x_\downarrow)} \ \ \   \mbox{ and }  \ \ \ \nega{x}_\downarrow=\nega{(x_\uparrow)} .
\end{equation}
We say for $Z\subseteq X$ that $Z$ is discretely embedded into $X$ if 
for $x\in Z$ it holds true that $x\notin\{ x_\uparrow,x_\downarrow\}\subseteq Z$. % ( $\downarrow$ and $\uparrow$ are computed in $X$).
%Just like in the examples of the previous section, we denote the lexicographic product of two linearly ordered sets by $\lex$.
\end{definition}

Next, we introduce a construction, called {\em partial lexicographic product} 
(or partial lex product, or partial lex extension)
with four slightly different variations in Definition~\ref{FoKonstrukcio}. 
This definition lays the foundation for the most general construction in Definition~B, needed for our structural description purposes.
Roughly, only a subalgebra is used as a first component of a lexicographic product and the rest of the algebra is left unchanged, hence the adjective \lq partial\rq.
This results in an involutive FL$_e$-algebra, which is odd, too, provided that the second component of the lexicographic product is so, see Theorem~\ref{SubLexiTheo}.
%Type I and type II constructions will play a key role in the induction step of Theorem~\ref{Hahn_type}.\\

%$\wideparen{\times}$
In Remark~\ref{KonnyuLesz} we shall refer to some algebraic notions which appear later in the paper at the algebraic decomposition part, thus trying to make a bridge between the different components of the construction part and the decomposition part of the paper.

\begin{definition}\label{FoKonstrukcio}
%{\bf (Partial Lexicographic Products)}
\rm
Let ${\mathbf X}=(X, \wedge_X,\vee_X, \ast, \ite{\ast}, t_X, f_X)$ be an odd involutive FL$_e$-algebra
and ${\mathbf Y}=( Y, \wedge_Y,\vee_Y, \star, \ite{\star}, t_Y, f_Y )$
be an involutive FL$_e$-algebra, with residual complement $\kompM{\ast}$ and $\kompM{\star}$, respectively.
%\begin{enumerate}

\medskip
\begin{enumerate}
\item[A.]
Add a new element $\top$ to $Y$ as a top element and annihilator (for $\star$), then add a new element $\bot$ to $Y\cup\{\top\}$ as a bottom element and annihilator.
Extend $\kompM{\star}$ by $\negaM{\star}{\bot}=\top$ and $\negaM{\star}{\top}=\bot$.
%\begin{enumerate}
%\item[(III)]\label{ittnincstetejealja} 
Let 
$\mathbf V\leq\mathbf Z\leq\mathbf X_\mathbf{gr}$. 
%$\mathbf Z$ is a subalgebra of $\mathbf X_\mathbf{gr}$ and $\mathbf V$ be a subalgebra of $\mathbf Z$.
Let 
%$\PLPIII{X}{Z}{V}{Y}= (V\times Y)\cup (Z\times\{\top,\bot\})\cup \left((X\setminus Z)\times \{\bot\}\right),$
$$
\PLPIII{X}{Z}{V}{Y}= (V\times (Y\cup\{\top,\bot\}))\cup ((Z\setminus V)\times\{\top,\bot\}) \cup \left((X\setminus Z)\times \{\bot\}\right)
$$
and define $\PLPIII{\mathbf X}{\mathbf Z}{\mathbf V}{\mathbf Y}$, the {\em type III partial lexicographic product} of $\mathbf X,\mathbf Z,\mathbf V$ and $\mathbf Y$ as follows:
$$\PLPIII{\mathbf X}{\mathbf Z}{\mathbf V}{\mathbf Y}=\left(\PLPIII{X}{Z}{V}{Y}, \leq, \te, \ite{\te}, (t_X,t_Y),(f_X,f_Y)\right),$$
where $\leq$ is the restriction of the lexicographical order of $\leq_X$ and {\small $\leq_{Y\cup\{\top,\bot\}}$} to $\PLPIII{X}{Z}{V}{Y}$, 
$\te$ is defined coordinatewise, and the operation $\ite{\te}$ is given by
$
\res{\te}{(x_1,y_1)}{(x_2,y_2)}=\nega{\left(\g{(x_1,y_1)}{\nega{(x_2,y_2)}}\right)} ,
$
where
$$
\nega{(x,y)}=\left\{
\begin{array}{ll}
(\negaM{\ast}{x},\bot) 		& \mbox{if $x\not\in Z$}\\
(\negaM{\ast}{x},\negaM{\star}{y}) 	& \mbox{if $x\in Z$}\\
\end{array}
\right. .
$$
In the particular case when $\mathbf V=\mathbf Z$, we use the simpler notation 
$\PLPI{\mathbf X}{\mathbf Z}{\mathbf Y}$ for $\PLPIII{\mathbf X}{\mathbf Z}{\mathbf V}{\mathbf Y}$
and call it
the {\em type I partial lexicographic product} 
of $\mathbf X,\mathbf Z$, and $\mathbf Y$.
\item[B.]\label{B}
Assume that $X_{gr}$ 
%=( X_{gr}, \wedge, \vee, \ast, \ite{\ast}, t_X, f_X )$ 
is discretely embedded into $X$.
Add a new element $\top$ to $Y$ as a top element and annihilator.
%and extend $\star$ by ${\top}\star{y}={y}\star{\top}=\top$ for $y\in Y\cup\{\top\}$. 
%Let $\mathbf Z=( X_{gr}, \wedge, \vee, \ast, \ite{\ast}, t_X, f_X )$ be a linearly ordered, discretely embedded\footnote{We mean that for $x\in X_{gr}$, it holds true that $x\notin\{ x_\uparrow,x_\downarrow\}\subset Z$ ( $\downarrow$ and $\uparrow$ are computed in $X$).}, prime and cancellative subalgebra of $\mathbf X$.
%\item[(II)]
Let 
%$\mathbf V$ be a subalgebra of $\mathbf X_\mathbf{gr}$. 
$\mathbf V\leq\mathbf X_\mathbf{gr}$. 
Let
$$
%\PLPIV{X}{V}{Y}=(V\times  Y)\cup (X_{gr}\times\{\top\})\cup \left((X\setminus X_{gr})\times \{\top\}\right)
\PLPIV{X}{V}{Y}=(V\times Y)\cup(X\times \{\top\})
$$
and define 
$\PLPIV{\mathbf X}{\mathbf V}{\mathbf Y}$, 
the {\em type IV partial lexicographic product} of $\mathbf X$, $\mathbf V$ and $\mathbf Y$ as follows:
$$\PLPIV{\mathbf X}{\mathbf V}{\mathbf Y}=\left(\PLPIV{X}{V}{Y}, \leq, \te, \ite{\te}, (t_X,t_Y),(f_X,f_Y)\right),$$
where $\leq$ is the restriction of the lexicographical order of $\leq_X$ and $\leq_{ Y\cup\{\top\}}$ to 
$\PLPIV{X}{V}{Y}$,
$\te$ is defined coordinatewise, and the operation $\ite{\te}$ is given by
$
\res{\te}{(x_1,y_1)}{(x_2,y_2)}=\nega{\left(\g{(x_1,y_1)}{\nega{(x_2,y_2)}}\right)} ,
$
where $\komp$ is defined coordinatewise\footnote{Note that intuitively it would make up for a coordinatewise definition, too, in the second line of (\ref{FuraNegaEXT}) to define it as $(\negaM{\ast}{x},\bot)$. 
But $\bot$ is not amongst the set of possible second coordinates. However, since $X_{gr}$ is discretely embedded into $X$, if $(\negaM{\ast}{x},\bot)$ would be an element of the algebra then it would be equal to $((\negaM{\ast}{x})_\downarrow,\top)$. 
}
by
\begin{equation}\label{FuraNegaEXT}
\nega{(x,y)}=
\left\{
\begin{array}{ll}
(\negaM{\ast}{x},\top) 			& \mbox{if $x\not\in X_{gr}$ and $y=\top$}\\
((\negaM{\ast}{x})_\downarrow,\top) 	& \mbox{if $x\in X_{gr}$ and $y=\top$}\\
(\negaM{\ast}{x},\negaM{\star}{y}) 	& \mbox{if $x\in V$ and $y\in Y$}\\
\end{array}
\right. .
\end{equation}
In the particular case when $\mathbf V=\mathbf X_{\mathbf{gr}}$, we use the simpler notation 
$\PLPII{\mathbf X}{\mathbf Y}$ for $\PLPIV{\mathbf X}{\mathbf V}{\mathbf Y}$
and call it
the {\em type II partial lexicographic product} 
of $\mathbf X$ and $\mathbf Y$.
\end{enumerate}
\end{definition}

\begin{figure}\label{fig_sdughsdgh}
\begin{center}
  \includegraphics[width=0.39\textwidth]{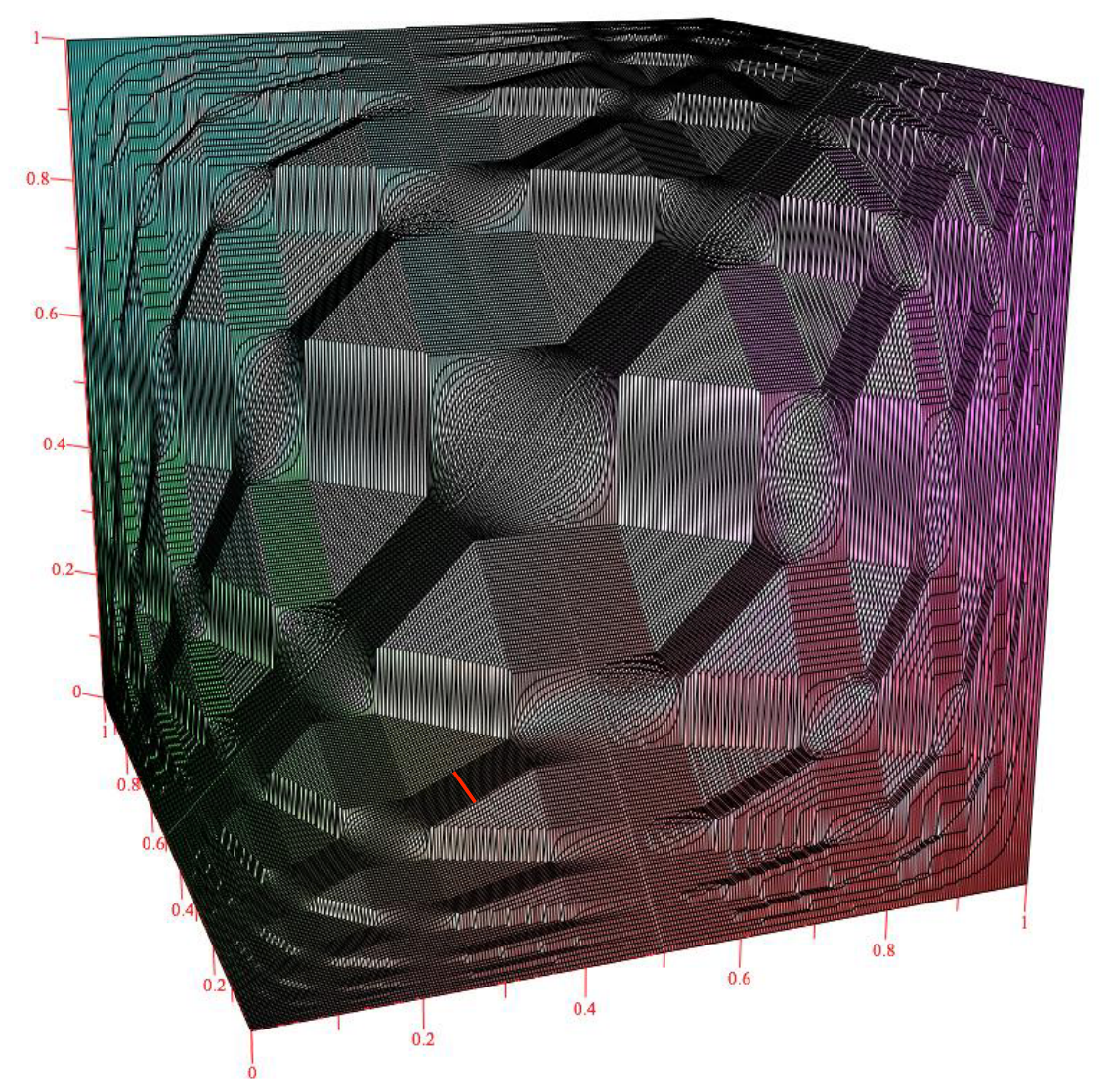} \ \ \ \ \ \ 
  \includegraphics[width=0.37\textwidth]{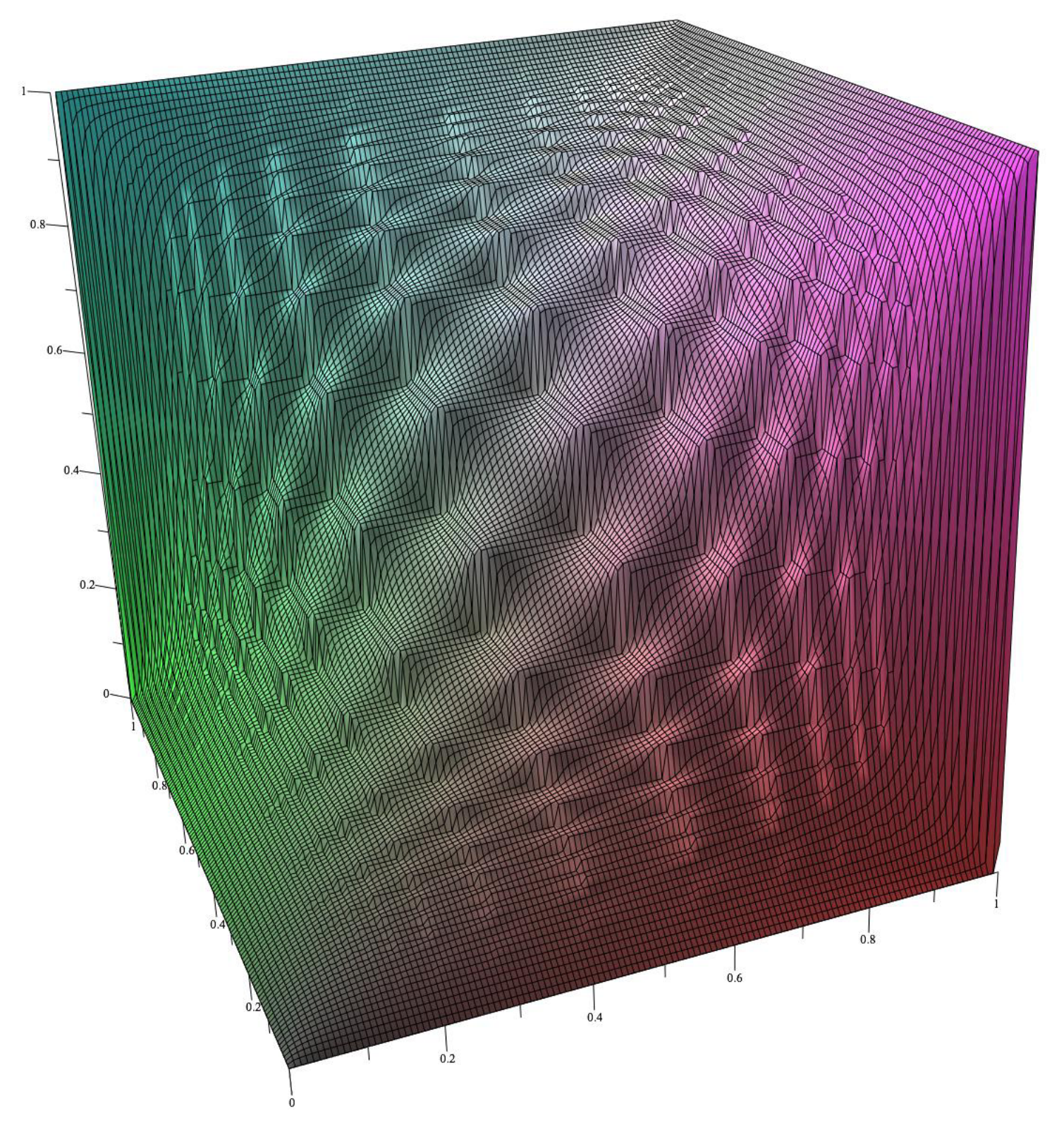}
\caption{Visualization:
$\PLPI{\pmb{\mathbb{R}}}{\pmb{\mathbb{Z}}}{\pmb{\mathbb{R}}}$ (left) and
$\PLPII{\pmb{\mathbb{Z}}}{\pmb{\mathbb{R}}}$ (right) shrank into $]0,1[$}
\label{fig:1}
\end{center}
\end{figure}

In order to reduce the number of different cases to be considered when checking the residuated nature of some structure (e.g. in the proof of Theorem~\ref{SubLexiTheo}) we will rely on the following statement, which generalizes the well-known equivalent formulation of involutive FL$_e$-algebras in terms of dualizing constants.

\begin{proposition}\label{PropClaim}
Let ${\mathcal M}=(M,\leq,\te)$ be a structure such that $(M,\leq)$ is a poset and $(M,\te)$ is a semigroup.
Call $c\in M$ a dualizing element\footnote{Dualizing elements are defined in {\em residuated} structures in the literature, see e.g. \cite[Section 3.4.17.]{gjko}.}  of $\mathcal M$, if 
(i) for $x\in M$ there exists $\res{\te}{x}{c}$\footnote{Here we do not assume that $\te$ is residuated. We only postulate that the greatest element of the set $\{ z\in M \ | \ \g{x}{z\leq c}\}$ exists for all $x$ and a fixed $c$.}, and 
(ii) for $x\in M$, $\res{\te}{(\res{\te}{x}{c})}{c}=x$.
\\
%Let $(M,\leq)$ be a poset, $\te$ be a semigroup operation on $M$. 
If there exists a dualizing element $c$ of $\mathcal M$ then $\te$ is residuated, and its residual operation is given by $\res{\te}{x}{y}=\res{\te}{\left(\g{x}{(\res{\te}{y}{c})}\right)}{c}$.
\end{proposition}
\begin{proof}
Indeed, 
%by definition, we have $\res{\te}{x}{y}=\bigvee\{z\ | \ \g{z}{x}\leq y\}$ and $\Rres{x}{y}=\bigvee\{z\ | \ \g{x}{z}\leq y\}$, provided that the right-hand sides exist.
$\g{z}{x}\leq y$ is equivalent to $\g{z}{x}\leq\res{\te}{\left(\res{\te}{y}{c}\right)}{c}$.
By adjointness it is equivalent to $\g{\left(\g{z}{x}\right)}{(\res{\te}{y}{c})}\leq c$.
By associativity it is equivalent to $\g{z}{\left(\g{x}{(\res{\te}{y}{c})}\right)}\leq c$, which is equivalent to $z\leq \res{\te}{\left(\g{x}{(\res{\te}{y}{c})}\right)}{c}$ by adjointness.
Thus we obtained $\res{\te}{x}{y}=\res{\te}{\left(\g{x}{(\res{\te}{y}{c})}\right)}{c}$, as stated.
\end{proof}

\begin{theorem}\label{SubLexiTheo}
Adapt the notation of Definition~\ref{FoKonstrukcio}.
%$\PLPI{\mathbf X}{\mathbf Z}{\mathbf Y}$,
$\PLPIII{\mathbf X}{\mathbf Z}{\mathbf V}{\mathbf Y}$
%$\PLPII{\mathbf X}{\mathbf Y}$
and
$\PLPIV{\mathbf X}{\mathbf V}{\mathbf Y}$
are
involutive FL$_e$-algebras with the same rank as that of $\mathbf Y$.
In particular, if $\mathbf Y$ is odd then so are 
%$\PLPI{\mathbf X}{\mathbf Z}{\mathbf Y}$,
$\PLPIII{\mathbf X}{\mathbf Z}{\mathbf V}{\mathbf Y}$
%$\PLPII{\mathbf X}{\mathbf Y}$
and
$\PLPIV{\mathbf X}{\mathbf V}{\mathbf Y}$.
In addition, 
$\PLPIII{\mathbf X}{\mathbf Z}{\mathbf V}{\mathbf Y}\leq\PLPI{\mathbf X}{\mathbf Z}{\mathbf Y}$ and
$\PLPIV{\mathbf X}{\mathbf V}{\mathbf Y}\leq\PLPII{\mathbf X}{\mathbf Y}$.
\end{theorem}
\begin{proof}
\item[A.]
\medskip\noindent
{\em Proof for the type III extension}
$\PLPIII{\mathbf X}{\mathbf Z}{\mathbf V}{\mathbf Y}$
\nopagebreak
\\
Clearly, $\komp$ is an order-reversing involution on $\PLPIII{X}{Z}{V}{Y}$.
When checking that $\PLPIII{X}{Z}{V}{Y}$ is closed under $\te$, the only non-trivial cases are 
(i)
when the product of the first coordinates in not in $Z$, because then the product of the second coordinates has to be $\bot$, and
(ii)
when the product of the first coordinates in in $Z\setminus V$, because then the product of the second coordinates has to be $\top$ or $\bot$. 
But if the product of the second coordinates is not $\bot$ then none of the second coordinates can be $\bot$ (since $\bot$ in annihilator), hence both first coordinates must be in $Z$, as required.
Second, if the product of the second coordinates is neither $\top$ nor $\bot$ then none of the second coordinates can be $\top$ or $\bot$ (because of the product table for $\top$ and $\bot$), hence both first coordinates must be in $V$, as required.

Since $\bot$ is annihilator, $\te$ can be expressed as
$$
\g{(x_1,y_1)}{(x_2,y_2)}=
\left\{
\begin{array}{ll}
(x_1\ast x_2,y_1\star y_2) & \mbox{if $x_1,x_2\in Z$}\\
(x_1\ast x_2,\bot) & \mbox{otherwise}\\
\end{array}
\right. ,
$$
hence associativity of $\te$ and that $(t_X,t_Y)$ is the unit of $\te$ are straightforward. 
%A moment's reflection shows that $(\PLPI{X}{Z}{Y},\te,(t_X,t_Y))$ is a submonoid of the full lexicographic product $(X\times Y,\te,(t_X,t_Y))$. Indeed, if $x_1\in Z$ and $z\not\in Z$ then $\g{(x_1,a)}{(z,\bot)}=(x_1\ast z,\bot)\in \PLPI{X}{Z}{Y}$ since $x_1\ast z\in Z$ by Theorem~\ref{csoportLesz} would imply $\nega{x_1}\ast x_1\ast z=t\ast z=z\in Z$, a contradiction. Also, if $z_1,z_2\not\in Z$ then $z_1\ast z_2\not\in Z$, too, since $\mathbf Z$ is prime; thus $(z_1\ast z_2,\bot)\in \PLPI{X}{Z}{Y}$.

Next, we state that $\te$ is residuated. By Proposition~\ref{PropClaim} it suffices to prove that
$(f_X,f_Y)$ is a dualizing element of $(\PLPIII{X}{Z}{V}{Y},\leq,\te)$.
In more details, that for any $(x,y)\in \PLPIII{X}{Z}{V}{Y}$,
%$$
%\mbox{$\res{\te}{(x,y)}{(f_X,f_Y)}$ exists, and it equals to $\nega{(x,y)}$.}
%$$
$\res{\te}{(x,y)}{(f_X,f_Y)}$ exists, and it equals to $\nega{(x,y)}$ (and thus $(x,y)\mapsto\res{\te}{(x,y)}{(f_X,f_Y)}$ is of order two).
Equivalently, that for $(x_1,y_1),(x_2,y_2)\in \PLPIII{X}{Z}{V}{Y}$, 
\begin{equation}\label{dualizing}
\g{(x_1,y_1)}{(x_2,y_2)}\leq (f_X,f_Y) \mbox{ if and only if } (x_2,y_2)\leq \nega{(x_1,y_1)} .
\end{equation}
\item[-]
Assume that at least one of $x_1$ and $x_2$ is not in $Z$. Due to commutativity of $\ast$ and $\star$, and the involutivity of $\komp$, we may safely assume $x_2\not\in Z$.
Then $y_2=\bot$ and thus $\g{(x_1,y_1)}{(x_2,y_2)}%=(x_1\ast x_2,y_1\star \bot)
=(x_1\ast x_2,\bot)$ is less or equal to $(f_X,f_Y)$ if and only if $x_1\ast x_2\leq_X f_X$, if and only if $x_2\leq_X \negaM{\ast}{x_1}$ by residuation, if and only if $(x_2,\bot)\leq(\negaM{\ast}{x_1}, \ldots)=\nega{(x_1,y_1)}$, and we are done.
\item[-]
Therefore, in the sequel we will assume $x_1,x_2\in Z$.
Then $\nega{(x_1,y_1)}=(\negaM{\ast}{x_1},\negaM{\star}{y_1})$.

-- First let $(x_1\ast x_2,y_1\star y_2)\leq (f_X,f_Y)$.
This holds if and only if either $x_1\ast x_2<_Xf_X$, or $x_1\ast x_2=f_X$ and $y_1\star y_2\leq_Y f_Y$ holds.
In the first case, since $\negaM{\ast}{x_1}$ is the inverse of $x_1$, $x_2<_X\negaM{\ast}{x_1}$ follows, confirming $(x_2,y_2)\leq (\negaM{\ast}{x_1},\negaM{\star}{y_1})$, while in the second case $x_2=\negaM{\ast}{x_1}$ and by residuation $y_2\leq_Y\negaM{\star}{y_1}$ follows, hence $(x_2,y_2)\leq (\negaM{\ast}{x_1},\negaM{\star}{y_1})$ holds, as stated.

-- Second, let $(x_2,y_2)\leq (\negaM{\ast}{x_1},\negaM{\star}{y_1})$. 
This holds if and only if either $x_2<_X\negaM{\ast}{x_1}$, or $x_2=\negaM{\ast}{x_1}$ and $y_2\leq_Y\negaM{\star}{y_1}$.
In the first case $x_1\ast x_2<_X f_X$ and hence $(x_1\ast x_2,y_1\star y_2)< (f_X,f_Y)$, while in the second case $x_1\ast x_2=x_1\ast \negaM{\ast}{x_1}=f_X$ holds, and by residuation $y_1\star y_2\leq_Y f_Y$, hence the proof of (\ref{dualizing}) is concluded.
\\
Summing up, $\te$ is residuated and  
$\res{\te}{(x_1,y_1)}{(x_2,y_2)}=\nega{\left(\g{(x_1,y_1)}{\nega{(x_2,y_2)}}\right)}$ holds. 

If $\mathbf Y$ is odd (resp.\ positive or negative rank) then $\negaM{\star}{f_Y}=f_Y$ (resp.\ $\negaM{\star}{f_Y}>_Y f_Y$, $\negaM{\star}{f_Y}<_Y f_Y$) and hence the dualizing element $(f_X,f_Y)$ is a fixed point of $\komp$ (resp.\ $\nega{(f_X,f_Y)}>(f_X,f_Y)$ or $\nega{(f_X,f_Y)}<(f_X,f_Y)$); that is 
$\PLPIII{\mathbf X}{\mathbf Z}{\mathbf V}{\mathbf Y}$ is odd (resp.\ positive or negative rank), too.

\item[B.]
\medskip\noindent
{\em Proof for the type IV extension}
$\PLPIV{\mathbf X}{\mathbf V}{\mathbf Y}$
\nopagebreak
\\
It is straightforward to verify that $\komp$ is an order-reversing involution on $\PLPIV{X}{V}{Y}$ at the first and third rows of (\ref{FuraNegaEXT}).
At the second row of (\ref{FuraNegaEXT}), we can use that $X_{gr}$ is discretely embedded into $X$ and so (\ref{FelNeg_NegLe}) holds for $x\in X_{gr}$.

Since $\top$ annihilates all elements of $Y$, $\te$ can be expressed as
$$
\g{(x_1,y_1)}{(x_2,y_2)}=
\left\{
\begin{array}{ll}
(x_1\ast x_2,y_1\star y_2) & \mbox{if $x_1,x_2\in V$}\\
(x_1\ast x_2,\top) & \mbox{otherwise}\\
\end{array}
\right. .
$$
hence associativity of $\te$ and that $((t_X,t_Y))$ is the unit of $\te$ are straightforward. 
%Hence, it is straightforward that $(\PLPII{X}{Y},\te,(t_X,t_Y))$ is a submonoid of the full lexicographic product $(X\times  (Y\cup\{\top\}),\te,(t_X,t_Y))$.

Next, we state that $\te$ is residuated. By the claim, it suffices to prove that
$(f_X,f_Y)$ is a dualizing element of $(\PLPIV{X}{V}{Y},\leq,\te)$.
In more details, that for any $(x,y)\in \PLPIV{X}{V}{Y}$,
%$$
%\mbox{$\res{\te}{(x,y)}{(f_X,f_Y)}$ exists, and it equals to $\nega{(x,y)}$.}
%$$
$\res{\te}{(x,y)}{(f_X,f_Y)}$ exists, and it equals to $\nega{(x,y)}$ (and thus $(x,y)\mapsto\res{\te}{(x,y)}{(f_X,f_Y)}$ is of order two).
Equivalently, that
%we state 
for $(x_1,y_1),(x_2,y_2)\in \PLPIV{X}{V}{Y}$,
\begin{equation}\label{tenylegResCompTOPPAL}
\g{(x_1,y_1)}{(x_2,y_2)}\leq (f_X,f_Y) \mbox{ if and only if } (x_2,y_2)\leq \nega{(x_1,y_1)} .
\end{equation}
\item[-]
Assume $y_1,y_2\in Y$. Then $x_1,x_2\in V$ hold.
Then $(x_1\ast x_2,y_1\star y_2)\leq (f_X,f_Y)$ holds
if and only if $x_1\ast x_2<_X f_X$ holds 
or
if both $x_1\ast x_2=f_X$ and $y_1\star y_2\leq_Y f_Y$ hold.
Since since 
$\negaM{\ast}{x_1}$ is the inverse of $x_1$,
%$\mathbf V$ is a group, 
the condition above is equivalent to 
$x_2<_X\negaM{\ast}{x_1}$
or 
$x_2=\negaM{\ast}{x_1}$ together with $y_2\leq_Y\negaM{\star}{y_1}$, 
which is equivalent to $(x_2,y_2)\leq(\negaM{\ast}{x_1},\negaM{\star}{y_1})=\nega{(x_1,y_1)}$.
\item[-]
Assume $y_1\in Y$ and $y_2=\top$.
Then $x_1\in V$ holds.
Since $\top>f_Y$, 
$(x_1\ast x_2,\top)\leq (f_X,f_Y)$ holds if and only if $x_1\ast x_2<_X f_X$, which is equivalent to $x_2<_X\negaM{\ast}{x_1}$ since 
$\negaM{\ast}{x_1}$ is the inverse of $x_1$, which in turn is equivalent to  
$(x_2,\top)\leq(\negaM{\ast}{x_1},\negaM{\star}{y_1})=\nega{(x_1,y_1)}$ and we are done.
The case $y_1=\top$ and $y_2\in Y$ is completely analogous.
%follows from the previous point since $\te$ is commutative and $\komp$ is involutive in (\ref{tenylegResCompTOPPAL}).
\item[-]
Finally, assume $y_1,y_2=\top$. Then, since $\top>f_Y$, it holds true that
$$
\mbox{$(x_1\ast x_2,\top)\leq (f_X,f_Y)$ holds if and only if $x_1\ast x_2<_X f_X$.}
$$

-- If $x_1\in X_{gr}$
then $x_1\ast x_2<_X f_X$ is equivalent to $x_2<_X\negaM{\ast}{x_1}$ since then $\negaM{\ast}{x_1}$ is the inverse of $x_1$.
Since $X_{gr}$ is discretely embedded into 
$X$ and since $\negaM{\ast}{x_1}\in X_{gr}$, therefore $x_2<_X\negaM{\ast}{x_1}$ is equivalent to $x_2\leq_X(\negaM{\ast}{x_1})_\downarrow$, which is equivalent to
$(x_2,\top)\leq((\negaM{\ast}{x_1})_\downarrow,\top)=\nega{(x_1,\top)}$.
The case $x_2\in X_{gr}$ is completely analogous.

-- If $x_1,x_2\notin X_{gr}$ then $x_1\ast x_2$ cannot be equal to $f_X$, therefore $x_1\ast x_2<_X f_X$ can equivalently be written as $x_1\ast x_2\leq_X f_X$,
which is equivalent to $x_2\leq_X \negaM{\ast}{x_1}$ by residuation.
Finally, it is equivalent to $(x_2,\top)\leq(\negaM{\ast}{x_1},\top)=\nega{(x_1,\top)}$, hence the proof of (\ref{tenylegResCompTOPPAL}) is concluded.
\\
Summing up, $\te$ is residuated and  
$\res{\te}{(x_1,y_1)}{(x_2,y_2)}=\nega{\left(\g{(x_1,y_1)}{\nega{(x_2,y_2)}}\right)}$ holds. 

If $\mathbf Y$ is odd (resp.\ positive or negative rank) then $\negaM{\star}{f_Y}=f_Y$  (resp.\ $\negaM{\star}{f_Y}>f_Y$, $\negaM{\star}{f_Y}<f_Y$) and hence the dualizing element $(f_X,f_Y)$ is fixed under $\komp$ (resp.\ $\nega{(f_X,f_Y)}>(f_X,f_Y)$ or $\nega{(f_X,f_Y)}<(f_X,f_Y)$); that is 
$\PLPII{\mathbf X}{\mathbf Y}$ is odd (resp.\ positive or negative rank), too.

\item[C.]
\medskip\noindent
{\em Proof for}
$\PLPIII{\mathbf X}{\mathbf Z}{\mathbf V}{\mathbf Y}\leq\PLPI{\mathbf X}{\mathbf Z}{\mathbf Y}$ and
$\PLPIV{\mathbf X}{\mathbf V}{\mathbf Y}\leq\PLPII{\mathbf X}{\mathbf Y}$
\nopagebreak
\\
Since type I and II partial lexicographic products are particular cases of type III and IV partial lexicographic products, it follows from the previous two claims that $\PLPI{\mathbf X}{\mathbf V}{\mathbf Y}$ and $\PLPII{\mathbf X}{\mathbf Y}$ are involutive FL$_e$-algebras, too.
We shall verify that 
$\PLPIII{\mathbf X}{\mathbf Z}{\mathbf V}{\mathbf Y}\leq\PLPI{\mathbf X}{\mathbf Z}{\mathbf Y}$ and
$\PLPIV{\mathbf X}{\mathbf V}{\mathbf Y}\leq\PLPII{\mathbf X}{\mathbf Y}$.
Evidently
$\PLPIII{X}{Z}{V}{Y}\subseteq\PLPI{X}{Z}{Y}$ and 
$\PLPIV{X}{V}{Y}\subseteq\PLPII{X}{Y}$ hold.
A moment's reflection shows that 
$\PLPIII{X}{Z}{V}{Y}$ (resp. $\PLPIV{X}{V}{Y}$)
is closed under the residual complement operation of 
$\PLPI{\mathbf X}{\mathbf Z}{\mathbf Y}$ (resp.
$\PLPII{\mathbf X}{\mathbf Y}$),
and that 
%
%By using that each element in $Z$ (resp. in $V$) has inverse, one can also easily verify that 
$\PLPIII{X}{Z}{V}{Y}$ (resp. $\PLPIV{X}{V}{Y}$)
is closed under the monoidal operation of 
$\PLPI{\mathbf X}{\mathbf Z}{\mathbf Y}$ (resp.
$\PLPII{\mathbf X}{\mathbf Y}$).
From the respective definition of the residual complement operations %in $\mathbf{X_{\Gamma(X_1, Y^{\top\bot})}}$, $\mathbf{X_{\Gamma(X_1,Y^{\top})}}$, $\mathbf{X^{\Gamma(X_2, Y)}_{\Gamma(X_1, ^{\top\bot})}}$ and $\mathbf{X^{\Gamma(X_2, Y)}_{\Gamma(X_1, ^\top)}}$, 
it also easily follows that 
$\PLPIII{X}{Z}{V}{Y}$ (resp. $\PLPIV{X}{V}{Y}$)
is closed under the residual operation of 
$\PLPI{\mathbf X}{\mathbf Z}{\mathbf Y}$ (resp.
$\PLPII{\mathbf X}{\mathbf Y}$).
\end{proof}

%\begin{definition}\rm\label{DEFiota}
%We introduce the following notation.
%Let $\mathbf A_1$, $\mathbf A_2$, \ldots, $\mathbf A_n$ and
%$\mathbf D$ be FL$_e$-algebras, and let $\mathbf D\leq\mathbf A_1\lex\mathbf A_2\lex\ldots\lex\mathbf A_n$.
%If, for $1\leq k\leq n$, the projection operation $\nu_k$ of $\mathbf D$ given by 
%$\nu_k(D)=\{ x\in A_k : \mbox{ there exists $(a_1,a_2,\ldots,a_n)\in D$ such that $a_k=x$ } \}$
%is surjective (that is, $\nu_k(A)=A_k$) then we denote it by
%$\mathbf D\leq_\nu\mathbf A_1\lex\mathbf A_2\lex\ldots\lex\mathbf A_n$.
%\\
%If $\mathbf D$ is only embedded into $\mathbf A_1\lex\mathbf A_2\lex\ldots\lex\mathbf A_n$ such that for its image $\varphi(\mathbf D)$ it holds true that $\varphi(\mathbf D)\leq_\nu\mathbf A_1\lex\mathbf A_2\lex\ldots\lex\mathbf A_n$, then we denote it by 
%$\mathbf D\hookrightarrow_\nu\mathbf A_1\lex\mathbf A_2\lex\ldots\lex\mathbf A_n$.
%\end{definition}

\medskip\noindent{\textsc{Definition A.}}
We introduce the following notation.
Let $\mathbf A_1$, $\mathbf A_2$ and $\mathbf D$ be FL$_e$-algebras, and let $\mathbf D\leq\mathbf A_1\lex\mathbf A_2$.
If $\nu$, the projection operation to the first coordinate maps $\mathbf D$ onto $A_1$, that is, if 
$\nu(D)=\{ a_1\in A_1 : \mbox{ there exists $(a_1,a_2)\in D$ } \}=A_1$ then we denote it by
$$\mathbf D\leq_\nu\mathbf A_1\lex\mathbf A_2.$$
If $\mathbf D$ is only embedded into $\mathbf A_1\lex\mathbf A_2$ such that for its image $\varphi(\mathbf D)$ it holds true that $\varphi(\mathbf D)\leq_\nu\mathbf A_1\lex\mathbf A_2$, then we denote it by 
$$\mathbf D\hookrightarrow_\nu\mathbf A_1\lex\mathbf A_2.$$

%\begin{definition} Adapt the notation of Definition~\ref{FoKonstrukcio}. Let $\mathbf A$ be an odd involutive FL$_e$-algebra, and let $\mathbf D=\PLPIII{\mathbf X}{\mathbf Z}{\mathbf V}{\mathbf Y}$ or $\mathbf D=\PLPIV{\mathbf X}{\mathbf V}{\mathbf Y}$. We denote an order-embedding of $\mathbf A$ to $\mathbf D$ such that the projection $\nu_1$ to the first coordinate is surjective by $\mathbf A \overset{\nu_1}{\hookrightarrow} \mathbf D$. \end{definition}
Next, we introduce the construction for the key result of the paper, called {\em partial sublex product} 
(or partial sublex extension).

\medskip\noindent{\textsc{Definition B.}}
Adapt the notation of Definition~\ref{FoKonstrukcio}. 
% and let $\mathbf Y$ be a linearly ordered abelian group. 
Let 
$\mathbf{A}=\PLPIII{\mathbf X}{\mathbf Z}{\mathbf V}{\mathbf Y}$ and 
$\mathbf{B}=\PLPIV{\mathbf X}{\mathbf V}{\mathbf Y}$.
Then
$
A=
(V\times Y)\cup(Z\times \{\top\})\cup \left(X\times \{\bot\}\right)
$,
$
B=(V\times Y)\cup(X\times \{\top\}),
$
and the group part $\mathbf A_{\mathbf{gr}}$ of $\mathbf{A}$, as well as the group part $\mathbf B_{\mathbf{gr}}$ of $\mathbf{B}$ is the group $\mathbf{V}\lex\mathbf{Y_{\mathbf{gr}}}$, which is the subalgebra of $\mathbf{A}$ and of $\mathbf{B}$ over $V\times Y_{gr}$ in both cases.
Let
\begin{equation}\label{HHKHfelesleges}
\mathbf H\leq_\nu\mathbf{V}\lex\mathbf{Y}_{\mathbf{gr}}.
\end{equation}
Replace $\mathbf A_{\mathbf{gr}}$ in $\mathbf{A}$ and $\mathbf B_{\mathbf{gr}}$ in $\mathbf{B}$ by $\mathbf H$ to obtain $\mathbf{A}_\mathbf{H}$ and $\mathbf{B}_\mathbf{H}$, respectively.
More formally, let
$$
A_H
=
H\cup(V\times(Y\setminus Y_{gr})\cup(Z\times\{\top\}) \cup \left(X\times \{\bot\}\right)
,
$$
$$
B_H
=
H\cup(V\times(Y\setminus Y_{gr}))\cup(X\times \{\top\}).
$$
Then, by relying on Theorem~4.4, and on the fact that the product of an invertible element by a non-invertible one is non-invertible, it readily follows that
$A_H$ and $B_H$ are closed under all operations of $\mathbf A$ and $\mathbf B$ (including the residual complement operation), respectively, and thus 
$\mathbf{A}_\mathbf{H}$ is a subalgebra of $\mathbf A$ over $A_H$, and 
$\mathbf{B}_\mathbf{H}$ is a subalgebra of $\mathbf B$ over $B_H$.
Therefore, $\mathbf{A}_\mathbf{H}%=\left(\PLPIII{\mathbf X}{\mathbf Z}{\mathbf V}{\mathbf Y}\right)_\mathbf{G}
$ 
and $\mathbf{B}_\mathbf{H}%=\left(\PLPIV{\mathbf X}{\mathbf V}{\mathbf Y}\right)_\mathbf{G}
$
are odd involutive FL$_e$-algebras, which cannot be constructed by the lexicographic product construction, in general.

An important particular instance is if $\mathbf Y$ is cancellative, and this is what we only need in the sequel.
Then 
$A_H$ and $B_H$ become simpler:
\begin{equation}\label{EgyszerubbUnik}
\mbox{
$
A_H
=
H\cup(Z\times\{\top\}) \cup \left(X\times \{\bot\}\right)
,
$
$
B_H
=
H\cup(X\times \{\top\})$,
}
\end{equation}
and thus both definitions become independent of $V$ since we can also equivalently assume
\begin{equation}\label{EgyszerubbReszcsoportos}
\mathbf H\leq\mathbf{Z}\lex\mathbf{Y}
\end{equation}
instead of (\ref{HHKHfelesleges}), where in the type II case $\mathbf Z=\mathbf X_\mathbf{gr}$.
%In fact, both definitions become independent of $\mathbf Y$, too, since on the one hand $Y$ does not occur in (\ref{EgyszerubbUnik}), and on the other hand, XXX $\mathbf Y$ can be recovered kind of ??? from $\mathbf H$ as follows: $\mathbf Y=pr_{\mathbf G_i}(\mathbf H)=\{g\in G_i : \exists a\in Z_{i-1}\mbox{ such that } (a,g)\in H_i \}.$ %, and the definition of $B_G$ becomes independent of $Y$, too.
For this case (when $\mathbf Y$ is cancellative) we are going to use the following notation:
$\PLPIs{\mathbf X}{\mathbf Z}{\mathbf Y}{\mathbf{H}}$ 
for
$\mathbf{A}_\mathbf{H}$,
and 
$\PLPIIs{\mathbf X}{\mathbf X_\mathbf{gr}}{\mathbf Y}{\mathbf H}$
for 
$\mathbf{B}_\mathbf{H}$,
and will
call them 
the {\em type I partial sub-lexicographic product} of $\mathbf X$, $\mathbf Z$, $\mathbf Y$ and $\mathbf H$, and 
the {\em type II partial sub-lexicographic product} of $\mathbf X$, $\mathbf Y$ and $\mathbf H$, respectively.

\bigskip

In the rest of the paper we will not use the partial sublex product construction in its full generality. 
For describing the structure of odd involutive FL$_e$-chains with only finitely-many idempotent elements, first of all, we will apply it to the linearly ordered setting only.
Second, it will be sufficient to use FL$_e$-chains with finitely many positive idempotent elements and linearly ordered abelian groups to play the role of $\mathbf X$ and $\mathbf Y$, the algebras in the first and in the second coordinate, respectively.

Keeping it in mind,
in order to ease the understanding of the rest of the paper, we make the following

\begin{remark}\label{KonnyuLesz}
Concerning Definition~\ref{FoKonstrukcio}:
\begin{itemize}
\item
Those subsets 
of $\PLPIII{X}{Z}{V}{Y}$
and 
of $\PLPIV{X}{V}{Y}$
which are of the form $\{v\}\times Y$ for $v\in V$,
will be recovered in an arbitrary odd involutive FL$_e$-chain $\mathbf X$ 
in Section~\ref{SectCollapse},
and will be called {\em convex components of the group part of $\mathbf X$}.

\item
The inverse operation of the enlargement of each element of $V$ by $Y$ will be done in Section~\ref{SectCollapse} by a homomorphism $\beta$ which collapses each convex component of the group part of $\mathbf X$ into a singleton.

\item
In case the residual complement of the smallest strictly positive idempotent element of $\mathbf X$ is also idempotent then the homomorphism $\gamma\circ\beta$ will collapse each convex component of the group part of $\mathbf X$ {\em together with its infimum and supremum} (in other words, together with its extremals, see below) into a singleton in Section~\ref{OneDecIdemp}.

\item
Those elements of 
$\PLPIII{X}{Z}{V}{Y}$ which are of the form $(v,\top)$ or $(v,\bot)$ for $v\in V$,
and 
those elements of
$\PLPIV{X}{V}{Y}$ which are of the form $(v,\top)$ for $v\in V$
will be recovered in an arbitrary odd involutive FL$_e$-chain in Section~\ref{SectExtremals},
and will be called {\em extremals}.

\item
Those elements of 
$\PLPIII{X}{Z}{V}{Y}$ which are of the form $(z,\top)$ or $(z,\bot)$ for $z\in Z\setminus V$,
and 
those elements of
$\PLPIV{X}{V}{Y}$ which are of the form $(z,\top)$ for $z\in X_{gr}\setminus V$
will be recovered in an arbitrary odd involutive FL$_e$-chain in Section~\ref{SectGaps},
and will be called {\em pseudo extremals}.
We will call these elements {\em pseudo} extremals, since there is no convex component in the algebra
$\PLPIII{\mathbf X}{\mathbf Z}{\mathbf V}{\mathbf Y}$ and 
$\PLPIV{\mathbf X}{\mathbf V}{\mathbf Y}$
for which they would serve as infimum or supremum (i.e., extremal).
Their convex component could be there (if it were we would talk about a type I or a type II extension) but it is \lq missing\rq\ (and thus we have a type III or a type IV extension, only a subalgebra of the related type I or a type II extension, see the last statement in Theorem~\ref{SubLexiTheo}).
\end{itemize}
\end{remark}

\section{Subuniverses}\label{decompsection}
%\section{Sketch of the main theorem and its proof}
{\em Sketch of the main theorem (Theorem~\ref{Hahn_type}) and its proof.}
For any odd involutive FL$_e$-chain $\mathbf X$, such that there exists its smallest strictly positive idempotent element, we introduce two different decomposition methods in  Lemmas~\ref{GammaHomo} and \ref{decompXtaugeq}.
The algebra will be decomposed by either the first or the second variant depending on whether the residual complement of its smallest strictly positive idempotent element is idempotent or not. 
%Linearly ordered abelian groups are known to be either discretely embedded or dense in their order topology, our scope is related to the dense case.

\medskip
If the residual complement of the smallest strictly positive idempotent element is idempotent then 
\begin{enumerate}
\item
we will define two consecutive homomorphisms $\beta$ and $\gamma$, and will isolate a homomorphic image $\gamma(\beta(\mathbf X))$ of $\mathbf X$,
which has one less positive idempotent elements then $\mathbf X$.
In this sense $\gamma(\beta(\mathbf X))$ will be smaller than the original algebra $\mathbf X$.
\item
We will also isolate two subgroups of $\gamma(\beta(\mathbf X))$, called $\gamma(\beta(\mathbf X_{\tau\geq u}^E))$ and $\gamma(\beta(\mathbf X_{\tau\geq u}^{E_c}))$.
\item
%Using these two algebras together with the kernel of $\alpha$ (as the third algebra, kernel of $\alpha$),
Finally we will recover $\mathbf X$, up to isomorphism, as the type I partial sublex product of $\gamma(\beta(\mathbf X))$, $\gamma(\beta(\mathbf X_{\tau\geq u}^E))$, $\overline{\mathbf{ker}_\beta}$, and a subgroup $\mathbf G$ of 
$\gamma(\beta(\mathbf X_{\tau\geq u}^{E_c}))\lex\overline{\mathbf{ker}_\beta}$,
in Section~\ref{OneDecIdemp}.
\end{enumerate}

If the residual complement of the smallest strictly positive idempotent element is not idempotent then 
\begin{enumerate}
\item
we will isolate a residuated subsemigroup $\mathbf X_{\tau\geq u}$ of $\mathbf X$, which will be an odd involutive FL$_e$-algebra, albeit not a subalgebra of $\mathbf X$, since its unit element and residual complement operation will differ from 
those of $\mathbf X$.
$\mathbf X_{\tau\geq u}$ has one less positive idempotent element then $\mathbf X$, so here too, $\mathbf X_{\tau\geq u}$ will be smaller than the original algebra $\mathbf X$.
\item
Then we prove that the group-part of $\mathbf X_{\tau\geq u}$ is discretely embedded into $\mathbf X_{\tau\geq u}$, and we isolate a subgroup ${\mathbf X_{\tau\geq u}^{T_c}}$ of it. %$\mathbf X_{\tau\geq u}$.
\item
%Using these two algebras together with the kernel of $\alpha$ (as the third algebra, kernel of $\alpha$),
Finally we will recover $\mathbf X$, up to isomorphism, as the type II partial sublex product of $\mathbf X_{\tau\geq u}$, ${\mathbf X_{\tau\geq u}^T}$, $\mathbf{ker}_\beta$, and a subgroup $\mathbf G$ of 
${\mathbf X_{\tau\geq u}^{T_c}}\lex\overline{\mathbf{ker}_\beta}$,
in Section~\ref{OneStepNotIdemp}.
\end{enumerate}

Therefore, if the original odd involutive FL$_e$-chain $\mathbf X$ has finitely many positive idempotents then (since then there exists its {\em smallest} strictly positive idempotent) we can  reconstruct $\mathbf X$ from a smaller odd involutive FL$_e$-chain via enlarging it by a linearly ordered abelian group (either type I or II in Sections~\ref{OneDecIdemp} and \ref{OneStepNotIdemp}, respectively).
By applying this decomposition iteratively to the obtained smaller algebra we end up with an odd involutive FL$_e$-chain with a single idempotent element, which is a linearly ordered abelian group by Theorem~\ref{csoportLesz}.
Summing up, any odd involutive FL$_e$-chain which has only finitely many positive idempotent elements will be described by the consecutive application of the partial sublex product construction using only linearly ordered abelian groups as building blocks; as many of them as the number of the positive idempotent elements of $\mathbf X$.

\smallskip

To this end first we need the classification of elements of $\mathbf X$ as described in Definition~\ref{def A referred}.

\smallskip

It is well-known that for any involutive FL$_e$-chain, if $c$ and $\nega{c}$ are idempotent and the interval between them contains the two constants then there is a subalgebra over that interval.
However, in order to describe the structure of odd involutive FL$_e$-chains these subalgebras will not be convenient to consider, they are too small. Rather, it will be wiser to put all elements of $X$ sharing the {\em same $\tau$-value} into one class:
%one has to consider subuniverses which are defined by the help of the $\tau$ function.

\begin{definition}\label{def A referred} 
\rm
For an odd involutive FL$_e$-chain $( X, \leq, \te, \ite{\te}, t, f )$, for $u\geq t$ and $\Box\in\{<,=,\geq\}$ denote 
\begin{eqnarray*}
X_{\tau \mathbin{\Box} u} & = & \{x\in X : \tau(x) \mathbin{\Box} u\}.
\end{eqnarray*}
\end{definition}

\noindent
The definition above reveals the true nature of $\tau$:
note that $\tau$ acts as a kind of \lq local unit\rq\ function: $\tau(z)$ is the unit element for the subset $X_{\tau\geq \tau(z)}$ of $X$.

\begin{proposition}\label{elozetes}%{\bf (Subuniverses)}
Let ${\mathbf X}=( X, \leq, \te, \ite{\te}, t, f )$ be an odd involutive FL$_e$-chain.
Let $u>t$ be idempotent.
\begin{enumerate}
\item\label{subunis}
$X_{\tau<u}$ and $X_{\tau=t}$ 
%, $X_{\tau\geq u}\cup \{t\}$ 
are nonempty subuniverses. %{\color{blue} Ez legutobbit kivettem, mert ezt maskeppen definialom kesobb}
%\item\label{kulsoveliszart} If $x\in X_{\tau\geq u}$ and $y\in X$ then $\g{x}{y}\in X_{\tau\geq u}$.
\item\label{nagyonSzigoru}
$X_{\tau=t}=X_{gr}$ and hence $X_{\tau=t}$ is a nonempty subuniverse of a linearly ordered abelian group.
For $x\in X_{\tau=t}$, the map $y\mapsto \g{y}{x}$ $(y\in X)$ is strictly increasing.
\end{enumerate}
\end{proposition}
The respective subalgebras of $\mathbf X$ will be denoted by $\mathbf X_{\tau<u}$ and $\mathbf X_{\tau=t}=\mathbf X_\mathbf{gr}$. %, and $\mathbf X_{\tau\geq u}$.
In the proof of Proposition~\ref{elozetes} and also later a key role will be played by Lemma~\ref{tau_lemma}. To prove it, it will be useful to observe first 
\begin{proposition}\label{diagonalSZIGORU}{{\bf (Diagonally Strict Increase)}}
For any odd involutive FL$_e$-chain $( X, \leq, \te, \ite{\te}, t, f )$, 
if $x_1>x$ and $y_1>y$ then 
$$
\g{x_1}{y_1}>\g{x}{y}  .
$$
\end{proposition}
\begin{proof}
By (\ref{eq_feltukrozes}), $\nega{\left(\g{\nega{x}}{\nega{y}}\right)}\geq\g{x}{y}$ holds, hence it suffices to prove $\g{x_1}{y_1}>\nega{\left(\g{\nega{x}}{\nega{y}}\right)}$.
Assume the opposite, which is $\g{x_1}{y_1}\leq\res{\te}{\left(\g{\nega{x}}{\nega{y}}\right)}{f}$ since $(X,\leq)$ is a chain.
By adjointness from this
we obtain $\g{(\g{\nega{x}}{x_1})}{(\g{\nega{y}}{y_1})}=\g{(\g{x_1}{y_1})}{(\g{\nega{x}}{\nega{y}})}\leq f=t$.
On the other hand, from $x_1>x=\nega{(\nega{x})}=\res{\te}{\nega{x}}{f}$ it follows by residuation that $\g{\nega{x}}{x_1}\not\leq f$, which is equivalent to $\g{\nega{x}}{x_1}>f=t$ since $(X,\leq)$ is a chain.
Analogously we obtain $\g{\nega{y}}{y_1}>t$.
Therefore $\g{(\g{\nega{x}}{x_1})}{(\g{\nega{y}}{y_1})}\geq \g{(\g{\nega{x}}{x_1})}{t}=\g{\nega{x}}{x_1}>t$ follows, which is a contradiction.
\end{proof}
Claims~(\ref{FENTszimmetrikus}) and (\ref{lemasoljaGENzeta}) in Proposition~\ref{zeta_1} show how $\tau$ behaves in relation to $\komp$ and $\te$, respectively. 
If the algebra is odd we can state more.

\begin{lemma}\label{tau_lemma}{\bf {($\tau$-lemma)}}
Let $( X, \leq, \te, \ite{\te}, t, f )$ be an odd involutive FL$_e$-chain and 
$A$ be an expression which contains only the operations $\te$, $\ite{\te}$ and \, $\komp$.
For any evaluation $e$ of the variables and constants of $A$ into $X$, $\tau(e(A))$ equals the maximum of the $\tau$-values of the $e$-values of the variables and constants of $A$.
\end{lemma}
\begin{proof}
First we claim $\tau(\g{x}{y})=\max(\tau(x),\tau(y))$ for $x,y\in X$.
By claim~(\ref{lemasoljaGENzeta}) in Proposition~\ref{zeta_1}, $\tau(\g{x}{y})\geq\max(\tau(x),\tau(y))$ holds. 
Assume $\tau(\g{x}{y})>\max(\tau(x),\tau(y))$ and let $z\in ]\max(\tau(x),\tau(y)),\tau(\g{x}{y})](\neq\emptyset)$ be arbitrary. 
By (\ref{tnelnagyobb}) it holds true that $t\leq\tau(x)<z$, hence the isotonicity of $\te$ it holds true that $x=\g{t}{x}\leq\g{z}{x}$ and $y=\g{t}{y}\leq\g{z}{y}$.
Since $\tau$ assigns to $x$ the greatest element of the stabilizer set of $x$, and $z>\max(\tau(x),\tau(y))$, 
it follows that $z$ does not stabilize $x$ or $y$.
Consequently, $x<\g{z}{x}$ and $y<\g{z}{y}$ should hold.
Since $t<z$ and $z\leq \tau(\g{x}{y})$ yields $z\in Stab_{\g{x}{y}}$, 
it follows that $\g{(\g{z}{x})}{(\g{z}{y})}=\g{(\g{(\g{x}{y})}{z})}{z}=\g{(\g{x}{y})}{z}=\g{x}{y}$, a contradiction to Proposition~\ref{diagonalSZIGORU}. This settles the claim.
\\
By (\ref{eq_quasi_inverse}), any expression which contains only the connectives $\te$, $\ite{\te}$ and \, $\komp$
can be represented by an equivalent expression using the same variables and constants but containing only $\te$ and $\komp$. 
An easy induction on the recursive structure of this equivalent expression using the claim above and claim~(\ref{FENTszimmetrikus}) in Proposition~\ref{zeta_1} concludes the proof.
\end{proof}

\begin{proof} {\em of Proposition~\ref{elozetes}.}
Let $Z\in\{ X_{\tau<u} ,  X_{\tau=t} \}$.
Then $t\in Z$, and Lemma~\ref{tau_lemma} ensures that $Z$ is closed under $\te$ and $\ite{\te}$; %, since if for $x,y\in X$, both $\tau(x)$ and $\tau(y)$ are $>$, $\geq$, $=$, $\leq$, or $<$ than $c$, respectively, then so is their maximum $\tau(\g{x}{y})$.
%Since $t$ is the unit of the multiplication, $Z\cup\{t\}$ is closed under $\te$, too.
%$Z\cup\{t\}$ is closed under $\komp$ since $\nega{t}=t$; 
%By (\ref{eq_quasi_inverse}), $Z\cup\{t\}$ is closed under $\ite{\te}$, too, 
thus claim~(\ref{subunis}) is verified.
%claim~(\ref{lemasoljaGENzeta}) in Proposition~\ref{zeta_1} implies (\ref{kulsoveliszart}), and  
For $x\in X_{\tau=t}$, $\g{x}{\nega{x}}=\nega{(\res{\te}{x}{x})}=\nega{\tau(x)}=\nega{\nega{t}}=t$ holds by (\ref{eq_quasi_inverse}) and (\ref{tnelnagyobb}), and thus $x\in X_{gr}$, too.
For $x\in X_{gr}$, 
if $\g{x}{z}=\g{x}{t}$ then $\g{x^{-1}}{\g{x}{z}}=\g{x^{-1}}{\g{x}{t}}$, that is $z=t$. Hence $\tau(x)=t$, and thus $x\in X_{\tau=t}$ too.
These prove $X_{\tau=t}=X_{gr}$.
The rest of claim~\ref{nagyonSzigoru} follows from the cancellation property in groups.
\end{proof}

\section{Collapsing the convex components of the group part of $\mathbf X$ into singletons -- the homomorphism $\beta$}\label{SectCollapse}

%\begin{definition}\rm
An abelian group $\mathbf A$ is called {\em divisible} if for any $g\in A$ and $n$ positive integer there exists $a\in A$ such that $\underbrace{\g{a}{\g{\ldots}{a}}}_n=g$.
For any abelian group $\mathbf A$ there exists a minimal divisible abelian group $\overline{\mathbf A}$ containing $\mathbf A$.
$\overline{\mathbf A}$ is uniquely determined up to isomorphism and called the {\em divisible hull} (or {\em injective hull}) of $\mathbf A$ (\cite{FuchsAbelianGroups}).
%\end{definition}

\smallskip
The next lemma and its proof (unlike the rest of the paper) is written in the usual additive notation of the theory of abelian groups.

%\medskip Similar to Definition~\ref{iotaEMBEDDING1} is \begin{definition} Let $A$, $B$, and $C$ be linearly ordered algebras. We denote an order-embedding of $A$ into $B\lexADD C$ such that the projection to the first coordinate maps the image of $A$ onto $B$ by $A \overset{\nu_1}{\hookrightarrow} B\lexADD C$. \end{definition}

\medskip\noindent{\textsc{Lemma A.}}
Let $G$ be a linearly ordered abelian group and $C$ a convex subgroup of $G$. 
Then 
$G \hookrightarrow_\nu G/C\lexADD \overline{C}$. %, where $\overline C$ is the divisible hull of $C$.
\begin{proof}
Consider $L=\langle G,\overline{C}\rangle$. 
It is straightforward to check that the linear order of $G$ extends to a linear linear order of $L$ by letting, for $x,y\in L$, $x\leq y$ iff $lkg+lkc\leq klh+kld$, where $x=g+c$, $y=h+d$, $a,b\in G$, $c,d\in\overline{C}$, $k,l\in\mathbb Z^+$, $kc,ld\in C$.
\\
$\overline{C}$ is convex in $L$: Let $x,y\in\overline{C}$ and $a\in L$ such that $x<a<y$.
There exist $n,m\in\mathbb Z^+$ such that $nx,my\in C$,
there exist $g\in G$, $d\in\overline{C}$ such that $a=g+d$, and
there exists $k\in\mathbb Z^+$ such that $kd\in C$.
It follows that $C\ni kmnx-nmkd<nmkg<knmy-nmkd\in C$ and $nmkg\in G$.
Since $C$ is convex in $G$, it yields $nmkg\in C$ thus ensuring $g\in\overline{C}$ and hence $a=g+d\in\overline{C}$. %Hence $\overline{C}$ is convex in $L$.
\\
Since $\overline{C}$ is a divisible subgroup of $L$, and divisible subgroups are known to be direct summands, there exists $M\leq L$ such that $L=M\oplus\overline{C}$.
Both $M$ and $\overline{C}$, being subgroups of $L$, are linearly ordered by the ordering of $L$,
%It is straightforward to check that the linear order of $C$ extends to a linear linear order of $\overline{C}$ by letting, for $c,d\in\overline{C}$, $c\leq d$ if $lkc\leq kld$, where $k,l\in\mathbb Z^+$ and $kc,ld\in C$. 
giving rise to consider the lexicographic ordering on $M\oplus\overline{C}$.
It coincides with $\leq$.
Indeed, let $x,y\in L$ with $x=a+c$, $y=b+d$, $a,b\in M$, $c,d\in\overline{C}$.
Contrary to the statement assume that $x$ is larger than $y$ in the lexicographic order.
Then either $a=b$ and $c> d$ which contradicts to $x\leq y$, or $a> b$.
In the latter case $\overline{C}\ni d-c\geq a-b\geq 0\in\overline{C}$ follows from $x\leq y$, yielding $a-b\in\overline{C}$ since $\overline{C}$ is convex in $L$.
Therefore $a-b\in M\cap\overline{C}=\{0\}$, hence $a=b$, a contradiction.
%We denote the lexicographic order on $L$ also by $\leq$.

Therefore, $L=M\lexADD\overline{C}$.
By the first isomorphism theorem for ordered abelian groups, $M\cong L/\overline{C}$. 
In addition, $L/\overline{C}\cong G/C$ since $\varphi$ is a natural order isomorphism from $G/C$ to $L/\overline{C}$ given by $\varphi(g+C)=g+\overline{C}$:
%\\ - $\varphi$ is well-defined: If $h\in g+C$ then $h=g+c$ for some $c\in C$, and $\varphi(h+C)=h+\overline{C}=g+c+\overline{C}=g+\overline{C}$ since $C\subseteq\overline{C}$.
%injective: : If $h\notin g+C$ then $h-g\notin C\subseteq\overline{C}$ showing $h\notin g+\overline{C}$.
$\varphi$ is is well-defined and injective
since $C\subseteq\overline{C}$ and since cosets of subgroups are known to be either disjoint or equal;
$\varphi$ is surjective, too, 
%since for $a\in L$ there exist $g\in G$, $d\in C$ such that $a=g+d$, and so 
since for $L\ni a=g+c$, $g\in G$, $c\in C$ it holds true that 
$\varphi(g+C)=g+\overline{C}=a-c+\overline{C}=a+\overline{C}$ since $C\subseteq\overline{C}$; 
and finally, $\varphi$ trivially preserves addition and the ordering. Summing up,
$
G\leq L=M\lexADD\overline{C}\cong G/C\lexADD\overline{C}
$.
To see that the projection operation of $M\lexADD \overline{C}$ to the first coordinate maps $G$ onto $M$ let $a\in M$.
Since $M\leq L=\langle G,\overline{C}\rangle$, $a=g+c$ for some $g\in G$ and $c\in\overline{C}$.
Therefore, the unique decomposition of $g$ is $g=a-c$ and hence its projection to the first coordinate is $a$.
%To see that the projection operation of $M\lexADD \overline{C}$ to the second coordinate maps $G$ onto $\overline{C}$ let $c\in\overline{C}$. Since $\overline{C}\leq L=\langle G,\overline{C}\rangle$, $c=0_M+c$ for some $g\in G$ and $c\in\overline{C}$. Therefore, $g=a-c$ and its projection is $a$.
\end{proof}

Roughly, in algebra under a natural homomorphisms each element is mapped into its class, and the classes are often of the same size. 
%In residuated lattices congruences and homomorphism are in one-to-one correspondence (see, \cite[]{gjko}). The slight notational difficulty in the definition of $\beta$ above and in Lemma~\ref{faktorok} lies in the fact that there the classes of the congruence $\sim$ does not have a fixed size.
In Definition~\ref{ElsoCollapsusDefi} it is not the case: some elements will be mapped into their class (into their respective convex component with respect to $\tau< u$, see below), while on other elements the map will be the identity map.
In order to keep notation as simple as possible, these elements will be mapped into sets, too, namely to a singleton containing the element itself.

\begin{definition}\label{ElsoCollapsusDefi}\rm
Let $( X, \leq, \te, \ite{\te}, t, f )$ be an odd involutive FL$_e$-chain with residual complement $\komp$, $u>t$ idempotent.
For $x,y\in X_{\tau<u}$, define $x\sim y$ if $z\in X_{\tau<u}$ holds for any $z\in X$ with $x<z<y$.
It is an equivalence relation on $X_{\tau<u}$ since the order is linear. Denote the component of $x$ by $[x]_{\tau<u}$ and call it the convex component of $x$ with respect to $\tau< u$.
If $u$ is clear from the context we shall simply write $[x]$.

For $z\in X$ let
$$
\beta(z)=\left\{
\begin{array}{ll}
[z]_{\tau<u} & \mbox{if $z\in X_{\tau<u}$}\\
\{z\} & \mbox{if $z\in X_{\tau\geq u}$}
\end{array}
\right. .
$$
For $Z\subseteq X$
let
$\beta(Z)=\{\beta(z) : z\in Z\}$. % and $\beta(X_{\tau<u})=\{\beta(z) : z\in X_{\tau<u}\}$.
For $x,y\in\beta(X)$ let $p\in x, q\in y$
(equivalently, let $p,q\in X$, $x=\beta(p)$ and $y=\beta(q)$), 
and over $\beta(X)$ define
\begin{eqnarray}
& x\leq_{\beta}y & \mbox{iff }  p\leq q, \label{betas_rendezese}\\
& \gbeta{x}{y} & =  \beta(\g{p}{q}), \label{betas_szorzatos}\\
& \res{\beta}{x}{y} & =  \beta(\res{\te}{p}{q}), \label{betas_implikacios}\\
& \negaM{\beta}{x} & =  \beta(\nega{p}) \label{betas_nega}.
\end{eqnarray}
Finally, let $\beta(\mathbf X)=(\beta(X),\leq_\beta,\tebeta, \ite{\beta}, \beta(t),\beta(f))$.
\end{definition}

Although the constructions will be different for the two cases (depending on the idempotency of  the residual complement of the smallest strictly positive idempotent element in $\mathbf X$),
in both cases the proof relies on Lemma~\ref{faktorok}. 
The map $\beta$ which makes each convex component of $\mathbf X_\mathbf{gr}$ collapse into a singleton (its convex component with respect to $\tau< u$), and leaves the rest of the algebra unchanged, will be shown to be a homomorphism.
%In addition, the group part of $\mathbf X$ is characterized in

\begin{lemma}\label{faktorok}%{{\bf (Collapsing the convex components)}}
Let ${\mathbf X}=( X, \leq, \te, \ite{\te}, t, f )$ be an odd involutive FL$_e$-chain with residual complement $\komp$, and $u>t$ be idempotent.
\begin{enumerate}
\item\label{<_Xx6}
$\beta$ is a homomorphism from $\mathbf X$ to $\beta(\mathbf X)$, hence\\
$\beta(\mathbf X)$ is an odd involutive FL$_e$-chain with involution $\kompM{\beta}$,

\item\label{<_X1}
$\beta(\mathbf X_{\mathbf{gr}}):=(\beta(X_{gr}),\leq_\beta,\tebeta, \ite{\beta}, \beta(t), \beta(f))$ is a subgroup of $\mathbf X$.

%$\beta(\mathbf X_{\mathbf{gr}})=$ is a linearly ordered abelian group with inverse operation $\kompM{\beta}$.
\item\label{Akozepe}
$[t]_{\tau<u}=\, ]\nega{u},u[$ and it is a nonempty subuniverse of $\mathbf X$; denote the related subalgebra by $\mathbf{ker}_\beta$.

\item\label{lexikoS}
If $u$ is the smallest strictly positive idempotent element then
\begin{enumerate}
\item
$\mathbf X_{\tau<u}=\mathbf X_\mathbf{gr}$,
\item
$\mathbf{ker}_\beta$
is a convex subgroup of $\mathbf X_\mathbf{gr}$, 
\item
$\mathbf X_\mathbf {gr} \hookrightarrow_\nu   \beta(\mathbf X_{\mathbf{gr}})\lex\overline{\mathbf{ker}_\beta}$ \footnote{Here $\lex$ is the multiplicative notation of $\lexADD$ of Lemma~A.}
%$\mathbf X_\mathbf {gr}$ can be embedded (qua a linearly ordered abelian group) into $\beta(\mathbf X_{\mathbf{gr}})\lex\overline{\mathbf{ker}_\beta}$ and the projection operation of $\beta(\mathbf X_{\mathbf{gr}})\lex\overline{\mathbf{ker}_\beta}$ to the first coordinate maps $\mathbf X_\mathbf {gr}$ onto $\beta(\mathbf X_{\mathbf{gr}})$.
\end{enumerate}
\end{enumerate}
\end{lemma}
\begin{proof}
\medskip
\item[(\ref{<_Xx6})]
Proving that $\beta$ is a homomorphism amounts to proving that the definitions in (\ref{betas_rendezese})-(\ref{betas_nega}) are independent of the chosen representatives $p$ and $q$.
The definition of $\leq_\beta$ is independent of the choice of $p$ and $q$, since the components are convex. The ordering $\leq_{\beta}$ is linear since so is $\leq$. Therefore, (\ref{betas_rendezese}) is well-defined.
To prove that (\ref{betas_szorzatos}) is well-defined, we prove the following two statements:
First we prove that for $p_1,p_2,q\in X_{\tau<u}$, if $p_2\in[p_1]$ then $\g{p_2}{q}\in[\g{p_1}{q}]$.
Indeed, since $X$ is a chain, we may assume $p_1<p_2$.
If $\g{p_1}{q}=\g{p_2}{q}$ then we are done, therefore, by isotonicity of $\te$ we may assume $\g{p_1}{q}<\g{p_2}{q}$.
Contrary to the claim, suppose that there exists $a\in]\g{p_1}{q},\g{p_2}{q}[$ such that $\tau(a)\geq u$. 
By residuation, $\g{p_1}{q}<a$ implies $p_1\leq\res{\te}{q}{a}$, and since $X$ is a chain, $a<\g{p_2}{q}$ is equivalent to $\res{\te}{q}{a}<p_2$. Thus $\res{\te}{q}{a}\in[p_1,p_2[$ holds, hence $\tau(\res{\te}{q}{a})<u$ follows since $p_2\in[p_1]$ and $[p_1]$ is convex.
However, Lemma~\ref{tau_lemma} ensure $\tau(\res{\te}{q}{a})\geq\tau(a)\geq u$, a contradiction.
Second, we prove that for $p\in X_{\tau\geq u}$ and $q_1,q_2\in X_{\tau<u}$, if $q_2\in[q_1]$ then $\g{p}{q_1}=\g{p}{q_2}$.
Again, we may assume $q_1<q_2$ since $X$ is a chain.
Contrary to the claim suppose $\g{p}{q_1}<\g{p}{q_2}$.
By residuation it is equivalent to $\res{\te}{p}{(\g{p}{q_1})}<q_2$ since $X$ is a chain.
By residuation $q_1\leq\res{\te}{p}{(\g{p}{q_1})}$ holds, too.
%$\g{p}{q_1}\not\geq\g{p}{q_2}$ since $\leq_\beta$ is a linear order, which is equivalent to $\res{\te}{p}{(\g{p}{q_1})}\not\geq q_2$ by adjointness. Finally, it is equivalent to $\res{\te}{p}{(\g{p}{q_1})}<q_2$ since 
Thus $\res{\te}{p}{(\g{p}{q_1})}\in[q_1,q_2[$ holds, hence $q_2\in[q_1]$ and the convexity of $[q_1]$ implies $\tau(\res{\te}{p}{(\g{p}{q_1})})\leq u$.
However, Lemma~\ref{tau_lemma} ensures $\tau(\res{\te}{p}{(\g{p}{q_1})})\geq\tau(p)>u$, a contradiction.
Therefore, (\ref{betas_szorzatos}) is well-defined. %independent of the chosen representatives.
Next, claim~(\ref{FENTszimmetrikus}) in Proposition~\ref{zeta_1} implies that the image of a convex component under $\kompM{\beta}$ is a convex component, hence the definition in (\ref{betas_nega}) is independent of the chosen representative.
Finally, using (\ref{eq_quasi_inverse}) and that (\ref{betas_szorzatos}) and (\ref{betas_nega}) are well-defined, 
for $p_1,p_2,q_1,q_2\in X$, if $\beta(p_2)=\beta(p_1)$ and $\beta(q_2)=\beta(q_1)$ then it holds true that 
\begin{eqnarray}\label{IgyKellEzt}
\beta(\res{\te}{p_2}{q_2})&=&
\beta(\nega{(\g{p_2}{\nega{q_2}})})=
\negaM{\beta}{(\gbeta{\beta(p_2)}{\negaM{\beta}{\beta(q_2)}})}=
\\
&=&
\negaM{\beta}{(\gbeta{\beta(p_1)}{\negaM{\beta}{\beta(q_1)}})}=
\nonumber
\beta(\nega{(\g{p_1}{\nega{q_1}})})=
\beta(\res{\te}{p_1}{q_1}).
\end{eqnarray}
Therefore, (\ref{betas_implikacios}) is well-defined, too.
Summing up, $\beta$ is a homomorphism from $\mathbf X$ to $\beta(\mathbf X)$.
Therefore, 
%Since the definition of $\beta$ together with the definitions in (\ref{betas_rendezese})-(\ref{betas_nega}) means that $\beta$ commutes with all the operations, 
it readily follows that $\beta(\mathbf X)$ which was defined as the image of $\mathbf X$ under $\beta$ is an odd involutive FL$_e$-chain with involution $\kompM{\beta}$.

%$\beta(\mathbf X_{\mathbf{gr}})$ is prime in $\beta(\mathbf X)$ since (\ref{betas_szorzatos}) is well-defined.
\medskip
\item[(\ref{<_X1})]
In the light of the previous claim, to verify that $\beta(\mathbf X_{\mathbf{gr}})$ is a subalgebra of $\beta(\mathbf X)$ it suffices to note that $X_{gr}$ is a nonempty subuniverse of $X$ by Proposition~\ref{elozetes}.
By Theorem~\ref{csoportLesz}, showing cancellativity of $\beta(\mathbf X_{\mathbf{gr}})$ is equivalent to showing that each of its elements has inverse; but it is straightforward since each element of $\mathbf X_{\mathbf{gr}}$ has inverse and a homomorphism preserves this property.
% To prove the cancellativity of $\beta(\mathbf X_{\mathbf{gr}})$, we need to show that for $v,w,z\in X_{\tau<u}$, $\gbeta{[v]}{[w]}=\gbeta{[v]}{[z]}$ implies $[w]=[z]$. By contradiction assume $[\g{v}{w}]=[\g{v}{z}]$ and $[w]\neq[z]$. We may assume $w<z$ since $X$ is a chain. Since $[w]\neq[z]$, there exists $a$ such that $w<a<z$ and $\tau(a)\geq u$. By monotonicity, $\g{v}{w}\leq\g{v}{a}\leq\g{v}{z}$ holds and hence the convexity of $[\g{v}{w}](=[\g{v}{z}])$ implies $\tau(\g{v}{a})<u$. However, by Lemma~\ref{tau_lemma}, $\tau(\g{v}{a})\geq \tau(a)\geq u$, a contradiction.

\medskip
\item[(\ref{Akozepe})]
Clearly, $[t]$ is nonempty since $t\in[t]$.
%By (\ref{betas_nega}), $\negaM{\beta}{[t]}=[\nega{t}]=[t]$ holds, hence $[t]_{\tau<u}$ is closed under $\kompM{\beta}$. 
By (\ref{betas_szorzatos}), $\gbeta{[t]}{[t]}=[\g{t}{t}]=[t]$ holds, hence $[t]_{\tau<u}$ is closed under $\te$.
By (\ref{betas_implikacios}), $\res{\beta}{[t]}{[t]}=[\res{\te}{t}{t}]=[t]$ holds, hence $[t]_{\tau<u}$ is closed under $\ite{\te}$, too.
For any $z\in]\nega{u},u[$, $\tau(z)=\tau(\max(z,\nega{z}))\leq \max(z,\nega{z})<u$ holds by claims (\ref{FENTszimmetrikus}) and (\ref{item_boundary_zeta}) in Proposition~\ref{zeta_1}.
On the other hand, since $u$ is idempotent, by claims (\ref{FENTszimmetrikus}) and (\ref{ZetaOfIdempotent}) in Proposition~\ref{zeta_1} we obtain $\tau(\nega{u})=\tau(u)=u$, showing $\nega{u},u\not\in[t]$, and in turn, $[t]=]\nega{u},u[$.

\medskip
\item[(\ref{lexikoS})]
For $x\in X_{\tau<u}$, $\tau(x)>t$ would yield $\tau(x)$ be an idempotent by  claim~(\ref{ZetaOfIdempotent}) in Proposition~\ref{zeta_1}, which is strictly between $t$ and $u$, contrary to assumption.
Hence $X_{\tau<u}=X_{\tau=t}$ holds.
Claim~(\ref{nagyonSzigoru}) in Proposition~\ref{elozetes} concludes the proof of $X_{\tau<u}=X_{gr}$.
Therefore, $\mathbf X_{\tau<u}$ is the group part of $\mathbf X$.
By claim~(\ref{Akozepe}), $[t]_{\tau<u}$ is a subuniverse of $\mathbf X$, hence it follows from $\mathbf X_{\tau<u}=\mathbf X_\mathbf{gr}$ that $[t]_{[\tau<u]}$ is a universe of a linearly ordered group, and thus $\mathbf{ker}_\beta$ is a
% by claim~(\ref{nagyonSzigoru}) in Proposition~\ref{elozetes}; the latter forming a 
convex subgroup of $\mathbf X_\mathbf{gr}$.
Lemma~A ends the proof.
\end{proof}

%For a nonempty linearly ordered set $I$ let $\lexH{I}\mathbb R$ denote the Hahn product of $\mathbb R$'s over $I$, which is the set of all functions from $I$ to $\mathbb R$, whose support is either empty or well-ordered, the support of $g\in\lexH{I}\mathbb R$ being $\{\gamma\in I : g(\gamma)\neq 0\}$. Let $\mathbf G$ be a linearly ordered abelian group. For $0\neq g\in \mathbf G$, let $\Gamma(g)$, called the value of $g$ in $\mathbf G$, be the maximal convex subgroup $C$ of $\mathbf G$ such that $g\notin C$ (it exists by Zorn's lemma). Denote by $\Gamma(\mathbf G)$ the set of all values $\Gamma(g)$ of all $g\in\mathbb G\setminus\{0\}$. $\Gamma(\mathbf G)$ can be regarded as a chain ordered by inclusion, call it the spine of $\mathbf G$. Hahn's embedding theorem says that $\mathbf G$, as an ordered abelian group, can be embedded into $\lexH{\Gamma(\mathbf G)}\mathbb R$.

\section{Extremals}\label{SectExtremals}
In this section we define and investigate the extremal elements of the convex components of the group part of $\mathbf X$.
\begin{definition}\rm
Let $( X, \leq, \te, \ite{\te}, t, f )$ be an odd involutive FL$_e$-chain and $u>t$ be idempotent.
Provisionally, for $v\in X_{\tau<u}$ let 
$$
\top_{[v]}=\left\{
\begin{array}{ll}
\bigvee_{z\in[v]}z & \mbox{in case $u\neq t_\uparrow$}\\
v_\uparrow  & \mbox{in case $u=t_\uparrow$}\\
\end{array}
\right. 
\hskip0.1cm
\mbox{and}
\hskip0.2cm
\bot_{[v]}=\left\{
\begin{array}{ll}
\bigwedge_{z\in[v]}z & \mbox{in case $u\neq t_\uparrow$}\\
v_\downarrow  & \mbox{in case $u=t_\uparrow$}\\
\end{array}
\right. 
.
$$
%Assume $u\neq t_\uparrow$. Provisionally, for $v\in X_{\tau<u}$ let $\top_{[v]}:=\bigvee_{z\in[v]}z$ and $\bot_{[v]}:=\bigwedge_{z\in[v]}z$.
Call $\top_{[v]}$ and $\bot_{[v]}$ the top- and the bottom-extremals of the (convex) component of $v$.
Denote
$$
\begin{array}{llll}
X_{\tau\geq u}^{T_c}&=&\{\top_{[v]}\ | \ v\in X_{\tau<u}\}  & \hskip2mm \mbox{(top-extremals of components)},\\
X_{\tau\geq u}^{B_c}&=&\{\bot_{[v]}\ | \ v\in X_{\tau<u}\}  & \hskip2mm \mbox{(bottom-extremals of components)},\\
X_{\tau\geq u}^{E_c}&=&X_{\tau\geq u}^{T_c}\cup X_{\tau\geq u}^{B_c}  & \hskip2mm \mbox{(extremals of components)}.
\end{array}
$$
\end{definition}
The $\tau\geq u$ subscript in the notations refers to the fact that $X_{\tau\geq u}^{T_c}$, $X_{\tau\geq u}^{B_c}$, and $X_{\tau\geq u}^{E_c}$ are subsets of $X_{\tau\geq u}$. We shall shortly prove it in

\begin{proposition}\label{topbotEXISTszorzotabla}%{{\bf (Product table -- Extremals)}}\\
Let ${\mathbf X}=( X, \leq, \te, \ite{\te}, t, f )$ be an odd involutive FL$_e$-chain such that there exists $u$, the smallest strictly positive idempotent element, and let $v\in X_{\tau<u}$.

\begin{enumerate}
\item\label{TopProps}
If $u\neq t_\uparrow$ then $[v]$ is infinite, and if $u=t_\uparrow$ then $[v]=\{v\}$.
If $u\neq t_\uparrow$ then 
$\top_{[v]}$ and $\bot_{[v]}$ exist.
\item[]
It holds true that
%\item\label{Props}
\begin{equation}\label{slhhsdjksjkgb}
\top_{[v]}=\g{v}{u},
\end{equation}
\begin{equation}\label{JHGFJhgkjkHKJH}
\bot_{[v]}=\g{v}{\nega{u}},
\end{equation}
\begin{equation}\label{phvsfkfb}
\bot_{[v]}=\nega{(\top_{[\nega{v}]})} .
\end{equation}
%\begin{equation}\label{hhsdjBOT}
%\bot_{[v]}=\g{v}{\nega{u}} ,
%\end{equation}
%HAHAHA
%$\bot_{[t]}=\g{t}{\nega{u}}$ igaz
%$
%\g{v}{\nega{u}}=
%\g{v}{\bot_{[t]}}=
%\g{v}{\left(\bigwedge_{z\in[t]}z\right)}=
%\g{v}{\nega{\left(\bigvee_{z\in[t]}\nega{z}\right)}}=
%$
%HAHAHA

If $u=t_\uparrow$ then and $v_\downarrow<v<v_\uparrow$.
\item\label{EXTRnincsKOMPONENSBEN}
$\bot_{[v]},\top_{[v]}\in X_{\tau\geq u}$,
$\bot_{[v]},\top_{[v]}\not\in[v]$.

\item\label{ExtrSzorzotabla}
The following product table holds true for $\te$:
\begin{table}[ht]
\small
\begin{center}
\caption{{(Extremals):} If $v,w\in X_{\tau<u}$, $a_\downarrow=a\in X_{\tau\geq u}\setminus X_{\tau\geq u}^{T_c}$, and $y\in X_{\tau\geq u}$ then the following holds true.}
\begin{tabular}{cccccc}
$\te$ & \vline & $\bot_{[w]}$  & $w$ & $\top_{[w]}$ & $y$ \\
\hline
$\bot_{[v]}$  & \vline &  & \color{midgrey}$\bot_{[\g{v}{w}]}$ & $\bot_{[\g{v}{w}]}$ & \\
$v$  & \vline & $\bot_{[\g{v}{w}]}$ & {\color{midgrey}$\g{v}{w}$} & $\top_{[\g{v}{w}]}$ & \\
$\top_{[v]}$ & \vline & $\color{midgrey}\bot_{[\g{v}{w}]}$ & $\color{midgrey}\top_{[\g{v}{w}]}$ & $\top_{[\g{v}{w}]}$  & $\g{v}{y}$\\
$a$ & \vline & $\g{a}{w}$ & {\color{midgrey}$\g{a}{w}$} &  \color{midgrey}$\g{a}{w}$ &\\
\label{Prods1stCIKK}
\end{tabular}
\end{center}
\end{table}
\item\label{HolTopAzBottom}
If $\nega{u}$ is idempotent then 
$X_{\tau\geq u}^{T_c}\cap X_{\tau\geq u}^{B_c}=\emptyset$.

\end{enumerate}
\end{proposition}
\begin{proof}
\item[(\ref{TopProps}):] 
By claim~(\ref{Akozepe}) in Lemma~\ref{faktorok}, $[t]=\, ]\nega{u},u[$ is a universe of a subgroup of $\mathbf X$. It is trivial if $u=t_\uparrow$, hence then $[t]=\{t\}$.
Non-trivial linearly ordered abelian groups are known to 
be infinite and 
having no greatest or least element.
$[t]$ %It 
is non-trivial if $u\neq t_\uparrow$, so $]\nega{u},u[$ %is infinite and 
has no greatest or least element;
this verifies $|[t]|=\infty$ and the existence of $\top_{[t]}$ and $\bot_{[t]}$ along with $\top_{[t]}=u$, $\bot_{[t]}=\nega{u}$. 
\\
Note that $[v]$ is equal to the coset $\g{v}{[t]}$ since $\mathbf X_{\tau<u}$ is a group.
Indeed, by (\ref{betas_szorzatos}) $\g{v}{[t]}\subseteq\gbeta{[v]}{[t]} \subseteq[\g{v}{t}]=[v]$ holds, whereas 
for any $w\in[v]$, by (\ref{betas_szorzatos}) 
$\g{v^{-1}}{w}\in[\g{v^{-1}}{w}]=\gbeta{[v^{-1}]}{[w]}=\gbeta{[v^{-1}]}{[v]}=[\g{v^{-1}}{v}]=[t]$
holds, and hence $w=\g{v}{(\g{v^{-1}}{w})}\in\g{v}{[t]}$.
\\
Since for $z\in[t]$, the mapping $z\mapsto\g{v}{z}$ is a bijection between the two cosets $[t]$ and $[v]$,
it follows that $[v]$ has the same cardinality as $[t]$. 
\\
Assume $u\neq t_\uparrow$. 
Using the infinite distributivity of $\vee$ over $\te$, 
$
\g{v}{u}=
\g{v}{\bigvee_{w\in[t]}w}=
\g{v}{\bigvee_{w\in[t]}\g{\g{v^{-1}}{v}}{w}}=
\g{ \g{v}{v^{-1}} }{ \bigvee_{w\in[t]}{\g{v}{w}} }=
\bigvee_{w\in[t]}{\g{v}{w}}
$ follows.
Since the mapping above is a bijection, it is equal to
$
\bigvee_{\g{v}{w}\in[v]}{\g{v}{w}}=
\bigvee_{z\in[v]}{z}
$,
thus confirming the existence of $\top_{[v]}$ along with (\ref{slhhsdjksjkgb}) in the case of $u\neq t_\uparrow$.
\\
Next we prove, the existence of $\bot_{[v]}$ 
if $u\neq t_\uparrow$.
Because of (\ref{betas_nega}), $\nega{(\top_{[\nega{v}]})}=\nega{(\bigvee_{z\in[\nega{v}]}z)}$
%$z\in\beta(\nega{v})$
%claim~(\ref{FENTszimmetrikus}) in Proposition~\ref{zeta_1}, 
%$z$ runs through $[\nega{v}]$ if and only if $\nega{z}$ runs through $[v]$.
%Therefore, $\nega{(\bigvee_{z\in[\nega{v}]}z)}$ 
is equal to $\nega{(\bigvee_{\nega{z}\in[v]}z)}$.
Therefore $\bigvee_{\nega{z}\in[v]}z$ exists and since $\komp$ is an order reversing involution of $X$, it follows that
$\nega{(\bigvee_{\nega{z}\in[v]}z)}$ is equal to $\bigwedge_{\nega{z}\in[v]}\nega{z}=\bot_{[v]}$.
\\
The previous argument also shows (\ref{phvsfkfb}) if $u\neq t_\uparrow$.
(\ref{phvsfkfb}) readily follows from (\ref{FelNeg_NegLe}) if $u=t_\uparrow$. 
\\
Next we prove (\ref{slhhsdjksjkgb}) and (\ref{JHGFJhgkjkHKJH}) when $u=t_\uparrow$.
By monotonicity of $\te$, $\g{v}{u}\geq\g{v}{t}=v$ holds, but there cannot hold equality, since $\tau(v)<u$, but $\tau(\g{v}{u})\geq\tau(u)=u$ by Lemma~\ref{tau_lemma}.
Hence, $v<\g{v}{u}$ follows. An analogous argument shows $\g{v}{\nega{u}}<v$.
Now, if there would exists $a,b$ such that $\g{v}{\nega{u}}<a<v<b<\g{v}{u}$ then by multiplying it by $\nega{v}$ (the inverse of $v$) %(that is, by $\nega{v}$) 
it would yield $\nega{u}<\g{\nega{v}}{a}<t<\g{\nega{v}}{b}<u$, contradicting $u=t_\uparrow$. Summing up, it holds true that
$\bot_{[v]}=v_\downarrow=\g{v}{\nega{u}}<v$ and $\top_{[v]}=v_\uparrow=\g{v}{u}>v$.
\\
It remains to prove (\ref{JHGFJhgkjkHKJH}) in the $u=t_\uparrow$ case.
By (\ref{phvsfkfb}), (\ref{slhhsdjksjkgb}), and (\ref{eq_feltukrozes})
$
\bot_{[v]}=
\nega{(\top_{[\nega{v}]})}=
\nega{(\g{\nega{v}}{u})}\geq
\g{v}{\nega{u}}
$.
By (\ref{betas_szorzatos}), $\g{v}{\nega{u}}=\g{w}{\nega{u}}$
for any $w>v$, $w\in[v]$ (such a $w$ exists since $[v]$ has no greatest element).
Finally by $(\ref{eq_feltukrozes_CS})$, $\g{w}{\nega{u}}\geq\nega{(\g{\nega{v}}{u})}=\nega{(\top_{[\nega{v}]})}$.
It follows that there is equality everywhere, hence (\ref{JHGFJhgkjkHKJH}) follows.

\item[(\ref{EXTRnincsKOMPONENSBEN}):] 

By (\ref{slhhsdjksjkgb}), Lemma~\ref{tau_lemma}, and claim~(\ref{ZetaOfIdempotent}) in Proposition~\ref{zeta_1}, 
$\tau(\top_{[v]})=\tau(\g{v}{u})=\tau(u)=u$, and (\ref{phvsfkfb}) yields $\tau(\bot_{[v]})=u$, thus confirming $\bot_{[v]},\top_{[v]}\in X_{\tau\geq u}$, and in turn, $\bot_{[v]},\top_{[v]}\not\in[v]$.

\item[(\ref{ExtrSzorzotabla}):] 
Grey items in Table~\ref{Prods1stCIKK} are either straightforward or follow from the black items thereof  by commutativity.
Therefore, we prove only the black items.
\begin{itemize}
\item
Using (\ref{slhhsdjksjkgb}), 
$\g{v}{\top_{[w]}}=\g{v}{\g{w}{u}}=\top_{[\g{v}{w}]}$ confirms Table~\ref{Prods1stCIKK}$_{(2,3)}$.
\item
A similar proof using the idempotency of $u$ ensures Table~\ref{Prods1stCIKK}$_{(3,3)}$.
\item
By (\ref{slhhsdjksjkgb}), $\g{\top_{[v]}}{y}=\g{(\g{v}{u})}{y}=\g{v}{(\g{u}{y})}=\g{v}{y}$, where the latest equality is implied by $y\in X_{\tau\geq u}$; this proves Table~\ref{Prods1stCIKK}$_{(3,4)}$.
\item
By (\ref{JHGFJhgkjkHKJH}), $\g{v}{\bot_{[w]}}=\g{v}{(\g{w}{\nega{u}})}=\g{(\g{v}{w})}{\nega{u}}=\bot_{[\g{v}{w}]}$
holds as stated in Table~\ref{Prods1stCIKK}$_{(2,1)}$.

\item
Next,
by (\ref{slhhsdjksjkgb}) and (\ref{JHGFJhgkjkHKJH})
$
\g{\bot_{[v]}}{\top_{[w]}}=\g{(\g{v}{u})}{(\g{w}{\nega{u}})}=\g{(\g{v}{w})}{(\g{u}{\nega{u}})}=\g{(\g{v}{w})}{\nega{u}}=\bot_{[\g{v}{w}]}
$,
where $\g{u}{\nega{u}}=\nega{u}$ follows from 
$\g{u}{\nega{u}}=\nega{(\res{\te}{u}{u})}=\nega{\tau(u)}=\nega{u}$, so we are done with the proof of Table~\ref{Prods1stCIKK}$_{(1,3)}$.

\item
There exists $Z_2\subset X\setminus \{\nega{a}\}$ such that $\nega{a}=\bigwedge Z_2$, since 
$\nega{a}_\uparrow=\nega{a}$.
We may safely assume $Z_2\subset X_{\tau\geq u}\setminus \{\nega{a}\}$ since $\nega{a}$ is not the greatest lower bound of any component of $X_{\tau<u}$.
Therefore, by (\ref{eq_feltukrozes_CS}), Table~\ref{Prods1stCIKK}$_{(3,4)}$, and (\ref{eq_feltukrozes}), respectively, for $z\in Z_2$ it holds true that 
$\g{a}{\bot_{[w]}}\geq\nega{(\g{z}{\top_{[\nega{w}]}})}=\nega{(\g{z}{\nega{w}})}\geq\g{\nega{z}}{w}$.
Hence, using that $\te$ is residuated, $\g{a}{w}=\g{(\bigvee_{z\in Z_2}\nega{z})}{w}=\bigvee_{z\in Z_2}(\g{\nega{z}}{w})\leq\g{a}{\bot_{[w]}}$ follows.
On the other hand, by the isotonicity of $\te$, $\g{a}{\bot_{[w]}}\leq\g{a}{w}$ holds, too.
This proves Table~\ref{Prods1stCIKK}$_{(4,1)}$.

\item[(\ref{HolTopAzBottom}):]
%{\em Claim:}
First we claim that
if $\nega{u}$ is idempotent then 
$X_{\tau\geq u}^{T_c}$
and
$X_{\tau\geq u}^{B_c}$ are either disjoint or coincide. 
%Either $X_{\tau\geq u}^{T_c}\cap X_{\tau\geq u}^{B_c}\neq\emptyset$ or $X_{\tau\geq u}^{T_c}=X_{\tau\geq u}^{B_c}$.
Indeed, 
if there exists $c,d\in X_{\tau<u}$ such that $\top_{[c]}=\bot_{[d]}$ then 
$\g{c}{\nega{c}}=t$ follows from claim~(\ref{nagyonSzigoru}) in Proposition~\ref{elozetes}, and 
by Table~\ref{Prods1stCIKK}$_{(2,3)}$ and Table~\ref{Prods1stCIKK}$_{(2,1)}$, 
for any $v\in X_{\tau<u}$, 
$\top_{[v]}=\top_{[\g{c}{\g{\nega{c}}{v}}]}=\g{\top_{[c]}}{\g{\nega{c}}{v}}=\g{\bot_{[d]}}{\g{\nega{c}}{v}}=\bot_{[\g{d}{\g{\nega{c}}{v}}]}$ 
and
$\bot_{[v]}=\bot_{[\g{d}{\g{{\nega{d}}}{v}}]}=\g{\g{\bot_{[d]}}{\nega{d}}}{v}=\g{\g{\top_{[c]}}{\nega{d}}}{v}=\top_{[\g{\g{c}{\nega{d}}}{v}]}$ hold, thus the proof of the claim is concluded.
\\
Therefore, it suffices to prove that $\bot_{[t]}\notin X_{\tau\geq u}^{T_c}$.
If there would exist $v\in X_{\tau<u}$ such that $\bot_{[t]}=\top_{[v]}$ then by Table~\ref{Prods1stCIKK}$_{(1,3)}$, $\g{\top_{[v]}}{\bot_{[t]}}=\bot_{[v]}$ would follow, a contradiction to the idempotency of $\bot_{[t]}$.
\end{itemize}
\end{proof}

\section{Gaps -- Motto: \lq\lq Not all gaps created equal\rq\rq}\label{SectGaps}

%We are ready to prove the \lq\lq one-step decomposition\rq\rq\ lemmas (c.f. Sections~\ref{OneDecIdemp} and \ref{OneStepNotIdemp}) for densely ordered odd involutive FL$_e$-chains. However, since in the Table~\ref{Prods1stCIKK}$_{(4,1)}$ entry a kind of density condition must have been assumed for $a$, further investigations are needed to treat the case when the order density of $\mathbf X$ is not assumed. Therefore we classify gaps of $X_{\tau\geq u}$. It turns out that some gaps behave like extremal elements of \lq\lq missing\rq\rq\ components of $\mathbf X$, those will de called pseudo extremals.

In order to treat the case when the order density of $\mathbf X$ is not assumed (c.f. Table~\ref{Prods1stCIKK}$_{(4,1)}$, where a kind of density condition must have been assumed for $a$), first we classify gaps in $X_{\tau\geq u}$.
It turns out that some gaps behave like extremal elements of \lq\lq missing\rq\rq\ components of $\mathbf X$, those will be called pseudo extremals. %, while other gaps are just ordinary ones.

\begin{definition}\label{sokmindenki}
\rm
Let ${\mathbf X}=( X, \leq, \te, \ite{\te}, t, f )$ be an odd involutive FL$_e$-chain with residual complement $\komp$, such that there exists $u$, the smallest strictly positive idempotent. % and $u\neq t_\uparrow$.
Let
$$
\begin{array}{llll}
%X_{\tau\geq u}^d&=&\{ x\in X_{\tau\geq u} \ | \ x_\downarrow=x \},\\
X_{\tau\geq u}^{Gap}&=&\{ x\in X_{\tau\geq u} \ | \ x_\downarrow<x \} & \hskip2mm \mbox{(upper elements of gaps in $X_{\tau\geq u}$)}\\
X_{\tau\geq u}^{G_2}&=&\{ x\in X_{\tau\geq u}^{Gap} \ | \ \g{x}{\nega{u}}=x \} & \hskip2mm \mbox{}\\
X_{\tau\geq u}^{T_{ps}}&=&\{ x\in X_{\tau\geq u}^{Gap} \ | \ \g{x}{\nega{u}}=x_\downarrow \} & \hskip2mm \mbox{(pseudo top-extremals)}\\
X_{\tau\geq u}^{B_{ps}}&=&\{ x_\downarrow\in X \ | \ x\in X_{\tau\geq u}^{T_{ps}} \} & \hskip2mm \mbox{(pseudo bottom-extremals)}\\
%X_{\tau\geq u}^{B_{ps}}&=& \g{X_{\tau\geq u}^{T_{ps}}}{\nega{u}},\\
X_{\tau\geq u}^{E_{ps}}&=& X_{\tau\geq u}^{T_{ps}}\cup X_{\tau\geq u}^{B_{ps}} & \hskip2mm \mbox{(pseudo extremals)}\\
%X_{\tau\geq u}^{T_{ps}}&=&\{ x\in X_{\tau\geq u} \ | \ x_\downarrow<x, \g{x}{\nega{u}}=x_\downarrow \},\\
%X_{\tau\geq u}^{G_2}&=&\{ x\in X_{\tau\geq u} \ | \ x_\downarrow<x, \g{x}{\nega{u}}=x \},
X_{\tau\geq u}^T&=&X_{\tau\geq u}^{T_c}\cup X_{\tau\geq u}^{T_{ps}} & \hskip2mm \mbox{(top-extremals)}\\
X_{\tau\geq u}^B&=&X_{\tau\geq u}^{B_c}\cup X_{\tau\geq u}^{B_{ps}} & \hskip2mm \mbox{(bottom-extremals)}\\
X_{\tau\geq u}^E&=&X_{\tau\geq u}^T\cup X_{\tau\geq u}^B & \hskip2mm \mbox{(extremals)}\\
%X_{\tau\geq u}^{G_{ord}}&=&X_{\tau\geq u}^{G_2}\setminus X_{\tau\geq u}^{B_{ps}} & \hskip2mm \mbox{()}\\
\end{array}
$$
For $x\in X_{\tau\geq u}$ denote $x_\Downarrow$ the predecessor of $x$ inside $X_{\tau\geq u}$, that is let
$$
x_\Downarrow=\left\{
\begin{array}{ll}
z & \mbox{if there exists $X_{\tau\geq u}\ni z<x$ such that there is}\\ & \mbox{no element in $X_{\tau\geq u}$ strictly between $z$ and $x$,}\\
x & \mbox{if for any $X_{\tau\geq u}\ni z<x$ there exists $v\in X_{\tau\geq u}$}\\
& \mbox{such that $z<v<x$ holds.}\\
\end{array}
\right. 
$$
Clearly,
$$
x_\Downarrow=\left\{
\begin{array}{ll}
x_\downarrow & \mbox{if $x\in X_{\tau\geq u}\setminus X_{\tau\geq u}^{T_c}$}\\
\bot_{[v]} & \mbox{if $x=\top_{[v]}\in X_{\tau\geq u}^{T_c}$}\\
\end{array}
\right. .
$$
Define $x_\Uparrow$ dually.
\end{definition}

\begin{proposition}\label{khkfgjkdfjkdjkddjss}
%{\bf (Product table -- Gaps)}
%{\bf (General gap treatment)}
Let ${\mathbf X}=( X, \leq, \te, \ite{\te}, t, f )$ be an odd involutive FL$_e$-chain with residual complement $\komp$, such that there exists $u$, the smallest strictly positive idempotent. % and $u\neq t_\uparrow$.
\begin{enumerate}
\item\label{NoTopGap1lehetseges}
$X_{\tau\geq u}^{B_{ps}}\subseteq X_{\tau\geq u}$,
$X_{\tau\geq u}^{T_{ps}}\cap X_{\tau\geq u}^{T_c}=\emptyset$,
$X_{\tau\geq u}^{B_{ps}}\cap X_{\tau\geq u}^{B_c}=\emptyset$,
and
$X_{\tau\geq u}^{G_2}\cap X_{\tau\geq u}^T=\emptyset$.
\item\label{dfbskdbfhfb}
$X_{\tau\geq u}^{Gap}\setminus X_{\tau\geq u}^{T_c}=X_{\tau\geq u}^{T_{ps}}\cup X_{\tau\geq u}^{G_2}$.
\item\label{skfhshsjkshfjj12}
For $v\in X_{\tau<u}$ and $x\in X_{\tau\geq u}$, 
\begin{equation}\label{dgskjdfhsasfGAPOS}
\g{x}{\bot_{[v]}}=
\left\{
\begin{array}{ll}
(\g{x}{v})_\Downarrow(<\g{x}{v}\in X_{\tau\geq u}^{T_c}) & \mbox{if $x\in X_{\tau\geq u}^{T_c}$}\\
(\g{x}{v})_\downarrow(<\g{x}{v}\in X_{\tau\geq u}^{T_{ps}}) & \mbox{if $x\in X_{\tau\geq u}^{T_{ps}}$}\\
\ \g{x}{v} & \mbox{if $x\in X_{\tau\geq u}\setminus X_{\tau\geq u}^T$}\\
\end{array}
\right. .
\end{equation}
\item\label{Ranegalok}
If $x\in X_{\tau\geq u}^{T_{ps}}$ then $\nega{x_\downarrow}\in X_{\tau\geq u}^{T_{ps}}$, too.
\item\label{csekken}
If $x,y\in X_{\tau\geq u}^{T_{ps}}$ then $\g{(\g{x}{y})}{\nega{u}}\neq\g{x}{y}$.
\item\label{G1tau=u}
For $x\in X_{\tau\geq u}^{T_{ps}}$, $\tau(x)=u$. 
\item\label{ProdGapTable}
The following product table holds for $\te$: % (see it in page \pageref{ProdsGeneralCase}):
\begin{table}[h]
\tiny
\centering
\caption{{(Gaps):} If $v,w\in X_{\tau<u}$, $x,y\in X_{\tau\geq u}^{T_{ps}}$, $z,s\in X_{\tau\geq u}\setminus X_{\tau\geq u}^T$ then the following holds true.}
\label{ProdsGeneralCase}
\begin{tabular}{cccccccc}
$\te$ & \vline & $\bot_{[w]}$  & $w$ & $\top_{[w]}$  & $s$ & $y_\downarrow$ & $y$\\
\hline
$\bot_{[v]}$  & \vline &  & {\color{midgrey}$\bot_{[\g{v}{w}]}$} & {\color{midgrey}$\bot_{[\g{v}{w}]}$}  & {\color{midgrey}$\g{v}{s}$} &  & $(\g{v}{y})_\downarrow(<\g{v}{y}\in X_{\tau\geq u}^{T_{ps}})$\\
$v$  & \vline & {\color{midgrey}$\bot_{[\g{v}{w}]}$} & {\color{midgrey}$\g{v}{w}$} & {\color{midgrey}$\top_{[\g{v}{w}]}$} & $\in X_{\tau\geq u}\setminus X_{\tau\geq u}^T$  &  & ${\color{midgrey}\g{v}{y}}\in X_{\tau\geq u}^{T_{ps}}$ \\
$\top_{[v]}$ & \vline & {\color{midgrey}$\bot_{[\g{v}{w}]}$} & {\color{midgrey}$\top_{[\g{v}{w}]}$} & {\color{midgrey}$\top_{[\g{v}{w}]}$} & $\in X_{\tau\geq u}\setminus X_{\tau\geq u}^T$ &  & {\color{midgrey}$\g{v}{y}$}\\
$z$ & \vline & $\g{z}{w}$ & {\color{midgrey}$\g{z}{w}$} & {\color{midgrey}$\g{z}{w}$} & $\in X_{\tau\geq u}\setminus X_{\tau\geq u}^T$ & $\g{z}{y}$ & {\color{midgrey}$\g{z}{y}$}\\
$x_\downarrow$ & \vline &  &  &  & {\color{midgrey}$\g{x}{s}$} &  & \\
$x$ & \vline & \color{midgrey}$(\g{x}{w})_\downarrow$ & ${\color{midgrey}\g{x}{w}\in X_{\tau\geq u}^{T_{ps}}}$ & {\color{midgrey}$\g{x}{w}$} & $\in X_{\tau\geq u}\setminus X_{\tau\geq u}^T$ &  & \\
\end{tabular}
\end{table}
\end{enumerate}
\end{proposition}
\begin{proof}
\item[(\ref{NoTopGap1lehetseges}):]
By Lemma~\ref{tau_lemma} and claims (\ref{FENTszimmetrikus}) and (\ref{lemasoljaGENzeta}) in Proposition~\ref{zeta_1},
for $x\in X_{\tau\geq u}^{T_{ps}}$, $\tau(x_\downarrow)=\tau(\g{x}{\nega{u}})\geq\tau(\nega{u})=u$ holds, ensuring $X_{\tau\geq u}^{B_{ps}}\subseteq X_{\tau\geq u}$.
Second, if $X_{\tau\geq u}^{T_{ps}}\ni x=\top_{[v]}\in X_{\tau\geq u}^{T_c}$ then 
$x>x_\downarrow\in X_{\tau\geq u}$ holds by the previous item, hence $\top_{[v]}>(\top_{[v]})_\downarrow$.
It follows that $(\top_{[v]})_\downarrow$ is in $X_{\tau<u}$ since $\top_{[v]}$ is a top-extremal of a convex component in $X_{\tau<u}$, a contradiction to $x_\downarrow\in X_{\tau\geq u}$.
A completely analogous argument proves $X_{\tau\geq u}^{B_{ps}}\cap X_{\tau\geq u}^{B_c}=\emptyset$:
If $x\in X_{\tau\geq u}^{B_{ps}}$ then $x<x_\uparrow\in X_{\tau\geq u}^{T_{ps}}\subseteq X_{\tau\geq u}$. 
But if $x\in X_{\tau\geq u}^{B_c}$ then $x<x_\uparrow\in X_{\tau<u}$, a contradiction.
Finally, obviously, $X_{\tau\geq u}^{G_2}\cap X_{\tau\geq u}^{T_{ps}}=\emptyset$, hence it suffices to prove
$X_{\tau\geq u}^{G_2}\cap X_{\tau\geq u}^{T_c}=\emptyset$.
If $x\in X_{\tau\geq u}^{G_2}$ then $\g{x}{\nega{u}}=x$, whereas if $x\in X_{\tau\geq u}^{T_c}$ then by Table~\ref{Prods1stCIKK}$_{(3,1)}$, $\g{x}{\nega{u}}=\bot_{[v]}$ holds. But $\bot_{[v]}\neq x$.

\item[(\ref{dfbskdbfhfb}):]
First we claim $\g{x}{\bot_{[v]}}\geq\g{x_\downarrow}{v}$ for $x\in X_{\tau\geq u}^{Gap}\setminus X_{\tau\geq u}^{T_c}$.
Indeed, $\g{x}{\bot_{[v]}}\geq\nega{(\g{\nega{x_\downarrow}}{\top_{[\nega{v}]}})}$ follows from (\ref{eq_feltukrozes_CS}) since $x>x_\downarrow$.
Now $x\neq\top_{[v]}$, and it implies $x_\downarrow\in X_{\tau\geq u}$, that is, $\nega{x_\downarrow}\in X_{\tau\geq u}$ by Lemma~\ref{tau_lemma}, and hence $\nega{(\g{\nega{x_\downarrow}}{\top_{[\nega{v}]}})}=\nega{(\g{\nega{x_\downarrow}}{\nega{v}})}$
follows by Table~\ref{Prods1stCIKK}$_{(3,4)}$.
By (\ref{eq_feltukrozes}), $\nega{(\g{\nega{x_\downarrow}}{\nega{v}})}\geq\g{x_\downarrow}{v}$, and we are done.
Setting $v=t$ in the claim proves claim~(\ref{dfbskdbfhfb}).

\item[(\ref{skfhshsjkshfjj12}):]
For $x\in X_{\tau\geq u}$, either $x\in X_{\tau\geq u}^{T_c}$ holds, or $x\notin X_{\tau\geq u}^{T_c}$ and $x_\downarrow=x$ holds, or $x\notin X_{\tau\geq u}^{T_c}$ and $x_\downarrow<x$ holds.
%0
In the first case, by Table~\ref{Prods1stCIKK}$_{(1,3)}$ and Table~\ref{Prods1stCIKK}$_{(2,3)}$, we obtain $\g{\top_{[w]}}{\bot_{[v]}}=\bot_{[\g{w}{v}]}=(\top_{[\g{w}{v}]})_\Downarrow=(\g{\top_{[w]}}{v})_\Downarrow$, as required.
In the second case, $\g{x}{\bot_{[v]}}=\g{x}{v}$ holds by Table~\ref{Prods1stCIKK}$_{(4,1)}$, as required in the third row of (\ref{dgskjdfhsasfGAPOS}).

In the third case, either $x\in X_{\tau\geq u}^{T_{ps}}$ or $x\in X_{\tau\geq u}^{G_2}$ holds by item (\ref{dfbskdbfhfb}).
Hence, since $X_{\tau\geq u}^{G_2}\cap X_{\tau\geq u}^T=\emptyset$ by claim (\ref{NoTopGap1lehetseges}), 
to prove (\ref{dgskjdfhsasfGAPOS}) it remains to prove
$$\g{x}{\bot_{[v]}}=
\left\{
\begin{array}{ll}
(\g{x}{v})_\downarrow(<\g{x}{v}\in X_{\tau\geq u}^{T_{ps}})
 & \mbox{if $x\in X_{\tau\geq u}^{T_{ps}}$}\\
\g{x}{v} & \mbox{if $x\in X_{\tau\geq u}^{G_2}$}\\
\end{array}
\right. .
$$
By (\ref{JHGFJhgkjkHKJH}), $\g{x}{\bot_{[v]}}=\g{x}{\g{v}{\nega{u}}}$ holds.
If $x\in X_{\tau\geq u}^{G_2}$ then $\g{x}{\g{v}{\nega{u}}}$ is equal to $\g{x}{v}$ and thus the proof of the second row is concluded.
If $x\in X_{\tau\geq u}^{T_{ps}}$ then $\g{x}{\g{v}{\nega{u}}}$ cannot be equal to $\g{x}{v}$ since -- referring to claim~(\ref{nagyonSzigoru}) in Proposition~\ref{elozetes} -- by cancelling $v$ we would obtain $x_\downarrow=\g{x}{\nega{u}}=x$, a contradiction.
Therefore, $\g{(\g{x}{v})}{\nega{u}}<\g{x}{v}$ holds.
Finally, the assumption that there exists $c\in X$ such that $\g{(\g{x}{v})}{\nega{u}}<c<\g{x}{v}$ would lead, again cancelling by $v$, to $x_\downarrow=\g{x}{\nega{u}}<\g{c}{\nega{v}}<x$, a contradiction.
Hence, $\g{(\g{x}{v})}{\nega{u}}=(\g{x}{v})_\downarrow<\g{x}{v}$.
The proof of (\ref{dgskjdfhsasfGAPOS}) is concluded.

\item[(\ref{Ranegalok}):]
Let $x\in X_{\tau\geq u}^{T_{ps}}$.
Then $x_\downarrow\in X_{\tau\geq u}^{B_{ps}}\subseteq X_{\tau\geq u}$ and $x_\downarrow\notin X_{\tau\geq u}^{B_c}$ follow by claim~(\ref{NoTopGap1lehetseges}).
Hence $\nega{x_\downarrow}\notin X_{\tau\geq u}^{T_c}$ follows from (\ref{phvsfkfb}).
By Lemma~\ref{tau_lemma}, $\nega{x_\downarrow}\in X_{\tau\geq u}$ follows from $x_\downarrow\in X_{\tau\geq u}$,
thus $\nega{x_\downarrow}\in X_{\tau\geq u}\setminus X_{\tau\geq u}^{T_c}$ holds.
Hence by (\ref{dgskjdfhsasfGAPOS}), 
$\g{\nega{u}}{\nega{x_\downarrow}}\geq(\nega{x_\downarrow})_\downarrow$ holds true.
On the other hand, $\g{x}{\nega{u}}=x_\downarrow$ implies $\res{\te}{\nega{u}}{x_\downarrow}\geq x$ and hence by (\ref{eq_quasi_inverse}), $x_\downarrow<x$ and (\ref{FelNeg_NegLe}), $\g{\nega{u}}{\nega{x_\downarrow}}\leq\nega{x}=(\nega{x_\downarrow})_\downarrow<\nega{x_\downarrow}$ holds.
Summing up, we obtained $\g{\nega{x_\downarrow}}{\nega{u}}=(\nega{x_\downarrow})_\downarrow$ and $(\nega{x_\downarrow})_\downarrow<\nega{x_\downarrow}$, as required.

\item[(\ref{csekken}):]
From the opposite of claim~(\ref{csekken}), $\g{x}{y}=\g{(\g{(\g{x}{y})}{\nega{u}})}{\nega{u}}=\g{(\g{x}{\nega{u}})}{(\g{y}{\nega{u}})}=\g{x_\downarrow}{y_\downarrow}$ would follow, contradicting Proposition~\ref{diagonalSZIGORU}. 
\item[(\ref{G1tau=u}):]
$\tau(x)\geq u$ is straightforward.
It holds true that $\g{x}{(\g{\tau(x)}{\nega{u}})}=\g{(\g{x}{\tau(x)})}{\nega{u}}=\g{x}{\nega{u}}=x_\downarrow<x$.
Hence, by monotonicity of $\te$, $\g{\tau(x)}{\nega{u}}<t$ follows, which implies $\tau(x)\leq \nega{(\nega{u})}=u$ by residuation. 
\item[(\ref{ProdGapTable}):]
Grey items in Table~\ref{ProdsGeneralCase} are either straightforward, or 
can readily be seen by the commutativity of $\te$,
or are inherited from Table~\ref{Prods1stCIKK}, so it suffices to prove the items in black.
\begin{itemize}
\item
%Let $v,w\in X_{\tau<u}$, $x,y\in X_{\tau\geq u}^{T_{ps}}$, $z\in X_{\tau\geq u}\setminus X_{\tau\geq u}^T$.
Table~\ref{ProdsGeneralCase}$_{(4,1)}$ and Table~\ref{ProdsGeneralCase}$_{(1,6)}$ follow from (\ref{dgskjdfhsasfGAPOS}).
\item
As for Table~\ref{ProdsGeneralCase}$_{(2,4)}$, Table~\ref{ProdsGeneralCase}$_{(3,4)}$, 
Table~\ref{ProdsGeneralCase}$_{(4,4)}$, Table~\ref{ProdsGeneralCase}$_{(6,4)}$, by Lemma~\ref{tau_lemma}, the products
$\g{v}{s}$,  
$\g{\top_{[v]}}{s}$,  
$\g{z}{s}$,  
$\g{x}{s}$,  respectively
are in $X_{\tau\geq u}$. Assume by contradiction that any of these products is in $X_{\tau\geq u}^T$, that is, 
$\g{?}{s}\in X_{\tau\geq u}^T$ for $?\in\{v,\top_{[v]},z,x\}$. 
Then, by (\ref{dgskjdfhsasfGAPOS}), $\g{\nega{u}}{(\g{?}{s})}=(\g{?}{s})_\Downarrow<\g{?}{s}$ holds, whereas $\g{?}{(\g{\nega{u}}{s})}=\g{?}{s}$ holds by Table~\ref{ProdsGeneralCase}$_{(1,4)}$, a contradiction.
\item
As for Table~\ref{ProdsGeneralCase}$_{(2,6)}$, we need to prove that 
$(\g{v}{y})_\downarrow<\g{v}{y}$ and $\g{(\g{v}{y})}{\nega{u}}=(\g{v}{y})_\downarrow$.
Indeed, $\g{(\g{v}{y})}{\nega{u}}=\g{v}{y}$ and cancellation by $v$ (see claim~(\ref{nagyonSzigoru}) in Proposition~\ref{elozetes}) would yield $y_\downarrow=\g{y}{\nega{u}}=y$, a contradiction, so $\g{(\g{v}{y})}{\nega{u}}<\g{v}{y}$ follows.
If there were $c$ such that $\g{(\g{v}{y})}{\nega{u}}<c<\g{v}{y}$ then a cancellation by $v$ would yield $y_\downarrow=\g{y}{\nega{u}}<\g{\nega{v}}{c}<y$, a contradiction.
Summing up, $\g{(\g{v}{y})}{\nega{u}}=(\g{v}{y})_\downarrow<\g{v}{y}$, and hence the proof of Table~\ref{ProdsGeneralCase}$_{(2,6)}$ is concluded.
\item
Finally, $\g{z}{y_\downarrow}=\g{z}{(\g{y}{\nega{u}})}=\g{(\g{z}{\nega{u}})}{y}$, and by (\ref{dgskjdfhsasfGAPOS}) it is equal to $\g{z}{y}$ as stated in Table~\ref{ProdsGeneralCase}$_{(4,5)}$.
\end{itemize}
\end{proof}

\section{One-step decomposition  -- when $\nega{u}$ is idempotent}\label{OneDecIdemp}

%Collapsing convex components of the group-part of $\mathbf X$ together with their extremals, and collapsing pairs of coherent pseudo extremals into singletons -- the homomorphism $\gamma$ -- when $\nega{u}$ is idempotent
\begin{definition}\label{GammaCollapsusDefi}
\rm
Let ${\mathbf X}=( X, \leq, \te, \ite{\te}, t, f )$ be an odd involutive FL$_e$-chain with residual complement $\komp$, such that there exists $u$, the smallest strictly positive idempotent. % and $u\neq t_\uparrow$.
Assume $\nega{u}$ is idempotent.
Let
\begin{eqnarray*}
%X_{\tau\geq u}^G&=&\{x,x_\downarrow\ | \ x\in X_{\tau\geq u}^{T_{ps}}\},\\
%X_{\tau\geq u}^{E_c}&=&\{\top_{[v]}, \bot_{[v]}\ | \ v\in X_{\tau<u}\},\\
X_{[\tau\geq u]}^{E_{ps}}&=&\{\{\{p\},\{p_\downarrow\}\}\ | \ p\in X_{\tau\geq u}^{T_{ps}}\},\\
X_{[\tau\geq u]}^{E_c}&=&\{\{\{\top_{[v]}\}, \{\bot_{[v]}\}\}\ | \ v\in X_{\tau<u}\},\\
X_{[\tau\geq u]}^E & = & X_{[\tau\geq u]}^{E_c}\cup X_{[\tau\geq u]}^{E_{ps}}.
%X_{[\tau\geq u]} & = & X_{[\tau\geq u]}^E \cup X_{\tau\geq u}\setminus \{x \ | \ x\in q\in X_{[\tau\geq u]}^E\}
\end{eqnarray*}
Let $\gamma$ be defined on $\beta(X)$ by
$\gamma(x)=$
%$$
%[x]=
%\left\{
%\begin{array}{ll}
%\{\top_{[v]}, \bot_{[v]}\} & \mbox{if $x=\top_{[v]}$ or $x=\bot_{[v]}$}\\
%\{x,x_\downarrow\} & \mbox{if $x\in Gap_1$}\\
%\{x,x_\uparrow\} & \mbox{if $x_\uparrow\in Gap_1$}\\
%\{x\} & \mbox{otherwise}\\
%\end{array}
%\right. ,
%$$
{
\begin{equation}\label{GammaDefiiii}
\begin{array}{l}
\left\{
\begin{array}{cl}
\{{[v]},\{\top_{[v]}\}, \{\bot_{[v]}\}\} & \mbox{if $x\in\{[v],\{\top_{[v]}\},\{\bot_{[v]}\}\}$ for some $v\in X_{\tau<u}$}\\
\{\{p\},\{p_\downarrow\}\} & \mbox{if $x\in \{\{p\},\{p_\downarrow\}\}\in X_{[\tau\geq u]}^{E_{ps}}$}\\
\{x\} & \mbox{if $x=\{p\}$, $p\in X_{\tau\geq u}\setminus X_{\tau\geq u}^E$}\\
%\{\{p\}\} & \mbox{if $x=\{p\}$, $p\in X_{\tau\geq u}\setminus X_{\tau\geq u}^E$}
\end{array}
\right. ,\\
%X_{[\tau\geq u]}&=&\{\ [x] \ | \ x\in X_{\tau\geq u} \} .
\end{array}
\end{equation}
}
%and let
%$$
%X_{[u]}=\{[x] \ | \ x\in X\}.
%$$
%Let $$\varphi:X\to X_{[u]}$$ be given by $\varphi(x)=[x]$,
%Let $\varphi:X_{\tau\geq u}\to X_{[\tau\geq u]}$ be given by $\varphi(x)=[x]$,
By letting $x\in a, y\in b$ for $a,b\in\gamma(\beta(X))=\{\gamma(z) : z\in \beta(X)\}$, over $\gamma(\beta(X))$ define
\begin{eqnarray}
& a\leq_\gamma b & \mbox{iff }  x\leq_\beta y, \label{gammas_rendezese}\\
& \ggamma{a}{b} & =  \gamma(\gbeta{x}{y}), \label{gammas_szorzatos}\\
& \res{\gamma}{a}{b} & =  \gamma(\res{\tebeta}{x}{y}), \label{gammas_implikacios}\\
& \negaM{\gamma}{a} & =  \gamma(\negaM{\beta}{x}) \label{gammas_nega}.
\end{eqnarray}
\end{definition}

\begin{proposition}\label{Overlapping}%{{\bf (Overlapping of the convex components%$[x]$'s)}}
$\gamma$, defined in Definition~\ref{GammaCollapsusDefi}, is well-defined.
\end{proposition}
\begin{proof}
%To see that exactly one row of (\ref{GammaDefiiii}) defines $\gamma(x)$, ?
We need to prove the following statement: If $a,b\in X_{[\tau\geq u]}^E$, $a\neq b$ then $a\cap b=\emptyset$.
Indeed, the case $a,b\in X_{[\tau\geq u]}^{E_c}$, $a\neq b$ follows from claim~\ref{HolTopAzBottom} in Proposition~\ref{topbotEXISTszorzotabla}.
Next, if $p\in X_{\tau\geq u}^{T_{ps}}$ then $p_\downarrow\in X_{\tau\geq u}^{G_2}$, since 
$\g{p_\downarrow}{\nega{u}}=\g{(\g{p}{\nega{u}})}{\nega{u}}=\g{p}{(\g{\nega{u}}{\nega{u}})}=\g{p}{\nega{u}}=p_\downarrow$.
Therefore, the case $a,b\in X_{[\tau\geq u]}^{E_{ps}}$, $a\neq b$ follows, since $X_{\tau\geq u}^{T_{ps}}$ and $X_{\tau\geq u}^{G_2}$ are disjoint by claim~(\ref{NoTopGap1lehetseges}) in Proposition~\ref{khkfgjkdfjkdjkddjss}. 
Hence it remains to prove $\{\{\top_{[v]}\},\{\bot_{[v]}\}\}\cap\{\{p\},\{p_\downarrow\}\}=\emptyset$ for $v\in X_{\tau <u}$ and $p\in X_{\tau\geq u}^{T_{ps}}$.
The impossibility of 
$p=\top_{[v]}$ or $p_\downarrow=\bot_{[v]}$ is clear from $|[v]|\neq\emptyset$. 
As for $p\neq\bot_{[v]}$, by contradiction, $\g{\bot_{[v]}}{\nega{u}}={{(\bot_{[v]})}}_\downarrow$ would lead, by %(\ref{hgdifd}) 
Table~\ref{Prods1stCIKK}$_{(2,1)}$,
to $\bot_{[v]}=\g{v}{\bot_{[t]}}=\g{v}{\nega{u}}=\g{v}{(\g{\nega{u}}{\nega{u}})}=\g{(\g{v}{\nega{u}})}{\nega{u}}=\g{\bot_{[v]}}{\nega{u}}={{(\bot_{[v]})}}_\downarrow$, a contradiction to $p>p_\downarrow$.
Finally, $p_\downarrow=\top_{[v]}$ would imply 
$\top_{[v]}=p_\downarrow=\g{p}{\nega{u}}=\g{p}{(\g{\nega{u}}{\nega{u}})}=\g{(\g{p}{\nega{u}})}{\nega{u}}=\g{p_\downarrow}{\nega{u}}=\g{\top_{[v]}}{\nega{u}}=\bot_{[v]}$ by Table~\ref{Prods1stCIKK}$_{(3,1)}$, which is a contradiction.
\end{proof}

\begin{proposition}\label{Kalap}
%{\bf (Product Table  -- idempotent $\nega{u}$ case)}
Let ${\mathbf X}=( X, \leq, \te, \ite{\te}, t, f )$ be an odd involutive FL$_e$-chain with residual complement $\komp$, such that there exists $u$, the smallest strictly positive idempotent element. % and $u\neq t_\uparrow$. 
Assume $\nega{u}$ is idempotent.
%\begin{enumerate}
%\item\label{dkajskfhl}
The following product table holds for $\te$ (in the bottom-right $2\times2$ entries 
the formula on the left applies if $\g{x}{y}>(\g{x}{y})_\downarrow\in X_{\tau\geq u}$, whereas the formula on the right applies if
$\g{x}{y}>(\g{x}{y})_\downarrow\in X_{\tau<u}$
or
$\g{x}{y}=(\g{x}{y})_\downarrow$.
\begin{table}[h]
\tiny
\centering
\caption{{(When $\nega{u}$ is idempotent):} If $v,w\in X_{\tau<u}$, $x,y\in X_{\tau\geq u}^{T_{ps}}$, $z,s\in X_{\tau\geq u}\setminus X_{\tau\geq u}^T$ then the following holds true.}
\label{Prods1stCIKKkiegeszitveIDEMPOTENS}
\begin{tabular}{cccccccc}
$\te$ & \vline & $\bot_{[w]}$  & $w$ & $\top_{[w]}$  & $s$ & $y_\downarrow$ & $y$\\
\hline
$\bot_{[v]}$  & \vline & $\bot_{[\g{v}{w}]}$ & {\color{midgrey}$\bot_{[\g{v}{w}]}$} & {\color{midgrey}$\bot_{[\g{v}{w}]}$}  & {\color{midgrey}$\g{v}{s}$} & $(\g{v}{y})_\downarrow$ &  {\color{midgrey}$(\g{v}{y})_\downarrow$}\\
$v$  & \vline & {\color{midgrey}$\bot_{[\g{v}{w}]}$} & {\color{midgrey}$\g{v}{w}$} & {\color{midgrey}$\top_{[\g{v}{w}]}$} & {\color{midgrey}$\in X_{\tau\geq u}\setminus X_{\tau\geq u}^T$}  & $(\g{v}{y})_\downarrow$ & ${\color{midgrey}\g{v}{y}\in X_{\tau\geq u}^{T_{ps}}}$ \\
$\top_{[v]}$ & \vline & {\color{midgrey}$\bot_{[\g{v}{w}]}$} & {\color{midgrey}$\top_{[\g{v}{w}]}$} & {\color{midgrey}$\top_{[\g{v}{w}]}$} & {\color{midgrey}$\in X_{\tau\geq u}\setminus X_{\tau\geq u}^T$} & $(\g{v}{y})_\downarrow$ & {\color{midgrey}$\g{v}{y}$}\\
$z$ & \vline & {\color{midgrey}$\g{z}{w}$} & {\color{midgrey}$\g{z}{w}$} & {\color{midgrey}$\g{z}{w}$} & {\color{midgrey}$\in X_{\tau\geq u}\setminus X_{\tau\geq u}^T$} & {\color{midgrey}$\g{z}{y}$} & {\color{midgrey}$\g{z}{y}$}\\
$x_\downarrow$ & \vline & \color{midgrey}$(\g{x}{w})_\downarrow$ & \color{midgrey}$(\g{x}{w})_\downarrow$ & \color{midgrey}$(\g{x}{w})_\downarrow$ & {\color{midgrey}$\g{x}{s}$} 
& $(\g{x}{y})_\downarrow$ $|$ $\bot_{[r]}$ & $(\g{x}{y})_\downarrow$ $|$ $\bot_{[r]}$\\
$x$ & \vline & {\color{midgrey}$(\g{x}{w})_\downarrow$} & {\color{midgrey}$\g{x}{w}\in X_{\tau\geq u}^{T_{ps}}$} & {\color{midgrey}$\g{x}{w}$} & {\color{midgrey}$\in X_{\tau\geq u}\setminus X_{\tau\geq u}^T$} & $(\g{x}{y})_\downarrow$ $|$ $\bot_{[r]}$ & $\g{x}{y}\in X_{\tau\geq u}^{T_{ps}}$ $|$ $\top_{[r]}$\\
\end{tabular}
\end{table}
%\end{enumerate}
\end{proposition}
\begin{proof}
%\item[(\ref{dkajskfhl})]
Grey items in Table~\ref{Prods1stCIKKkiegeszitveIDEMPOTENS} are either 
inherited from Table~\ref{ProdsGeneralCase}, or
straightforward, or 
are readily seen by the commutativity of  $\te$. 
Hence it suffices to prove the items in black.
%Let $v,w\in X_{\tau<u}$, $x,y\in X_{\tau\geq u}^{T_{ps}}$, $z\in X_{\tau\geq u}\setminus X_{\tau\geq u}^T$.
\begin{itemize}
\item
By claim~(\ref{HolTopAzBottom}) in Proposition~\ref{topbotEXISTszorzotabla}, $\bot_{[v]}\notin X_{\tau\geq u}^T$ holds, and thus
$\g{\bot_{[v]}}{\bot_{[w]}}=\g{\bot_{[v]}}{w}=\bot_{[\g{v}{w}]}$
follows from (\ref{dgskjdfhsasfGAPOS}) and Table~\ref{Prods1stCIKK}$_{(1,2)}$, respectively, thus confirming Table~\ref{Prods1stCIKKkiegeszitveIDEMPOTENS}$_{(1,1)}$.
\item
Next, $\g{\bot_{[v]}}{y_\downarrow}=\g{\bot_{[v]}}{(\g{y}{\nega{u}})}=\g{(\g{\bot_{[v]}}{\bot_{[t]}})}{y}$, which is equal to $\g{\bot_{[v]}}{y}$ by Table~\ref{Prods1stCIKKkiegeszitveIDEMPOTENS}$_{(1,1)}$, and hence the second row of (\ref{dgskjdfhsasfGAPOS}) concludes the proof of Table~\ref{Prods1stCIKKkiegeszitveIDEMPOTENS}$_{(1,5)}$.
\item
Next, $\g{\top_{[v]}}{y_\downarrow}=\g{\top_{[v]}}{(\g{y}{\nega{u}})}=\g{(\g{\top_{[v]}}{y})}{\nega{u}}=\g{(\g{\top_{[v]}}{y})}{(\g{\nega{u}}{\nega{u}})}=\g{(\g{\top_{[v]}}{\nega{u}})}{(\g{\nega{u}}{y})}$. 
By Table~\ref{Prods1stCIKK}$_{(3,1)}$ it is equal to $\g{\bot_{[v]}}{y_\downarrow}$ and by Table~\ref{Prods1stCIKKkiegeszitveIDEMPOTENS}$_{(1,5)}$ the latest is equal to $(\g{v}{y})_\downarrow$, as stated in Table~\ref{Prods1stCIKKkiegeszitveIDEMPOTENS}$_{(3,5)}$.
\item
By monotonicity of $\te$, Table~\ref{Prods1stCIKKkiegeszitveIDEMPOTENS}$_{(1,5)}$ and Table~\ref{Prods1stCIKKkiegeszitveIDEMPOTENS}$_{(3,5)}$ ensure Table~\ref{Prods1stCIKKkiegeszitveIDEMPOTENS}$_{(2,5)}$.
%\item
%Next, we prove Table~\ref{Prods1stCIKKkiegeszitveIDEMPOTENS}$_{(2,6)}$, that is, $(\g{v}{y})_\downarrow<\g{v}{y}$ and $\g{(\g{v}{y})}{\nega{u}}=(\g{v}{y})_\downarrow$. Indeed, $(\g{v}{y})_\downarrow=\g{v}{y}$ would yield $\g{v}{y_\downarrow}=\g{v}{y}$ by Table~\ref{Prods1stCIKKkiegeszitveIDEMPOTENS}$_{(2,5)}$, and in turn, $y_\downarrow=y$ by claim~(\ref{nagyonSzigoru}) in Proposition~\ref{elozetes}, a contradiction. On the other hand, $\g{(\g{v}{y})}{\nega{u}}=\g{v}{(\g{y}{\nega{u}})}=\g{v}{y_\downarrow}$, which is equal to $(\g{v}{y})_\downarrow$ by Table~\ref{Prods1stCIKKkiegeszitveIDEMPOTENS}$_{(2,5)}$.
\item
Finally, we prove the remaining $2\times2$ entries at the bottom-right of Table~\ref{Prods1stCIKKkiegeszitveIDEMPOTENS}. 

As for Table~\ref{Prods1stCIKKkiegeszitveIDEMPOTENS}$_{(6,6)}$, 
by claim~(\ref{csekken}) in Proposition~\ref{khkfgjkdfjkdjkddjss} 
%it holds that $\g{(\g{x}{y})}{\nega{u}}\neq\g{x}{y}$. Therefore, from 
and by 
(\ref{dgskjdfhsasfGAPOS}) it follows that $\g{x}{y}$ is either in $X_{\tau\geq u}^{T_c}$ or in $X_{\tau\geq u}^{T_{ps}}$.
Clearly, $\g{x}{y}\in X_{\tau\geq u}^{T_{ps}}$ if and only if $\g{x}{y}>(\g{x}{y})_\downarrow\in X_{\tau\geq u}$, and 
$\g{x}{y}\in X_{\tau\geq u}^{T_c}$ if and only if $\g{x}{y}>(\g{x}{y})_\downarrow\in X_{\tau<u}$
or
$\g{x}{y}=(\g{x}{y})_\downarrow$ holds.
As for 
Table~\ref{Prods1stCIKKkiegeszitveIDEMPOTENS}$_{(5,5)}$,
Table~\ref{Prods1stCIKKkiegeszitveIDEMPOTENS}$_{(5,6)}$, 
and
Table~\ref{Prods1stCIKKkiegeszitveIDEMPOTENS}$_{(6,5)}$,
$\g{x_\downarrow}{y_\downarrow}=
\g{(\g{x}{\nega{u}})}{(\g{y}{\nega{u}})}=
\g{(\g{x}{y})}{(\g{\nega{u}}{\nega{u}})}=
\g{(\g{x}{y})}{\nega{u}}=
\g{(\g{x}{\nega{u}})}{y}=
\g{x}{(\g{y}{\nega{u}})}=
\g{x_\downarrow}{y}=
\g{x}{y_\downarrow}$,
which is equal to $(\g{x}{y})_\downarrow$ if $\g{x}{y}\in X_{\tau\geq u}^{T_{ps}}$, 
and is equal to $\bot_{[r]}$ if $\g{x}{y}=\top_{[r]}\in X_{\tau\geq u}^{T_c}$.

\end{itemize}
\end{proof}

\begin{lemma}\label{GammaHomo}%{{\bf (Collapsing the convex components)}}
Let ${\mathbf X}=( X, \leq, \te, \ite{\te}, t, f )$ be an odd involutive FL$_e$-chain with residual complement $\komp$, such that there exists $u$, the smallest strictly positive idempotent and $\nega{u}$ is idempotent. 
\begin{enumerate}
\item\label{GAMMA<_Xx6}
$\mathbf Y=(\gamma(\beta(X)),\leq_\gamma,\tegamma, \ite{\gamma}, \gamma(\beta(t)),\gamma(\beta(f)))$ is an odd involutive FL$_e$-chain with involution $\kompM{\gamma}$.
The set of positive idempotent elements of $\mathbf Y$ is order isomorphic to the set of positive idempotent elements of $\mathbf X$ deprived of $t$.
\item\label{GAMMA<_X1}
%cancellative subalgebra
Over $\gamma(\beta(X_{\tau\geq u}^E))$ there is a subgroup 
$\mathbf Z$ of $\mathbf Y$ 
and over $\gamma(\beta(X_{\tau\geq u}^{E_c}))$ there is a subgroup 
$\mathbf V$ of $\mathbf Z$ such that
%\item\label{GAMMAlexikoS} $\mathbf X \overset{\nu_1}{\hookrightarrow} \PLPIII{\mathbf Y}{\mathbf Z}{\mathbf V}{\overline{\mathbf{ker}_\beta}}$.
\item\label{GAMMAlexikoS}
$\mathbf X \cong\PLPIs{\mathbf Y}{\mathbf Z}{\overline{\mathbf{ker}_\beta}}{\mathbf G}$ for some 
$\mathbf G\leq\mathbf V\lex{\overline{\mathbf{ker}_\beta}}$.
\end{enumerate}
\end{lemma}

\begin{proof}
\begin{enumerate}
\item
Because of claim~(\ref{<_Xx6}) in Lemma~\ref{faktorok}, it suffices to verify that $\gamma$ is a homomorphism from $\beta(X)$ to $\gamma(\beta(X))$.
The definition of $\leq_\gamma$ is independent of the choice of $x$ and $y$, since the sets
$\{{[v]},\{\top_{[v]}\}, \{\bot_{[v]}\}\}$ (for $v\in X_{\tau<u}$) and $\{\{p\},\{p_\downarrow\}\}$ (for $p\in X_{\tau\geq u}^{T_{ps}}$) are convex in $\beta(X)$. The ordering $\leq_{\beta}$ is linear since so is $\leq$. Therefore, (\ref{gammas_rendezese}) is well-defined.
Preservation of the monoidal operation can readily be checked in Table~\ref{Prods1stCIKKkiegeszitveIDEMPOTENS}, and
preservation of the residual complement can readily be seen from (\ref{betas_nega}), (\ref{phvsfkfb}) and claim~(\ref{Ranegalok}) in Proposition~\ref{khkfgjkdfjkdjkddjss}.
In complete analogy to the way we proved the preservation of the implication under $\beta$ in (\ref{IgyKellEzt}), we can verify the preservation of the implication under $\gamma$. 

If $c\geq u$ is an idempotent element of $\mathbf X$ then $\gamma(\beta(c))$ is also idempotent and $\gamma(\beta(c))\geq\gamma(\beta(u))=\gamma(\beta(t))$ (thus also positive) since $\gamma\circ\beta$ is a homomorphism.
Assume $c,d\geq u$ are idempotent elements of $\mathbf X$ and $c<d$.
Then $c,d\in X_{\tau\geq u}$ holds because of claim~(\ref{ZetaOfIdempotent}) in Proposition~\ref{zeta_1}.
Therefore, 
referring to the definition of $\beta$ and $\gamma$, and claim~(\ref{EXTRnincsKOMPONENSBEN}) in Proposition~\ref{topbotEXISTszorzotabla}, $\gamma(\beta(c))=\gamma(\beta(d))$ could hold only if 
either $c=\bot_{[v]}$ and $d=\top_{[v]}$ for some $v\in X_{\tau<u}$,
or $c=d_\downarrow$ for some $d\in X_{\tau\geq u}^{T_{ps}}$.
The latter case readily leads to a contradiction, since by claim~(\ref{G1tau=u}) in Proposition~\ref{khkfgjkdfjkdjkddjss} $\tau(d)=u$, whereas claim~(\ref{ZetaOfIdempotent}) in Proposition~\ref{zeta_1} shows $\tau(d)=d>c\geq u$.
In the former case, by
%Table~\ref{Prods1stCIKKkiegeszitveIDEMPOTENS}$_{(1,1)}$ or by 
Table~\ref{Prods1stCIKKkiegeszitveIDEMPOTENS}$_{(3,3)}$, 
the idempotency of $d$ implies 
$\top_{[\g{v}{v}]}=\top_{[v]}$,
that is,
$[\g{v}{v}]=[v]$,
that is,
$\beta(\g{v}{v})=\beta(v)$
that is,
$\gbeta{\beta(v)}{\beta(v)}=\beta(v)$.
Since $\mathbf X_{\tau<u}=\mathbf X_\mathbf{gr}$ holds by claim~(\ref{lexikoS}) in Lemma~\ref{faktorok}, it holds that 
$\nega{v}$ is the inverse of $v$, and hence
$\beta(\nega{v})$ is the inverse of $\beta(v)$ in $\beta(\mathbf X)$.
Therefore, 
$
\beta(v)=
\gbeta{(\gbeta{\beta(\nega{v})}{\beta(v)})}{\beta(v)}=
\gbeta{\beta(\nega{v})}{(\gbeta{\beta(v)}{\beta(v)})}=
\gbeta{\beta(\nega{v})}{\beta(v)}=
\beta(\g{\nega{v}}{v})=
\beta(t)
$
follows, that is, $[v]=[t]$.
But it contradicts to $c=\bot_{[t]}\geq u$.
These conclude the proof of the second statement.
\item To prove the first statement, by Theorem~\ref{csoportLesz}, it suffices to prove that any element of $\gamma(\beta(X_{\tau\geq u}^E))$ has an inverse in $\gamma(\beta(X_{\tau\geq u}^E))$.
For $v\in X_{\tau<u}$, 
$\ggamma{\{{[v]},\{\top_{[v]}\}, \{\bot_{[v]}\}\}}{\{{[\nega{v}]},\{\top_{[\nega{v}]}\}, \{\bot_{[\nega{v}]}\}\}}=\gamma(\gbeta{[v]}{[\nega{v}]})=\gamma(\beta(\g{v}{\nega{v}}))$ $=\gamma(\beta(t))$,
and for $p\in X_{\tau\geq u}^{T_{ps}}$, 
$$
\ggamma{\{\{p\},\{p_\downarrow\}\}}{\{\{\nega{p}\},\{\nega{p}_\downarrow\}\}}=
\gamma(\gbeta{\{p\}}{\{\nega{p}\}\}})=
\gamma(\beta(\g{p}{\nega{p}}))=
$$
$$
\gamma(\beta(\nega{(\res{\te}{p}{p})}))=
\gamma(\beta(\nega{\tau(p)})).
$$
By claim~(\ref{G1tau=u}) in Proposition~\ref{khkfgjkdfjkdjkddjss}, it is equal to 
$
\gamma(\beta(\nega{u}))=
\gamma(\beta(t))
$.
The second statement is straightforward since $X_{\tau\geq u}^{E_c} \subseteq X_{\tau\geq u}^E$, and $\gamma\circ\beta$ is a homomorphism.
\item
Referring to Lemma~6.2, let $\delta=(\beta|_{ X_{gr}},\zeta)$ denote the embedding 
$\mathbf X_\mathbf {gr} \hookrightarrow_\nu \beta(\mathbf X_{\mathbf{gr}})\lex\overline{\mathbf{ker}_\beta}$.
Define a mapping $\alpha \ : \ X \to \PLPIII{Y}{Z}{V}{\overline{ker_\beta}}$ by
$$%\begin{equation}\label{alfa}
\alpha(x)=\left\{
\begin{array}{ll}
(\gamma(\beta(x)),\zeta(x)) 	& \mbox{if $x\in X_{\tau<u}$}\\
(\gamma(\beta(x)),\top)		& \mbox{if $x\in X_{\tau\geq u}^T$}\\
(\gamma(\beta(x)),\bot)		& \mbox{if $x\in X_{\tau\geq u}\setminus X_{\tau\geq u}^T$}\\
\end{array}
\right. ,
$$%\end{equation}
or in a more detailed form,
\begin{equation}\label{alfaDETAIL}
\alpha(x)=\left\{
\begin{array}{ll}
(\gamma(\beta(x)),\bot)		& \mbox{if $x\in X_{\tau\geq u}^{B_c}$}\\
(\gamma(\beta(x)),\zeta(x)) 	& \mbox{if $x\in X_{\tau<u}$}\\
(\gamma(\beta(x)),\top)		& \mbox{if $x\in X_{\tau\geq u}^{T_c}$}\\
(\gamma(\beta(x)),\bot)		& \mbox{if $x\in X_{\tau\geq u}^{B_{ps}}$}\\
(\gamma(\beta(x)),\top)		& \mbox{if $x\in X_{\tau\geq u}^{T_{ps}}$}\\
(\gamma(\beta(x)),\bot)		& \mbox{if $x\in X_{\tau\geq u}\setminus X_{\tau\geq u}^E$}\\
\end{array}
\right. .
\end{equation}
%\begin{equation}\label{alfaDETAIL2}
%\alpha(x)=\left\{
%\begin{array}{ll}
%(\{{[x]},\{\top_{[x]}\}, \{\bot_{[x]}\}\},\bot)		& \mbox{if $x\in X_{\tau\geq u}^{B_c}$}\\
%(\{{[x]},\{\top_{[x]}\}, \{\bot_{[x]}\}\},\zeta(x)) 	& \mbox{if $x\in X_{\tau<u}$}\\
%(\{{[x]},\{\top_{[x]}\}, \{\bot_{[x]}\}\},\top)		& \mbox{if $x\in X_{\tau\geq u}^{T_c}$}\\
%(\{\{x\},\{x_\downarrow\}\},\bot)		& \mbox{if $x\in X_{\tau\geq u}^{B_{ps}}$}\\
%(\{\{x\},\{x_\downarrow\}\},\top)		& \mbox{if $x\in X_{\tau\geq u}^{T_{ps}}$}\\
%(\{\{x\}\},\bot)		& \mbox{if $x\in X_{\tau\geq u}\setminus X_{\tau\geq u}^E$}\\
%\end{array}
%\right. .
%\end{equation}
Then $\alpha$ is an embedding of $\mathbf X$ to $\PLPIII{\mathbf Y}{\mathbf Z}{\mathbf V}{\overline{\mathbf{ker}_\beta}}$.
Indeed, 
intuitively, what $\gamma\circ\beta$ does is that it makes each component together with its extremals in $X$ collapse into a singleton, 
and also makes each coherent pair of pseudo extremals, that is, a gap $\{x_\downarrow,x\}$ ($x_\downarrow<x\in X_{\tau\geq u}^{T_{ps}}$) collapse into a singleton.
On the other hand, the construction $\PLPIII{\mathbf Y}{\mathbf Z}{\mathbf V}{\overline{\mathbf{ker}_\beta}}$ does almost the opposite, namely, it replaces any image of such a component by the {\em divisible hull of the} component equipped with a top and a bottom (extremals), and it replaces each image of a pair of coherent pseudo extremals by a top and a bottom (pseudo extremals), just like $\alpha$ does in (9.6); see the related definitions, in claim~(2) and Definitions~6.1, 9.1, and 4.2/A.
Thus, $\alpha$ is an order-embedding from $X$ to $\PLPIII{Y}{Z}{V}{\overline{{ker}_\beta}}$, the universe of $\PLPIII{\mathbf Y}{\mathbf Z}{\mathbf V}{\overline{\mathbf{ker}_\beta}}$.
A straightforward verification using 
Definition~4.2/A and Table~3 shows the preservation of the monoidal operation under $\alpha$.
Moreover, (6.4), (7.3) and claim~(4) in Proposition~8.2 shows the preservation of the residual complement, under $\alpha$.
In complete analogy to the way we proved the preservation of the implication under $\beta$ in (6.5), we can verify the preservation of the implication under $\alpha$, too. 
Finally, it is clear that $\alpha$ maps the unit element of $\mathbf X$ to the unit element of $\PLPIII{\mathbf Y}{\mathbf Z}{\mathbf V}{\overline{\mathbf{ker}_\beta}}$.

If any element in $\PLPIII{\mathbf Y}{\mathbf Z}{\mathbf V}{\overline{\mathbf{ker}_\beta}}$ 
is not invertible then it is the $\alpha$-image of a suitable chosen $x\in X$, as it is easy to see from the shorter definition of $\alpha$ (before (\ref{alfaDETAIL})).
Adding that the projection of $\delta(\mathbf X_\mathbf {gr})$ to the first coordinate  
inside 
$\beta(\mathbf X_{\mathbf{gr}})\lex\overline{\mathbf{ker}_\beta}$ is onto, 
by the definition in (\ref{alfaDETAIL}), 
it yields that $\alpha$ is an embedding of
$\mathbf X$ into 
$\PLPIII{\mathbf Y}{\mathbf Z}{\mathbf V}{\overline{\mathbf{ker}_\beta}}$ 
hence into
$\PLPI{\mathbf Y}{\mathbf Z}{\overline{\mathbf{ker}_\beta}}$.
%such that the projection of $\alpha(\mathbf X)$ to the first coordinate inside $\PLPI{\mathbf Y}{\mathbf Z}{\overline{\mathbf{ker}_\beta}}$ is also onto.
Summing up, 
$\mathbf X \cong\PLPIs{\mathbf Y}{\mathbf Z}{\overline{\mathbf{ker}_\beta}}{\mathbf G}$ for some 
$\mathbf G\leq\mathbf V\lex{\overline{\mathbf{ker}_\beta}}$.
\end{enumerate}
\end{proof}

\section{One-step decomposition -- when $\nega{u}$ is not idempotent}\label{OneStepNotIdemp}

%When $\nega{u}$ is not idempotent, we shall consider a residuated subsemigroup $\mathbf X_{\tau\geq u}$ of $\mathbf X$, along and the subgroups $\mathbf X_{\tau\geq u}^T$ and $\mathbf X_{\tau\geq u}^{T_c}$ of $\mathbf X_{\tau\geq u}$

In order to handle the case when $\nega{u}$ is not idempotent, we shall consider a residuated subsemigroup $\mathbf X_{\tau\geq u}$ of $\mathbf X$, which is an odd involutive FL$_e$-algebra, albeit not a subalgebra of $\mathbf X$, since its unit element $u$ and residual complement $\kompM{u}$ will differ from 
those 
%the unit element $t$ and residual complement $\komp$ 
of $\mathbf X$.
Then we prove that the group-part of $\mathbf X_{\tau\geq u}$ (namely ${\mathbf X_{\tau\geq u}^T}$) is discretely embedded into $\mathbf X_{\tau\geq u}$, that ${\mathbf X_{\tau\geq u}^{T_c}}\leq{\mathbf X_{\tau\geq u}^T}$, and finally we will recover $\mathbf X$, up to isomorphism, as a type II partial sublex product of $\mathbf X_{\tau\geq u}$, ${\mathbf X_{\tau\geq u}^T}$, $\mathbf{ker}_\beta$ and a suitable chosen 
$
\mathbf G\leq
\mathbf X_{\tau\geq u}^{T_c}\lex{\overline{\mathbf{ker}_\beta}}
$.

\begin{proposition}\label{Kalap2}
%{\bf\color{blue} (OK)} {\bf (Product Table -- non idempotent $\nega{u}$ case)}
Let ${\mathbf X}=( X, \leq, \te, \ite{\te}, t, f )$ be an odd involutive FL$_e$-chain with residual complement $\komp$, such that there exists $u$, the smallest strictly positive idempotent. % and $u\neq t_\uparrow$. 
Assume $\nega{u}$ is not idempotent.
%, that is, $\g{\nega{u}}{\nega{u}}=\nega{u}_\downarrow<\nega{u}$. 
\begin{enumerate}
\item\label{slsdkfjkkjjjj}
The following product table holds for $\te$ (in the bottom-right entry 
the formula on the left applies if $\g{x}{y}>(\g{x}{y})_\downarrow\in X_{\tau\geq u}$, whereas the formula on the right applies if
$\g{x}{y}>(\g{x}{y})_\downarrow\in X_{\tau<u}$
or
$\g{x}{y}=(\g{x}{y})_\downarrow$.
{\footnotesize
\begin{table}[ht]
\centering
\caption{{(When $\nega{u}$ is not idempotent):} If $v,w\in X_{\tau<u}$, and $x,y\in X_{\tau\geq u}^{T_{ps}}$ then it holds true that}
\begin{tabular}{cccc}
$\te$ & \vline & $\top_{[w]}$ & $y$ \\
\hline
$\top_{[v]}$ & \vline & {\color{midgrey}$\top_{[\g{v}{w}]}$}  & {\color{midgrey}$\g{v}{y}\in X_{\tau\geq u}^{T_{ps}}$}\\
$x$ & \vline &   ${\color{midgrey}\g{x}{w}}\in X_{\tau\geq u}^{T_{ps}}$ &$\in X_{\tau\geq u}^{T_{ps}}$ $|$ $\in X_{\tau\geq u}^{T_c}$\\
\label{Prods1stCIKKu'u'NEMu'S}
\end{tabular}
\end{table}
}

\item\label{hldjskfjgjhdd2}
$X_{\tau\geq u}^T$ is closed under $\komp$.

%If $(\bot_{[w]})_\downarrow=\bot_{[w]}$ then there exists $v\in X_{\tau<u}$ such that $\top_{[v]}=\bot_{[w]}$. If $(\bot_{[w]})_\downarrow<\bot_{[w]}$ then either $\bot_{[w]}\in X_{\tau\geq u}^{T_{ps}}$ or there exists $v\in X_{\tau<u}$ such that $\top_{[v]}=\bot_{[w]}$.
\item\label{bottomosLIKE2}
%$\beta(\mathbf X_{\mathbf{gr}})$ is discretely embedded into $\beta(\mathbf X)$.
$X_{\tau\geq u}^T$ is discretely embedded into $X_{\tau\geq u}$.

%\item\label{fshfihsdjhfskfh2} If $(\g{x}{y})_\downarrow<\g{x}{y}$ then $\g{x_\downarrow}{y_\downarrow}=((\g{x}{y})_\downarrow)_\downarrow$ or $\g{x_\downarrow}{y_\downarrow}=\bot_{[v]}$ for some $v\in X_{\tau<u}$.
%\item\label{hksksfghjj2} if $\g{x}{y}=(\g{x}{y})_\downarrow$ then $\g{x_\downarrow}{y_\downarrow}=\bot_{[v]}$ for the same $v\in X_{\tau<u}$ as in claim~(\ref{szorzatK_TOP}).
%\item\label{hldjskfjgjhdd} Either $\bot_{[t]}\in X_{\tau\geq u}^{T_{ps}}$ or there exists $v\in X_{\tau<u}$ such that $\top_{[v]}=\bot_{[t]}$.
\end{enumerate}
\end{proposition}
\begin{proof}
\begin{enumerate}
\item[(\ref{slsdkfjkkjjjj})]
Grey items in Table~\ref{Prods1stCIKKu'u'NEMu'S} are either readily seen by symmetry of $\te$
or are inherited from Table~\ref{ProdsGeneralCase}, so it suffices to prove the items in black.
\begin{itemize}
\item
By Lemma~\ref{tau_lemma}, $\g{x}{y}\in X_{\tau\geq u}$ holds.
Since by claim~(\ref{csekken}) in Proposition~\ref{khkfgjkdfjkdjkddjss}, $\g{(\g{x}{y})}{\nega{u}}\neq\g{x}{y}$ holds, it follows from (\ref{dgskjdfhsasfGAPOS}) that either $\g{x}{y}\in X_{\tau\geq u}^{T_c}$ or $\g{x}{y}\in X_{\tau\geq u}^{T_{ps}}$.
Clearly, the former holds if $\g{x}{y}>(\g{x}{y})_\downarrow\in X_{\tau<u}$.
If $\g{x}{y}=(\g{x}{y})_\downarrow$ then $\g{x}{y}\in X_{\tau\geq u}^{T_{ps}}$ cannot hold, thus $\g{x}{y}\in X_{\tau\geq u}^{T_c}$ holds.
Finally, if $\g{x}{y}>(\g{x}{y})_\downarrow\in X_{\tau\geq u}$ then $\g{x}{y}\in X_{\tau\geq u}^{T_c}$ cannot hold, thus $\g{x}{y}\in X_{\tau\geq u}^{T_{ps}}$ holds.
This proves Table~\ref{Prods1stCIKKu'u'NEMu'S}$_{(2,2)}$.
\item
To prove Table~\ref{Prods1stCIKKu'u'NEMu'S}$_{(2,1)}$, that is, $\g{x}{\top_{[w]}}\in X_{\tau\geq u}^{T_{ps}}$ we proceed as follows.

First we show that $\g{x}{\top_{[w]}}$ cannot be in $X_{\tau\geq u}^{T_c}$.
Indeed, if $\g{x}{\top_{[w]}}=\top_{[v]}$ then
by Table~\ref{Prods1stCIKK}$_{(3,4)}$ and Table~\ref{Prods1stCIKK}$_{(3,1)}$, 
$\g{x_\downarrow}{w}=\g{(\g{x}{\nega{u}})}{w}=\g{(\g{x}{w})}{\nega{u}}=\g{(\g{x}{\top_{[w]}})}{\nega{u}}=\g{\top_{[v]}}{\nega{u}}=\bot_{[v]}$.
Multiplying both sides by $\nega{w}$ we obtain, using Table~\ref{Prods1stCIKK}$_{(2,1)}$, $x_\downarrow=\g{\nega{w}}{\bot_{[v]}}=\bot_{[\g{\nega{w}}{v}]}$, a contradiction to claim~(\ref{NoTopGap1lehetseges}) in Proposition~\ref{khkfgjkdfjkdjkddjss}.

Next, we show $(\g{x}{\top_{[w]}})_\downarrow<\g{x}{\top_{[w]}}$.
Assume the opposite, that is, $(\g{x}{\top_{[w]}})_\downarrow=\g{x}{\top_{[w]}}$. Since $\g{x}{\top_{[w]}}$ is not in $X_{\tau\geq u}^{T_c}$, by Table~\ref{Prods1stCIKK}$_{(4,1)}$, $\g{(\g{x}{\top_{[w]}})}{\nega{u}}=\g{x}{\top_{[w]}}$ follows.
But it leads to $\g{x_\downarrow}{\bot_{[w]}}=\g{(\g{x}{\nega{u}})}{(\g{\top_{[w]}}{\nega{u}})}
=\g{(\g{(\g{x}{\top_{[w]}})}{\nega{u}})}{\nega{u}}=\g{x}{\top_{[w]}}$, a contradiction to Proposition~\ref{diagonalSZIGORU}.

Note that we have just proved $\g{(\g{x}{\top_{[w]}})}{\nega{u}}\neq\g{x}{\top_{[w]}}$, too.
Hence, using that $\g{x}{\top_{[w]}}\in X_{\tau\geq u}$ holds by Lemma~\ref{tau_lemma}, by (\ref{dgskjdfhsasfGAPOS}),  $\g{x}{\top_{[w]}}\in X_{\tau\geq u}^{T_{ps}}$ must hold, since the other two lines of (\ref{dgskjdfhsasfGAPOS}) cannot apply.
\end{itemize}
%\item[\bf\ref{hldjskfjgjhdd}] If $\bot_{[t]}=\nega{u}\neq\top_{[v]}$ for any $v\in X_{\tau<u}$ then by (\ref{dgskjdfhsasfGAPOS}), $\g{\nega{u}}{\nega{u}}$ is equal to $(\nega{u})_\downarrow$, since $\nega{u}$ is not idempotent, hence $\bot_{[t]}\in X_{\tau\geq u}^{T_{ps}}$.

%\item[\bf\ref{fshfihsdjhfskfh2}] $\g{x_\downarrow}{y_\downarrow}=\g{(\g{x}{\nega{u}})}{(\g{y}{\nega{u}})}=\g{(\g{(\g{x}{y})}{\nega{u}})}{\nega{u}}$, which is by claim~(\ref{szorzatK_TOP}) is equal to $\g{(\g{x}{y})_\downarrow}{\nega{u}}$. By claim~(\ref{szorzatK_TOP}) and (\ref{bottomosLIKE2}), $(\g{x}{y})_\downarrow=\top_{[v]}$ for some $v\in X_{\tau<u}$ or $(\g{x}{y})_\downarrow\in X_{\tau\geq u}^{T_{ps}}$. In the first case, by Table~\ref{Prods1stCIKK}$_{(3,1)}$, $\g{(\g{x}{y})_\downarrow}{\nega{u}}=\bot_{[v]}$ holds, whereas in the latter case, $\g{(\g{x}{y})_\downarrow}{\nega{u}}=((\g{x}{y})_\downarrow)_\downarrow$ holds, as stated.

\item[(\ref{hldjskfjgjhdd2})]
{\em Claim:} For any odd involutive FL$_e$-chain $(X, \leq, \te, \ite{\te}, t, f )$, $x,y,a,b\in X$ and $a<b$,
if $\g{x}{[a,b]}=y$ \footnote{We mean $\g{x}{z}=y$ for all $z\in[a,b]$.} then $\g{\nega{y}}{]a,b]}=\nega{x}$.
%{\em Proof of claim:}
Indeed, let $u\in]a,b]$.
By (\ref{eq_quasi_inverse}) it suffices to prove $\res{\te}{u}{y}=x$.
Now, $\g{x}{u}=y$ holds, and if, by contradiction, for $x_1>x$, $\g{x_1}{u}=y$ would hold, than together with $\g{x}{a}=y$ it would contradict to Proposition~\ref{diagonalSZIGORU}. This settles the claim.

Let $x\in X_{\tau\geq u}^T$. Then $\nega{x}\in X_{\tau\geq u}$ by Lemma~\ref{tau_lemma}.
If $\nega{x}\notin X_{\tau\geq u}^T$ then by (\ref{dgskjdfhsasfGAPOS}), $\g{\nega{x}}{\nega{u}}=\nega{x}$ would hold.
It would imply $\nega{x}=\g{(\g{\nega{x}}{\nega{u}})}{\nega{u}}=\g{\nega{x}}{(\g{\nega{u}}{\nega{u}})}=\g{\nega{x}}{(\nega{u}_\downarrow)}$, that is,
$\g{\nega{x}}{[\nega{u}_\downarrow,\nega{u}]}=\nega{x}$,
which, by the claim would yield $\g{x}{\nega{u}}=x$, contradicting to $x\in X_{\tau\geq u}^T$ by (\ref{dgskjdfhsasfGAPOS}). 

\item[(\ref{bottomosLIKE2})]
%Since top-extremals are in one-to-one correspondence with convex components of the group part of $\mathbf X$, it suffices to 
We prove that 
%$X_{\tau\geq u}^T$ is discretely embedded into $X_{\tau\geq u}$, that is, 
$X_{\tau\geq u}^T$ is closed under $_\Downarrow$ and $_\Uparrow$, and neither $_\Downarrow$ nor $_\Uparrow$ have fixed point in $X_{\tau\geq u}^T$.
By (\ref{phvsfkfb}), for $w\in X_{\tau<u}$, $(\top_{[w]})_\Downarrow=\bot_{[w]}=\nega{(\top_{[\nega{w}]})}$ holds by (\ref{phvsfkfb}), and
by claim~(\ref{hldjskfjgjhdd2}) it is in $X_{\tau\geq u}^T$.
Therefore, it remains to prove 
$x_\downarrow\in X_{\tau\geq u}^T$ if $x\in X_{\tau\geq u}^{T_{ps}}$.
The opposite, that is, $x_\downarrow\in X_{\tau\geq u}\setminus X_{\tau\geq u}^T$
would imply 
$\g{x_\downarrow}{\nega{u}}=x_\downarrow$ by  (\ref{dgskjdfhsasfGAPOS}), and it 
would lead to
$\g{x}{\nega{u}}=x_\downarrow=\g{x_\downarrow}{\nega{u}}=\g{(\g{x_\downarrow}{\nega{u}})}{\nega{u}}=
\g{x_\downarrow}{(\g{\nega{u}}{\nega{u}})}$, a contradiction to Proposition~\ref{diagonalSZIGORU}.

It is straightforward that $_\Downarrow$ has no fixed point in $X_{\tau\geq u}^T$.
The statements on $_\Uparrow$ follow from (\ref{FelNeg_NegLe}).

%For $x\in X_{\tau\geq u}^T$, $x<x_\Uparrow=\nega{(\nega{x}_\Downarrow)}\in X_{\tau\geq u}^T$ follows from items~(\ref{hldjskfjgjhdd2}) and (\ref{bottomosLIKE2}), hence
%Therefore, $X_{\tau\geq u}^T$ is discretely embedded into $X_{\tau\geq u}$, as stated.
\end{enumerate}
\end{proof}

\begin{lemma}\label{decompXtaugeq}
Let ${\mathbf X}=( X, \leq, \te, \ite{\te}, t, f )$ be an odd involutive FL$_e$-chain with residual complement $\komp$, such that there exists $u$, the smallest strictly positive idempotent, and $\nega{u}$ is not idempotent.
\begin{enumerate}
\item\label{GAPOSqqJOzartraegeszmuvelet}
$\mathbf X_{\tau\geq u}=(X_{\tau\geq u}, \leq, \te, \ite{\te}, u, u )$ is an odd involutive FL$_e$-chain with residual complement 
\begin{equation}\label{circesKOMP}
\kompM{u} : \ x\mapsto
\left\{
\begin{array}{ll}
\nega{x} &\mbox{if $x\in X_{\tau\geq u}\setminus X_{\tau\geq u}^T$}\\
\nega{(x_\Downarrow)}=\nega{x}_\Uparrow & \mbox{if $x\in X_{\tau\geq u}^T$}
\end{array}
\right. .
\end{equation}
The set of positive idempotent elements of $\mathbf X_{\tau\geq u}$ equals the set of positive idempotent elements of $\mathbf X$ deprived of $t$.
\item\label{PriMe}
$\mathbf X_{\tau\geq u}^T=(X_{\tau\geq u}^T, \leq, \te, \ite{\te}, u, u )$ is the group part of $\mathbf X_{\tau\geq u}$, and \\
${\mathbf X_{\tau\geq u}^{T_c}=(X_{\tau\geq u}^{T_c}}, \leq, \te, \ite{\te}, u, u )$ is a subalgebra of ${\mathbf X_{\tau\geq u}^T}$.
%\item\label{AzIgazsagPillanata} $\mathbf X \overset{\nu_1}{\hookrightarrow} \PLPIV{\left(\mathbf X_{\tau\geq u}\right)}{\left(\mathbf X_{\tau\geq u}^{T_c}\right)}{\overline{\mathbf{ker}_\beta}}$.
\item\label{AzIgazsagPillanata}
$\mathbf X \cong\PLPIIs{\left(\mathbf X_{\tau\geq u}\right)}{\mathbf X_{\tau\geq u}^T}{\overline{\mathbf{ker}_\beta}}{\mathbf G}$
for some
$$
\mathbf G\leq
\mathbf X_{\tau\geq u}^{T_c}\lex{\overline{\mathbf{ker}_\beta}}
.
$$

\end{enumerate}
\end{lemma}
\begin{proof}
\begin{enumerate}
\item[(\ref{GAPOSqqJOzartraegeszmuvelet})]
The set $X_{\tau\geq u}$ is nonempty since by claim~(\ref{ZetaOfIdempotent}) in Proposition~\ref{zeta_1}, $u$ is an element in it.
$X_{\tau\geq u}$ is closed under $\te$ and $\komp$ by Lemma~\ref{tau_lemma}, and hence under $\ite{\te}$, too, by (\ref{eq_quasi_inverse}).
Clearly, $u$ is the unit element of $\te$ over $X_{\tau\geq u}$.
The residual complement operation in $\mathbf X_{\tau\geq u}$ is $\kompM{u}$, since 
by (\ref{eq_quasi_inverse}) and (\ref{dgskjdfhsasfGAPOS}), for $x\in X_{\tau\geq u}$ it holds true that
$\res{\te}{x}{u}=\nega{(\g{x}{\nega{u}})}=\nega{(\g{x}{\bot_{[t]}})}=\negaM{u}{x}$.
Next we prove that  $\mathbf X_{\tau\geq u}$ is involutive, that is, that $\kompM{u}$ is an order reversing involution of $X_{\tau\geq u}$. 
Indeed, if $x\in X_{\tau\geq u}\setminus X_{\tau\geq u}^T$ then 
$\negaM{u}{x}=\nega{x}\in X_{\tau\geq u}\setminus X_{\tau\geq u}^T$ by Lemma~\ref{tau_lemma} and claim~(\ref{hldjskfjgjhdd2}), and hence $\negaM{u}{(\negaM{u}{x})}=x$.
If $x\in X_{\tau\geq u}^T$ then $\negaM{u}{x}=\nega{x}_\Uparrow\in X_{\tau\geq u}^T$ holds by claims~(\ref{hldjskfjgjhdd2}) and (\ref{bottomosLIKE2}) in Proposition~\ref{Kalap2}, and
using claim~(\ref{bottomosLIKE2}) in Proposition~\ref{Kalap2}, it follows that $\negaM{u}{(\negaM{u}{x})}= \nega{(\nega{x}_\Uparrow)}_\Uparrow =x$, so we are done.
Next, $\negaM{u}{u}=\nega{(u_\Downarrow)}=\nega{(\nega{u})}=u$, so $\mathbf X_{\tau\geq u}$ is odd.
Finally, for any idempotent element $c$ such that $c\geq u$, also $c\in\mathbf X_{\tau\geq u}$ holds because of claim~(\ref{ZetaOfIdempotent}) in Proposition~\ref{zeta_1}.
Since $u$ is the smallest strictly positive idempotent element in $X$, the same claim shows that the only positive idempotent element of $\mathbf X$ which is missing from $\mathbf X_{\tau<u}$ is $t$.

\item[(\ref{PriMe})]
%First we prove that $x,y\in X_{\tau\geq u}$ and $\g{x}{y}=u$ implies $x,y\in X_{\tau\geq u}^T$, thus confirming $\mathbf X_{\tau\geq u}^T\subseteq(\mathbf X_{\tau\geq u})_\mathbf{gr}$.
{\em Claim:} $\kompM{u}$ is inverse operation on $X_{\tau\geq u}^T$.
Indeed, first
let $x\in X_{\tau\geq u}^{T_c}$.
By (\ref{phvsfkfb}) and Table~\ref{Prods1stCIKK}$_{(3,3)}$, 
$\g{x}{\negaM{u}{x}}=\g{\top_{[v]}}{\negaM{u}{(\top_{[v]})}}=\g{\top_{[v]}}{\nega{((\top_{[v]})_\Downarrow)}}=\g{\top_{[v]}}{\nega{(\bot_{[v]})}}=\g{\top_{[v]}}{\top_{[\nega{v}]}}=\top_{[\g{v}{\nega{v}}]}=\top_{[t]}=u$.
Second, let $x\in X_{\tau\geq u}^{T_{ps}}$.
Then
$\g{x}{\negaM{u}{x}}=\g{x}{\nega{(x_\Downarrow)}}=\g{x}{\nega{(x_\downarrow)}}=\g{x}{\nega{(\g{x}{\nega{u}})}}=\g{x}{(\res{\te}{x}{u})}\leq u$. 
On the other hand, since $x>x_\downarrow$, by residuation and by using that $X$ is a chain, $\g{x}{\negaM{u}{x}}=\g{x}{\nega{(x_\downarrow)}}>t$ holds, and it implies $\g{x}{\negaM{u}{x}}\geq u$, since $\g{x}{\negaM{u}{x}}\in X_{\tau\geq u}$ by Lemma~\ref{tau_lemma}, and $u$ is the smallest element of $X_{\tau\geq u}$, which is greater than $t$.
This settles the claim.

Therefore, $\mathbf X_{\tau\geq u}^T\subseteq(\mathbf X_{\tau\geq u})_\mathbf{gr}$ holds.
If, by contradiction, there would exists $p\in (\mathbf X_{\tau\geq u})_\mathbf{gr}\setminus \mathbf X_{\tau\geq u}^T$ then
its inverse is $\negaM{u}{p}$ by Theorem~\ref{csoportLesz}, and by the claim, $\negaM{u}{p}$ would also be in $(\mathbf X_{\tau\geq u})_\mathbf{gr}\setminus \mathbf X_{\tau\geq u}^T$.
Thus $\g{p}{\negaM{u}{p}}=u\in X_{\tau\geq u}^T$ follows, a contradiction to Table~\ref{ProdsGeneralCase}$_{(4,4)}$. %, which asserts that the product is in $X_{\tau\geq u}\setminus X_{\tau\geq u}^T$.

As for the second statement, Table~\ref{Prods1stCIKKu'u'NEMu'S}$_{(1,1)}$ shows that $X_{\tau\geq u}^{T_c}$ is closed under $\te$.
(\ref{phvsfkfb}) and the second line of (\ref{circesKOMP}) show that $X_{\tau\geq u}^{T_c}$ is closed under $\kompM{u}$.
Since $\mathbf X_{\tau\geq u}$ is an odd involutive FL$_e$-chain with residual complement $\kompM{u}$, it follows that 
$\res{\te}{x}{y}=\negaM{u}{(\g{x}{\negaM{u}{y}})}$, and hence $X_{\tau\geq u}^{T_c}$ is closed under $\ite{\te}$, too.
\item[(\ref{AzIgazsagPillanata})]
By the previous claim and by claim~(\ref{bottomosLIKE2}) in Proposition~\ref{Kalap2}, 
$$\PLPIV{\left(\mathbf X_{\tau\geq u}\right)}{\left(\mathbf X_{\tau\geq u}^{T_c}\right)}{\overline{\mathbf{ker}_\beta}}$$ is well-defined.
Referring to Lemma~\ref{faktorok}, let $\delta=(\beta|_{ X_{gr}},\zeta)$ denote the embedding $\mathbf X_\mathbf {gr} \hookrightarrow_\nu \beta(\mathbf X_{\mathbf{gr}})\lex\overline{\mathbf{ker}_\beta}$.
%$\delta(x)=(\delta_1(x),\delta_2(x))$
Define a mapping $\alpha \ : \ X \to \PLPIV{Z}{V}{\overline{ker_\beta}}$ by
$$
\alpha(x)=\left\{
\begin{array}{ll}
(\top_{\beta(x)},\zeta(x)) 	& \mbox{if $x\in X_{\tau<u}$}\\
(x,u)				& \mbox{if $x\in X_{\tau\geq u}$}\\
\end{array}
\right. .
$$
$\alpha$ is clearly injective.

\end{enumerate}

\begin{itemize}
\item
Denote the monoidal operation of
$\PLPIV{\left(\mathbf X_{\tau\geq u}\right)}{\left(\mathbf X_{\tau\geq u}^{T_c}\right)}{\overline{\mathbf{ker}_\beta}}$
by $\ted=(\te,\te)$.
\\
- If $x,y\in X_{\tau<u}$ then $\g{x}{y}\in X_{\tau<u}$ holds by Lemma~\ref{tau_lemma}, and hence by Table~\ref{Prods1stCIKKu'u'NEMu'S}$_{(1,1)}$ and by using that $\zeta$ is an embedding
$
\gd{\alpha(x)}{\alpha(y)}=
\gd{(\top_{\beta(x)},\zeta(x))}{(\top_{\beta(y)},\zeta(y))}=
(\g{\top_{\beta(x)}}{\top_{\zeta(y)}},\g{\zeta(x)}{\zeta(y)})=
(\top_{\beta(\g{x}{y})},\zeta(\g{x}{y}))=
\alpha(\g{x}{y})
$.
\\
- If $x,y\in X_{\tau\geq u}$ then $\g{x}{y}\in X_{\tau\geq u}$ holds by Lemma~\ref{tau_lemma}, and using the idempotency of $u$,
$
\alpha(\g{x}{y})=
(\g{x}{y},u)=
\gd{(x,u)}{(y,u)}=
\gd{\alpha(x)}{\alpha(y)}$
follows.
\\
- If $x\in X_{\tau<u}$ and $y\in X_{\tau\geq u}$ then $\g{x}{y}\in X_{\tau\geq u}$ holds by Lemma~\ref{tau_lemma}, and hence, using Table~\ref{Prods1stCIKK}$_{(3,4)}$ and that $u$ annihilates all elements in $]\nega{u},u[$ (see Table~\ref{Prods1stCIKK}$_{(2,3)}$), 
$
\alpha(\g{x}{y})=
(\g{x}{y},u)=
(\g{\top_{\beta(x)}}{y},\g{\zeta(x)}{u})=
\gd{(\top_{\beta(x)},\zeta(x))}{(y,u)}=
\gd{\alpha(x)}{\alpha(y)}
$
holds.

Summing up, we have verified $\alpha(\g{x}{y})=\gd{\alpha(x)}{\alpha(y)}$ for all $x,y\in X$.

\item
Denote the residual complement of
$\PLPIV{\left(\mathbf X_{\tau\geq u}\right)}{\left(\mathbf X_{\tau\geq u}^{T_c}\right)}{\overline{\mathbf{ker}_\beta}}$
by $\kompM{\ted}$.
Referring to (\ref{FuraNegaEXT}) and (\ref{circesKOMP}), respectively,
$\kompM{\ted}$ can be written in the following form
\begin{eqnarray*}
\negaM{\ted}{(x,y)} &=&
\left\{
\begin{array}{ll}
(\negaM{u}{x},u) 			& \mbox{if $x\not\in (X_{\tau\geq u})_{gr}$ and $y=u$}\\
((\negaM{u}{x})_\Downarrow,u) 	& \mbox{if $x\in (X_{\tau\geq u})_{gr}$ and $y=u$}\\
(\negaM{u}{x},\nega{y}) 	& \mbox{if $x\in X_{\tau\geq u}^{T_c}$ and $y\in ]\nega{u},u[$}\\
\end{array}
\right. 
\\
&=&
\left\{
\begin{array}{ll}
(\nega{x},u) 			& \mbox{if $x\in X_{\tau\geq u}\setminus X_{\tau\geq u}^T$ and $y=u$}\\
(\nega{x},u) 	& \mbox{if $x\in X_{\tau\geq u}^T$ and $y=u$}\\
(\negaM{u}{x},\nega{y}) 	& \mbox{if $x\in X_{\tau\geq u}^{T_c}$ and $y\in ]\nega{u},u[$}\\
\end{array}
\right. ,
\end{eqnarray*}
that is, 
%By claim~(\ref{bottomosLIKE2}) in Proposition~\ref{Kalap2}, $X_{\tau\geq u}^T$ is discretely embedded, and hence $(\negaM{u}{x})_\Downarrow=(\nega{(x_\Downarrow)})_\Downarrow=\nega{x}$ holds. Using it (it the third line of (\ref{hogy_nega_pLex2})), $\kompM{\ted}$ can be written as follows:
%\begin{equation}\label{hogy_nega_pLex2}
$$
\negaM{\ted}{(x,y)}=\left\{
\begin{array}{ll}
(\nega{x},u) 			& \mbox{if $y=u$}\\
(\negaM{u}{x},\nega{y}) 	& \mbox{if $y\in ]\nega{u},u[$}\\
\end{array}
\right. .
$$
%\end{equation}
\\
- If $x\in X_{\tau<u}$ then by Lemma~\ref{tau_lemma}, $\nega{x}\in X_{\tau<u}$, too, and by using (\ref{phvsfkfb}), (\ref{circesKOMP}) and that $\zeta$ is an embedding, it holds true that
$$
\alpha(\nega{x})=
(\top_{\beta(\nega{x})},\zeta(\nega{x}))
=(\nega{\bot_{\beta(x)}},\nega{\zeta(x)})
=(\nega{((\top_{\beta(x)})_\Downarrow)},\nega{\zeta(x)})=
$$
$$
(\negaM{u}{\top_{\beta(x)}},\nega{\zeta(x)})
=\negaM{\ted}{(\top_{\beta(x)},\zeta(x))}
=\negaM{\ted}{\alpha(x)}
.$$
\\
- If $x\in X_{\tau\geq u}$ then by Lemma~\ref{tau_lemma}, $\nega{x}\in X_{\tau\geq u}$, too, and hence
$\alpha(\nega{x})=
(\nega{x},u)=
\negaM{\ted}{(x,u)}=
\negaM{\ted}{\alpha(x)}$.

Summing up, we have verified
$
\alpha(\nega{x})=\negaM{\ted}{\alpha(x)}
$ for all $x\in X$.

\item
Since $\alpha$ preserves multiplication and residual complements, and since both $\mathbf{X}$ and 
$\PLPIV{\left(\mathbf X_{\tau\geq u}\right)}{\left(\mathbf X_{\tau\geq u}^{T_c}\right)}{\overline{\mathbf{ker}_\beta}}$
are involutive, it follows by (\ref{eq_quasi_inverse}) that $\alpha$ preserves residual operation, too, that is, $\alpha(\res{\te}{x}{y})=\res{\ted}{\alpha(x)}{\alpha(y)}$.

\item
Also, $\alpha(t)=(u,t)$ holds, so $\alpha$ preserves the unit element, too.
\end{itemize}
Therefore, $\alpha$ is an embedding of $\mathbf X$ into $\PLPIV{\left(\mathbf X_{\tau\geq u}\right)}{\left(\mathbf X_{\tau\geq u}^{T_c}\right)}{\overline{\mathbf{ker}_\beta}}$, hence into 
$\PLPIV{\left(\mathbf X_{\tau\geq u}\right)}{\left(\mathbf X_{\tau\geq u}^T\right)}{\overline{\mathbf{ker}_\beta}}$
.
% If an element of $\PLPIV{\left(\mathbf X_{\tau\geq u}\right)}{\left(\mathbf X_{\tau\geq u}^{T_c}\right)}{\overline{\mathbf{ker}_\beta}}$ is non-invertible then it is the $\alpha$-image of a suitable chosen $x\in X$, as it is easy to see from the definition of $\alpha$. Adding that the projection of $\delta(\mathbf X_\mathbf {gr})$ to the first coordinate inside $\beta(\mathbf X_{\mathbf{gr}})\lex\overline{\mathbf{ker}_\beta}$ is onto, by the definition in (\ref{alfaDETAIL}), it yields that $\alpha$ is an embedding of $\mathbf X$ into $\PLPIV{\left(\mathbf X_{\tau\geq u}\right)}{\left(\mathbf X_{\tau\geq u}^{T_c}\right)}{\overline{\mathbf{ker}_\beta}}$ such that the projection of $\alpha(\mathbf X)$ to the first coordinate inside $\PLPIV{\left(\mathbf X_{\tau\geq u}\right)}{\left(\mathbf X_{\tau\geq u}^{T_c}\right)}{\overline{\mathbf{ker}_\beta}}$ is also onto.
Summing up, 
$\mathbf X \cong\left(\PLPIV{\left(\mathbf X_{\tau\geq u}\right)}{\left(\mathbf X_{\tau\geq u}^T\right)}{\overline{\mathbf{ker}_\beta}}\right)_{\mathbf G}$ for some 
$
\mathbf G\leq
\mathbf X_{\tau\geq u}^{T_c}\lex{\overline{\mathbf{ker}_\beta}}
$.
\end{proof}

\section{Group representation}\label{mainsection}
The main theorem of the paper asserts that up to isomorphism, any odd involutive FL$_e$-chain which has only finitely many positive idempotent elements can be built by iterating finitely many times the type I and type II partial sublex product constructions %of Definition~\ref{FoKonstrukcio} 
using only linearly ordered abelian groups, as building blocks.
\begin{theorem}\label{Hahn_type}
If $\mathbf X$ is an odd involutive FL$_e$-chain, which has only $n\in\mathbb N$, $n\geq 1$ positive idempotent elements then it has a partial sublex product group representation, that is, for $i=2,\ldots,n$
there exist totally ordered abelian groups $\mathbf H_1$, 
$\mathbf H_i$, $\mathbf G_i$, $\mathbf Z_{i-1}$ along with $\iota_i\in\{I,II\}$ such that 
$
\mathbf X\cong\mathbf X_n,
$
where for $i\in\{2,\ldots,n\}$,
%$${\mathbf X_{i-1}}_{\mathbf\Gamma(\mathbf H_{i-1,1}, ^{\iota_i})}^{\mathbf\Gamma(\mathbf H_{i-1,2}, \mathbf G_i)}$$
\begin{equation}\label{EzABeszed}
\mathbf X_1=\mathbf H_1 \mbox{\ \ and \ } 
\mathbf X_i=
\left\{
\begin{array}{ll}
\PLPIs{\mathbf X_{i-1}}{\mathbf Z_{i-1}}{\mathbf G_i}{\mathbf H_i} & \mbox{ if $\iota_i=I$}\\
\PLPIIs{\mathbf X_{i-1}}{{\mathbf X_{i-1}}_\mathbf{gr}}{\mathbf G_i}{\mathbf H_i} 	& \mbox{ if $\iota_i=II$}\\
\end{array}
\right. .
\end{equation}

Notice that Theorem~\ref{Hahn_type} claims isomorphism between $\mathbf X$ and $\mathbf X_n$ hence $\mathbf X_n$ and consequently for $i=n-1,\ldots, 2$, the $\mathbf X_i$'s are claimed explicitly to exist (to be well defined).
By Definition~B, using that $(\mathbf X_i)_{\mathbf{gr}}=\mathbf H_i$ holds for $i\in\{1,\ldots,n\}$, %(by assumption if $i=1$, and by Remark~\ref{REMgrouppart} if $i\in\{2,\ldots,n\}$) 
it is necessarily that
\begin{equation}\label{implicite}
\begin{array}{ll}
\mbox{
for $i=2,\ldots,n$, 
$\mathbf Z_{i-1}\leq \mathbf H_{i-1}$,
$\mathbf H_i\leq\mathbf Z_{i-1}\lex\mathbf G_i$ and
}
\\
\mbox{
for $i=2,\ldots,n$, if $\iota_i=II$ then
$%(\mathbf Y_{i-1})_{\mathbf{gr}}=
\mathbf H_{i-1}%=(\mathbf Y_{i-1})_{\mathbf{gr}}
$ is discretely embedded into $\mathbf X_{i-1}$.
}
\end{array}
\end{equation}
%
%
%
%
%
%
%{{\bf (Type III \& IV group representation)}}
%{{\bf (Group representation)}}
\end{theorem}
\begin{proof}
Let ${\mathbf X}=( X, \leq, \te, \ite{\te}, t, f )$.
Induction by $n$, the number of idempotent elements in $X^+$.
If $n=1$ then the only idempotent in $X^+$ is $t$, hence Theorem~2.4 implies that $( X, \leq, \te, t)$ is a linearly ordered abelian group $\mathbf H_1$ and we are done.
Assume that the theorem holds up to $k-1$ (for some $2\leq k<n$), and let $\mathbf X$ be an odd involutive FL$_e$-chain which has $k$ positive idempotent elements.
Since the number of idempotents in $X^+$ is finite, there exists $u$, the smallest idempotent above $t$.

If $\nega{u}$ is idempotent then by Lemma~9.4 (by denoting $\alpha=\gamma\circ\beta$) 
$$
\mathbf X\cong \PLPIs{\alpha(\mathbf X)}{\alpha\left(\mathbf X_{\tau\geq u}^E\right)}{\overline{\mathbf{ker}_\beta}}{\mathbf H_i}
$$
holds for some 
$
\mathbf H_i\leq
%\left(\PLPIII{\alpha(\mathbf X)}{\alpha\left(\mathbf X_{\tau\geq u}^E\right)}{\alpha\left(\mathbf X_{\tau\geq u}^{E_c}\right)}{\overline{\mathbf{ker}_\beta}}\right)_{\mathbf{gr}}=
\alpha\left(\mathbf X_{\tau\geq u}^{E_c}\right)\lex\overline{\mathbf{ker}_\beta},
$
where 
$\alpha\left(\mathbf X_{\tau\geq u}^{E_c}\right)\leq\alpha\left(\mathbf X_{\tau\geq u}^E\right)$ are subgroups of the odd involutive FL$_e$-chain $\alpha(\mathbf X)$, and 
$\mathbf{ker}_\beta$ is a linearly ordered abelian group.
Therefore, if $\nega{u}$ is idempotent then set
$$
\mbox{\small
$\mathbf X_{k-1}=\alpha(\mathbf X)$,
$\mathbf Z_{k-1}=\alpha\left(\mathbf X_{\tau\geq u}^E\right)$,
%$\mathbf V_{k-1}=\alpha\left(\mathbf X_{\tau\geq u}^{E_c}\right)$,
$\mathbf G_k=\overline{\mathbf{ker}_\beta}$, and
$\iota_k=I$.
}
$$
If $\nega{u}$ is not idempotent then by Lemma~10.2
$$
\mathbf X \cong \PLPIIs{\left(\mathbf X_{\tau\geq u}\right)}{\left(\mathbf X_{\tau\geq u}^T\right)}{\overline{\mathbf{ker}_\beta}}{\mathbf H_i}
$$
holds for some 
$
\mathbf H_i\leq
%\left( \PLPIV{\left(\mathbf X_{\tau\geq u}\right)}{\left(\mathbf X_{\tau\geq u}^{T_c}\right)}{\overline{\mathbf{ker}_\beta}} \right)_{\mathbf{gr}}=
\mathbf X_{\tau\geq u}^{T_c}\lex{\overline{\mathbf{ker}_\beta}}
$,
where %the group part of $\mathbf X_{\tau\geq u}$ is discretely embedded into $\mathbf X_{\tau\geq u}$, 
$\mathbf X_{\tau\geq u}^{T_c}\leq\mathbf X_{\tau\geq u}^T$ are subgroups of the odd involutive FL$_e$-chain $\mathbf X_{\tau\geq u}$, and $\mathbf{ker}_\beta$ is a linearly ordered abelian group.
Therefore, if $\nega{u}$ is not idempotent then set
$$
\mbox{
$\mathbf X_{k-1}=\mathbf X_{\tau\geq u}$,
%$\mathbf V_{k-1}=\mathbf X_{\tau\geq u}^{T_c}$,
$\mathbf G_k=\overline{\mathbf{ker}_\beta}$, and
$\iota_k=II$.
}
$$
By Lemmas~9.4 and 10.2 the number of positive idempotent elements of $\mathbf X_{k-1}$ (be it equal to either $\alpha(\mathbf X)$ or $\mathbf X_{\tau\geq u}$) is one less than that of $\mathbf X$.
Therefore, by the induction hypothesis, the theorem holds for $\mathbf X_{k-1}$, that is, 
there exist linearly ordered abelian groups 
$\mathbf H_1$ and 
for $i=2,\ldots,k-1$, linearly ordered abelian groups
$\mathbf H_i$, $\mathbf G_i$, $\mathbf Z_{i-1}$ along with $\iota_i\in\{I,II\}$ such that $\mathbf X_1:=\mathbf H_1$ and for $i\in\{2,\ldots,k-1\}$, (11.1) holds.
\end{proof}

\begin{remark}
Note that by Lemmas~\ref{GammaHomo} and \ref{decompXtaugeq}, and claim~(\ref{jsjhsgsjdb}) in Theorem~\ref{csoportLesz}, linearly ordered abelian groups are exactly the indecomposable algebras with respect to the type I and type II partial sublex product constructions.
\end{remark}

\begin{remark}
Denote the one-element odd involutive FL$_e$-algebra by $\mathbf 1$.
If the algebra in Theorem~\ref{Hahn_type} is bounded then 
in its group representation $\mathbf G_1=\mathbf 1$, since all other linearly ordered groups are infinite and unbounded, and 
both type I and type II constructions preserve boundedness of the first component.
\end{remark}

\begin{remark} {\bf (A unified treatment)}
The reader might think that entirely different things are going in Sections~\ref{OneDecIdemp} and \ref{OneStepNotIdemp}, that is, depending on whether the residual complement of the smallest strictly positive idempotent element $u$ is idempotent or not.
It is not the case. In both cases a particular nuclear retract\footnote{For notions which are not defined here, the reader may consult e.g. \cite{gjko}.} plays a key role: 
Let ${\mathbf X}=( X, \leq, \te, \ite{\te}, t, f )$ be an odd involutive FL$_e$-chain with residual complement $\komp$, and $u$ be the (existing) smallest strictly positive idempotent element of $X$.
Let $\varphi : X\to X$ be defined by $$\varphi(x)=\nega{(\g{\nega{(\g{x}{\nega{u}})}}{\nega{u}})}.$$
The interested reader might wish to verify using what has already been proven in this paper that 
\begin{enumerate}
\item 
$\varphi$ is a nucleus of $\mathbf X$. %, denote $\mathbf X_\varphi$ the related nuclear retract.
\item 
The related nuclear retract $\mathbf X_\varphi$ over $X_\varphi=\{\varphi(x):x\in X\}$ is isomorphic to $\mathbf Y$ of Lemma~\ref{GammaHomo} if $\nega{u}$ is idempotent, and (not only isomorphic, but also) equal to $\mathbf X_{\tau\geq u}$ of Lemma~\ref{decompXtaugeq} if $\nega{u}$ is not idempotent.
\item
There is a subgroup\footnote{We mean cancellative subalgebra, see Theorem~\ref{csoportLesz}.} $\mathbf X_1$ of $\mathbf X_\varphi$ over $X_1=\{x\in X_\varphi:\g{x}{\nega{u}<x}\}$.
\item
There is a subgroup $\mathbf X_2$ of $\mathbf X_1$ over $X_2=\varphi(X\setminus(X_\varphi\cup\nega{(X_\varphi)}))$.
\item
There is a subalgebra $\mathbf u$ of $\mathbf X$ over $]\nega{u},u[$, and 
there is a subalgebra $\mathbf{ker}_\varphi$ of the residuated semigroup reduct of $\mathbf X$ over $\ker_\varphi=\{x\in X : \varphi(x)=u\}$.
$\ker_\varphi=[\nega{u},u]$ if $\nega{u}$ is idempotent, and 
$\ker_\varphi=]\nega{u},u]$ if $\nega{u}$ is not idempotent\footnote{$\mathbf{ker}_\varphi$ is isomorphic to $\mathbf u^{\top\bot}$  if $\nega{u}$ is idempotent, and isomorphic to $\mathbf u^\top$ if $\nega{u}$ is not idempotent, where $\mathbf u^{\top\bot}$ and $\mathbf u^\top$ mean $\mathbf u$ equipped with a new top and bottom element 
or 
$\mathbf u$ equipped with a new top element, respectively, just like in items A and B in Definition~\ref{FoKonstrukcio}.}.
\end{enumerate}

One can recover $\mathbf X$ as follows:
\begin{enumerate}
\item
Enlarge\footnote{Throughout this example enlargement is meant in the sense of Definition~\ref{LexRendDefi}.} the subset $X_2$ of $X_\varphi$ by $\overline{\ker_\varphi}$\footnote{Here we can even save the amendment of the second algebra by top or top and bottom elements, compared to the definition of the partial lexicographic products.}.
\item 
Enlarge the subset $X_1\setminus X_2$ by $\overline{\ker_\varphi}\setminus\,]\nega{u},u[$.
\item 
Enlarge the subset $X_\varphi\setminus X_1$ of $X_\varphi$ by $\{\nega{\varphi(\nega{u})}\}$. 
\item 
Replace the group-part of the obtained algebra by a subgroup of it. 
\end{enumerate}
Equip the obtained set by the lexicographic ordering, and define the mo\-no\-idal operation on it coordinatewise.
It has then $(u,t)$ as its unit element, it is residuated and thus it determines a unique residual operation. 
Then $\mathbf X$ embeds to the here-constructed algebra.
We believe that this unified treatment, although more elegant, does not make the discussion any shorter or any more transparent.

\end{remark}

%\begin{definition}
\bigskip
We say that an odd involutive FL$_e$-chain $\mathbf X$ can be represented as a {\em finite} (or more precisely, $n$-ary) partial sublex product of linearly ordered abelian groups, if $\mathbf X$ arises via finitely many iterations of the type I and type II constructions using linearly ordered abelian groups in the way it is described in Theorem~\ref{EzABeszed}. 
%\end{definition}
By using this terminology we may rephrase Theorem~\ref{EzABeszed}: 
Every odd involutive FL$_e$-chain, which has only finitely many positive idempotent elements can be represented as a finite partial sublex product of linearly ordered abelian groups.
By Hahn's theorem one can embed the linearly ordered abelian groups into some lexicographic products of real groups.
Therefore, a side result of Theorem~\ref{EzABeszed} is the generalization of Hahn's embedding theorem from the class of linearly ordered abelian groups to the class of odd involutive FL$_e$-chains which have only finitely many positive idempotent elements:

\begin{corollary}\label{coroHAHN} %{\bf (Hahn-type embedding)}
Odd involutive FL$_e$-chains which have exactly $n\in\mathbb N$ positive idempotent elements embed in some $n$-ary partial sublex product of lexicographic products of real groups.
\end{corollary}
Finally, we remark that ordinary lexicographic products can be used instead of partial sublex products if the less ambitious goal of embedding only the monoidal reduct is aimed at.

\begin{corollary}\label{coroHAHN2} %{\bf (Embedding the monoidal reduct)} 
The monoid reduct of any odd involutive FL$_e$-chain which has only finitely many idempotent elements embeds in the finite lexicographic product $\mathbf H_1\lex\mathbf G_2^{\top\bot}\lex\ldots\lex\mathbf G_n^{\top\bot}$, where $\mathbf H_1,  \mathbf G_2, \ldots,\mathbf G_n$ are the linearly ordered abelian groups of its group representation.
\end{corollary}
\begin{proof}
By Definition,
$\PLPIs{\mathbf X}{\mathbf Z}{\mathbf G}{\mathbf H}\leq\PLPI{\mathbf X}{\mathbf Z}{\mathbf G}$ and
$\PLPIIs{\mathbf X}{{\mathbf X}_\mathbf{gr}}{\mathbf G}{\mathbf H}
\leq\PLPII{\mathbf X}{\mathbf G}$
hold.
Observe that the monoidal reduct of 
$\PLPII{\mathbf X}{\mathbf G}$
embeds into
the monoidal reduct of 
$\mathbf X\lex{\mathbf G}^{\top\bot}$. 
Now let an odd involutive FL$_e$-chain, which has only finitely many idempotent elements, be given.
Take its group representation.
Guided by its consecutive iterative steps, in each step consider $(\mathbf X_{i-1})\lex{\mathbf G_i}^{\top\bot}$
instead of $\PLPIs{\mathbf X_{i-1}}{\mathbf Z_{i-1}}{\mathbf G_i}{\mathbf H_i}$ or $\PLPIIs{\mathbf X_{i-1}}{{\mathbf X_{i-1}}_\mathbf{gr}}{\mathbf G_i}{\mathbf H_i}$.
In the end, this results in the original algebra being embedded into the lexicographic product 
$\mathbf H_1\lex\mathbf G_2^{\top\bot}\lex\ldots\lex\mathbf G_n^{\top\bot}$,
where
$\mathbf H_1,  \mathbf G_2, \ldots,\mathbf G_n$ are linearly ordered abelian groups.
\end{proof}

\section*{Acknowledgement}
The present scientific contribution was supported by the GINOP 2.3.2-15-2016-00022 grant
and the Higher Education Institutional Excellence Programme 20765-3/2018/FEKUTSTRAT of the Ministry of Human Capacities in Hungary.
%\begin{acknowledgements}
%If you'd like to thank anyone, place your comments here
%and remove the percent signs.
%\end{acknowledgements}

% BibTeX users please use one of
%\bibliographystyle{spbasic}      % basic style, author-year citations
%\bibliographystyle{spmpsci}      % mathematics and physical sciences
%\bibliographystyle{spphys}       % APS-like style for physics
%\bibliography{}   % name your BibTeX data base

% Non-BibTeX users please use

\end{document}